\documentclass[a4paper,twoside,11pt]{article}

\usepackage[a4paper]{geometry}
\usepackage[latin1]{inputenc}
\usepackage[T1]{fontenc}

\usepackage{lipsum}
\usepackage{comment}

\usepackage{color}
\usepackage{amssymb,amsmath,amsthm,amscd,amsbsy}
\usepackage{mathrsfs}
\usepackage{stmaryrd}
\usepackage{esint}
\usepackage{subfig}
\usepackage{tabularx}
\usepackage{calc}
\usepackage{cancel}
\usepackage{yhmath}
\usepackage[pdftex]{graphicx}
\usepackage[all]{xy}
\xyoption{v2}
\xyoption{2cell}
\UseAllTwocells

\usepackage{enumitem}

\usepackage[pdftitle={Uniform Lipschitz Estimates in Bumpy Half-Spaces},pdftex]{hyperref}

\numberwithin{equation}{section}

\DeclareMathOperator{\Idd}{I}

\DeclareMathOperator{\supp}{Supp}

\DeclareMathOperator{\dist}{dist}

\let\oldmarginpar\marginpar
\renewcommand\marginpar[1]{\oldmarginpar{\color{red}\raggedleft\tiny #1}}

\newcommand{\intbar}{{- \hspace{- 1.05 em}} \int}

\begin{document}


\newtheorem{theo}{Theorem}
\newtheorem{prop}[theo]{Proposition}
\newtheorem{lem}[theo]{Lemma}
\newtheorem{cor}[theo]{Corollary}
\newtheorem*{theo*}{Theorem}
\newtheorem{rst}{Result}
\renewcommand*{\therst}{\Alph{rst}}

\theoremstyle{definition}
\newtheorem{defi}[theo]{Definition}

\theoremstyle{remark}
\newtheorem{rem}{Remark}
\newtheorem*{rem*}{Remark}
\newtheorem*{rems*}{Remarks}

\pagestyle{plain}

\title{Uniform Lipschitz Estimates in Bumpy Half-Spaces}

\author{Carlos Kenig\thanks{The University of Chicago, $5734$ S. University Avenue, Chicago, IL $60637$, USA. \emph{E-mail address:} \texttt{cek@math.uchicago.edu}} \and Christophe Prange\thanks{The University of Chicago, $5734$ S. University Avenue, Chicago, IL $60637$, USA. \emph{E-mail address:} \texttt{cp@math.uchicago.edu}}}

\maketitle

\begin{abstract}
This paper is devoted to the proof of uniform H\"older and Lipschitz estimates close to oscillating boundaries, for divergence form elliptic systems with periodically oscillating coefficients. Our main point is that no structure is assumed on the oscillations of the boundary. In particular, those are neither periodic, nor quasiperiodic, nor stationary ergodic. We investigate the consequences of our estimates on the large scales of Green and Poisson kernels. Our work opens the door to the use of potential theoretic methods in problems concerned with oscillating boundaries, which is an area of active research.
\end{abstract}

\section{Introduction}

This paper is concerned with H\"older and Lipschitz estimates for elliptic systems and their consequences in potential theory. Our divergence form elliptic system reads
\begin{equation}\label{sysoscueps}
\left\{
\begin{array}{rll}
-\nabla\cdot A(x/\varepsilon)\nabla u^\varepsilon&=f+\nabla\cdot F,&x\in D^\varepsilon(0,1),\\
u^\varepsilon&=0,&x\in \Delta^\varepsilon(0,1),
\end{array}
\right.
\end{equation}
and is posed in the domain with oscillating boundary 
\begin{equation*}
D^\varepsilon(0,1):=\left\{(x',x_d),\ |x'|<1,\ \varepsilon\psi(x'/\varepsilon)<x_d<\varepsilon\psi(x'/\varepsilon)+1\right\}.
\end{equation*}
The main focus of the paper is on uniformity in $\varepsilon$. The coefficients are assumed to be periodically oscillating. However, no structure assumption is made on the boundary. In particular, no periodicity, quasiperiodicity, nor stationary ergodicity is assumed on the oscillations of $\psi$.

\begin{center}
\includegraphics[width=9.78cm,height=5cm]{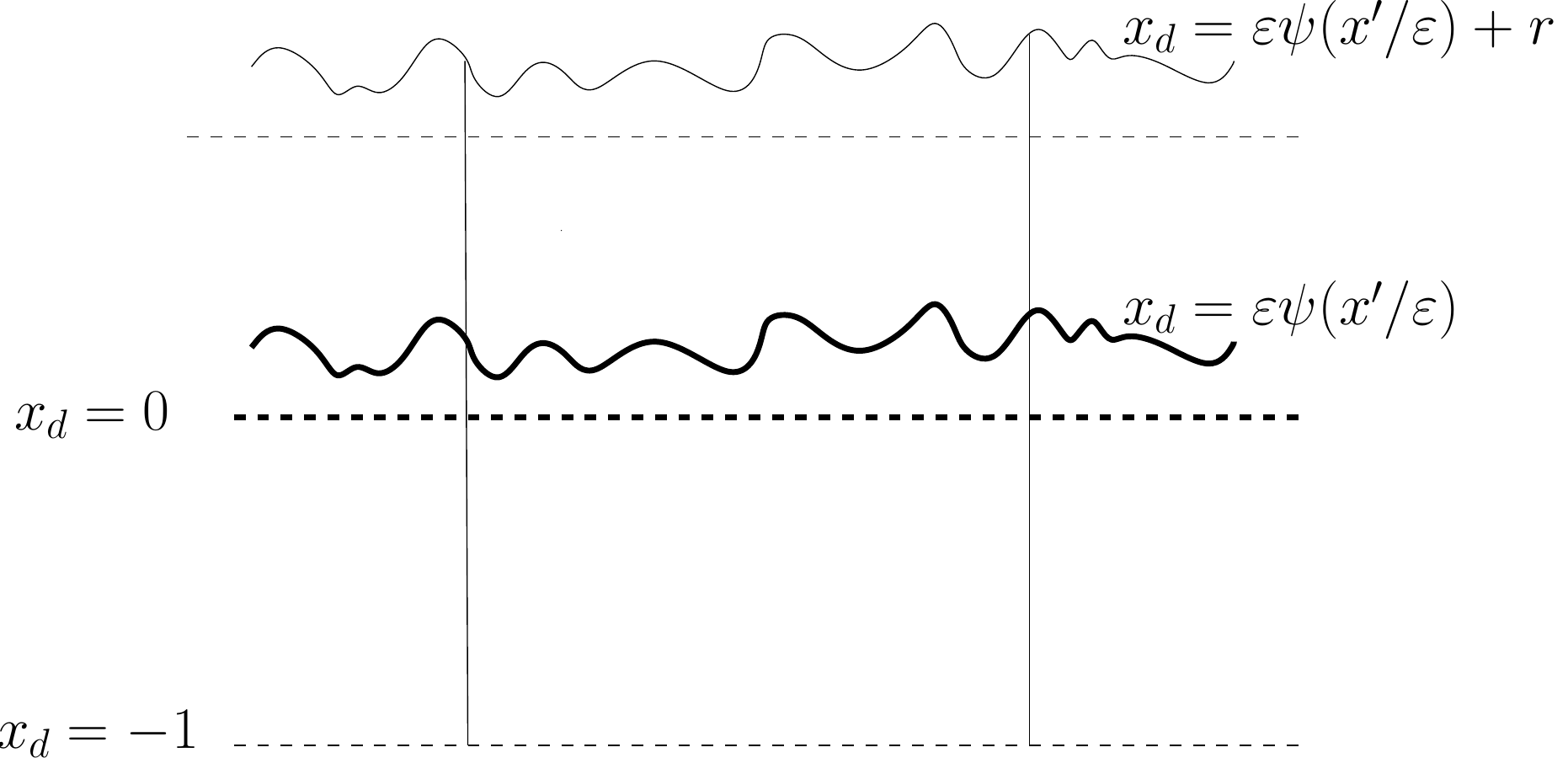}
\end{center}

The study of oscillating boundaries and roughness induced effects is an area of active applied and theoretical research. The applications involve a lot of different scales and range from geophysics \cite{DGVDormy06,CLS12,Lemouel_CBM06} to microfluidics \cite{ChuMa12}. From a mathematical point of view, the general goal is to describe the (averaged) effect of the oscillations of the boundary on the behavior of the solution to the partial differential equation. Two questions are of particular importance: well-posedness and asymptotic behavior far from the oscillating boundary. Well-posedness can be usually proved in very general settings. Indeed, some of the latest works \cite{DGVNMnoslip,DP_SC} have been focused on getting rid of any structure assumptions on the boundary. The analysis proves to be intricate, because a lot of the usual tools (Poincar\'e inequalities, Fourier analysis\dots) cannot be used. As far as the second question is concerned, some averaging properties of the oscillations have always been assumed in the existing litterature: the papers \cite{BassonDavid,GVALD_slip} are just two examples. Our hope is that the estimates in this paper will make it possible to resort to potential theoretical methods in order to investigate such questions.

Uniform Schauder estimates for elliptic systems with periodically oscillating coefficients have been pioneered by Avellaneda and Lin in a series of seminal papers \cite{alinscal,alin,alin2,Alin90P,alinLp}. Adapting a compactness method originating from the study of the regularity problem in the calculus of variations \cite{Almgren68,degiorgi61} (see also \cite{giaquintaBook}), they prove for instance that weak solutions $u^\varepsilon=u^\varepsilon(x)$ of 
\begin{equation*}
-\nabla\cdot A(x/\varepsilon)\nabla u^\varepsilon=0,\qquad x\in B(0,1),
\end{equation*}
satisfy the estimate 
\begin{equation}\label{lip1int}
\|\nabla u^\varepsilon\|_{L^\infty(B(0,1/2))}\leq C\|u^\varepsilon\|_{L^\infty(B(0,1))},
\end{equation}
with a constant $C>0$ uniform in $\varepsilon$. This interior estimate comes along with other interior and boundary H\"older and Lipschitz estimates for divergence form elliptic systems, as well as non-divergence form elliptic equations. A detailed review of interior estimates is done in section \ref{subsecintestalin} below. These works strongly rely on the periodicity assumption and are restricted to flat boundaries.

Another important contribution of the work of Avellaneda and Lin is the investigation of the asymptotics of Green $G^\varepsilon=G^\varepsilon(x,\tilde{x})$ and Poisson $P^\varepsilon=P^\varepsilon(x,\tilde{x})$ kernels associated to the operator with oscillating coefficients $-\nabla\cdot A(x/\varepsilon)\nabla$ and to the domain $\Omega\subset\mathbb R^d$. The key observation is that the analysis of the large scales $r=|x-\tilde{x}|\gg 1$ of $G^\varepsilon(x,\tilde{x})$ or $P^\varepsilon(x,\tilde{x})$ boils down, after proper rescaling, to the study of the local properties of $G^{\varepsilon/r}$ or $P^{\varepsilon/r}$. This can be done thanks to the local Schauder estimates uniform in $\varepsilon$. Such a line of ideas has been successfully implemented to expand the fundamental solution of $-\nabla\cdot A(x/\varepsilon)\nabla$ in \cite{alinLp}. It underlies the recent work of Kenig, Lin and Shen \cite{kls12}, where optimal expansions for Green and Neumann functions, along with their derivatives are derived. 

A boundary version of the Lipschitz estimate \eqref{lip1int} has also been used in the analysis of boundary layer correctors in homogenization. In the article \cite{BLtail} the investigation of the asymptotics far from the boundary of boundary layer correctors $v_{bl}=v_{bl}(y)$,
\begin{equation*}
\left\{
\begin{array}{rll}
-\nabla\cdot A(y)\nabla v_{bl}&=0,&y\cdot n>0,\\
v_{bl}&=v_0,&y\cdot n=0,
\end{array}
\right.
\end{equation*}
with $n\in\mathbb S^{d-1}$, is carried out using the representation of $v_{bl}$ via Poisson's kernel $P=P(y,\tilde{y})$ 
\begin{equation*}
v_{bl}(y)=\int_{\tilde{y}\cdot n=0}P(y,\tilde{y})v_0(\tilde{y})d\tilde{y}.
\end{equation*}
An expansion of $P(y,\tilde{y})$ for $|y-\tilde{y}|\gg 1$ is established: there exists an explicit kernel $P^{exp}=P^{exp}(y,\tilde{y})$ such that 
\begin{equation*}
|P(y,\tilde{y})-P^{exp}(y,\tilde{y})|\leq \frac{C}{|y-\tilde{y}|^{d-1+\kappa}},\qquad \kappa>0.
\end{equation*}
The latter makes it possible, using the ergodicity on the boundary, to show that
\begin{equation*}
v_{bl}(y)\stackrel{y\cdot n\rightarrow\infty}{\longrightarrow} V^\infty.
\end{equation*}
One of our motivations here is to generalize this analysis to systems in oscillating half-spaces
\begin{equation*}
\left\{
\begin{array}{rll}
-\nabla\cdot A(y)\nabla v_{bl}&=0,&y_d>\psi(y'),\\
v_{bl}&=v_0,&y_d=\psi(y').
\end{array}
\right.
\end{equation*}
Our paper is a first step in this direction. Going a step further, would require to assume some structure of $\psi$ ensuring averaging properties far from the boundary. Moreover, it would ask for a good understanding of the interplay between the oscillations of $\psi$, $A$ and $v_0$. We refer to \cite[sections 2 and 7]{BLtail} and to \cite[section 3]{DGVNMnoslip} to get an insight into these questions.

The work of Avellaneda and Lin on uniform estimates has far-reaching consequences in homogenization and potential theory: besides the references already cited let us mention \cite{dgvnm2,KLS12L2,KLSeig}. Generalizing it has been the purpose of intense research over the past few years. Without attempting to be exhaustive, we refer to the work of Kenig, Lin and Shen \cite{KLSNeumann} for the Neumann problem, of Geng and Shen \cite{Shen14parabolic} on parabolic systems, of Shen \cite{ShenAP14} on divergence form elliptic systems in an almost periodic setting and to the references cited in these papers. Moreover we are aware of one extension concerned with oscillating boundaries. In \cite{David09} G\'erard-Varet proves a uniform H\"older estimate close to an oscillating boundary for the Stokes system in fluid mechanics. The proof of a Lipschitz estimate for the system \eqref{sysoscueps}, which we address in this paper, is however much more involved and relies on new ideas.

Before going into the details of our results, let us state our setting.

\subsection{Framework}

Let $\lambda>0$, $0<\nu_0<1$ and $M_0>0$ be fixed in what follows. Let
\begin{equation*}
\omega:[0,\infty)\rightarrow[0,\infty),\ \mbox{such that}\ \omega(0)=0\ \mbox{and}\ \omega(t)\stackrel{t\rightarrow 0}{\longrightarrow}0,
\end{equation*}
be a fixed modulus of continuity. The boundary is a graph given by $\psi$ in the class $\mathcal C_{M_0}^{1,\omega}$ or $\mathcal C_{M_0}^{1,\nu_0}$ defined by
\begin{equation*}
\begin{aligned}
\mathcal C_{M_0}^{1,\omega}&:=\{\psi\in C^1(\mathbb R^{d-1}):\ 0\leq\psi\leq M_0,\ \|\nabla\psi\|_{L^\infty(\mathbb R^{d-1})}\leq M_0,\\
&\qquad\qquad |\nabla\psi(x')-\nabla\psi(\hat{x}')|\leq \omega(|x'-\hat{x}'|),\ \forall\ x',\ \hat{x}'\in\mathbb R^{d-1}\},\\
\mathcal C_{M_0}^{1,\nu_0}&:=\{\psi\in C^{1,\nu_0}(\mathbb R^{d-1}):\ 0\leq\psi\leq M_0,\ \|\nabla\psi\|_{L^\infty(\mathbb R^{d-1})}+[\nabla\psi]_{C^{0,\nu_0}(\mathbb R^{d-1})}\leq M_0\},
\end{aligned}
\end{equation*}
Keep in mind that $D^\varepsilon(0,r)$ and thus $u^\varepsilon$ depends on $\psi$, although we usually do not write explicitly the dependence in $\psi$. The reason for this is that all our results hold uniformly for $\psi$ in the above class $\mathcal C_{M_0}^{1,\omega}$ or $\mathcal C_{M_0}^{1,\nu_0}$.

We assume that the coefficients matrix $A=A(y)=(A^{\alpha\beta}_{ij}(y))$, with $1\leq\alpha,\ \beta\leq d$ and $1\leq i,\ j\leq N$ is real, that
\begin{equation}\label{smoothA2}
A\ \mbox{belongs to the class}\ C^{0,\nu_0}\ \mbox{and}\ \|A\|_{L^\infty(\mathbb R^{d})}+[A]_{C^{0,\nu_0}(\mathbb R^{d})}\leq M_0,
\end{equation}
that $A$ is uniformly elliptic i.e.
\begin{equation}\label{elliptA}
\lambda |\xi|^2\leq A^{\alpha\beta}_{ij}(y)\xi^\alpha_i\xi^\beta_j\leq \frac{1}{\lambda}|\xi|^2,\quad\mbox{for all}\ \xi=(\xi^\alpha_i)\in\mathbb R^{dN},\ y\in\mathbb R^d
\end{equation}
and periodic i.e. 
\begin{equation}\label{perA}
A(y+z)=A(y),\quad\mbox{for all}\ y\in\mathbb R^d,\ z\in\mathbb Z^d.
\end{equation}
We say that $A$ belongs to the class $\mathcal A^{0,\nu_0}$ if $A$ satisfies \eqref{smoothA2}, \eqref{elliptA} and \eqref{perA}. Starred quantities always refer to the transposed operator $-\nabla\cdot A^*(x/\varepsilon)\nabla$, where for all $\alpha,\ \beta\in\{1,\ldots\ d\}$ and $i,\ j\in\{1,\ldots\ N\}$, $(A^*)^{\alpha\beta}_{ij}:=A^{\beta\alpha}_{ji}$.

\subsection{Outline of our main results}

Our ultimate goal is to prove a Lipschitz estimate close to the oscillating boundary for $u^\varepsilon$ weak solution of \eqref{sysoscueps}. Our main focus is on getting uniformity in $\varepsilon$, $\psi\in \mathcal C_{M_0}^{1,\nu_0}$ and $A\in\mathcal A^{0,\nu_0}$. Our main results can be stated as follows.

\begin{rst}[H\"older, Proposition \ref{propboundaryholder}]\label{rst1}
Let $\kappa,\ \kappa'>0$. There exists $C>0$, such that for all 
\begin{equation*}
0<\mu<\min\left(1-d/(d+\kappa),2-d/(d/2+\kappa')\right), 
\end{equation*}
for all $\psi\in\mathcal C_{M_0}^{1,\omega}$, for all $A\in\mathcal A^{0,\nu_0}$, for all $\varepsilon>0$, for all $f\in L^{d/2+\kappa'}(D^\varepsilon(0,1))$, for all $F\in L^{d+\kappa}(D^\varepsilon(0,1))$, for all $u^\varepsilon$ weak solution to \eqref{sysoscueps}
\begin{equation*}
[u^\varepsilon]_{C^{0,\mu}(\overline{D^\varepsilon(0,1/2)})}\leq C\left\{\|u^\varepsilon\|_{L^2(D^\varepsilon(0,1))}+\|f\|_{L^{d/2+\kappa'}(D^\varepsilon(0,1))}+\|F\|_{L^{d+\kappa}(D^\varepsilon(0,1))}\right\}.
\end{equation*}
Notice that $C$ depends on $d$, $N$, $M_0$, on the modulus of continuity $\omega$ of $\nabla\psi$, $\lambda$, $\kappa$ and $\kappa'$.
\end{rst}

\begin{rst}[Lipschitz, Theorem \ref{theoboundarylip}]\label{rst2}
Let $0<\mu<1$ and $\kappa>0$. There exists $C>0$, such that for all $\psi\in\mathcal C_{M_0}^{1,\nu_0}$, for all $A\in\mathcal A^{0,\nu_0}$, for all $\varepsilon>0$, for all $f\in L^{d+\kappa}(D^\varepsilon(0,1))$, for all $F\in C^{0,\mu}(D^\varepsilon(0,1))$, for all $u^\varepsilon$ weak solution to \eqref{sysoscueps} 
\begin{equation*}
\|\nabla u^\varepsilon\|_{L^\infty(D^\varepsilon(0,1/2))}\leq C\left\{\|u^\varepsilon\|_{L^\infty(D^\varepsilon(0,1))}+\|f\|_{L^{d+\kappa}(D^\varepsilon(0,1))}+\|F\|_{C^{0,\mu}(D^\varepsilon(0,1))}\right\}.
\end{equation*}
Notice that $C$ depends on $d$, $N$, $M_0$, $\lambda$, $\nu_0$, $\kappa$ and $\mu$.
\end{rst}

In section \ref{sec2scales}, we address a generalization of these estimates to the case when the coefficients are oscillating at a scale $\alpha$ and the boundary is oscillating at another scale $\beta$. There is no connection between $\alpha$ and $\beta$. A boundary H\"older estimate uniform in $\alpha$ and $\beta$ is stated in Proposition \ref{propboundaryholderab}. A boundary Lipschitz estimate uniform in $\alpha$ and $\beta$ is stated in Theorem \ref{theoboundarylipalphabeta}.

The next estimate compares the Green function $G^\varepsilon=G^\varepsilon(x,\tilde{x})$ associated to the operator with oscillating coefficients $-\nabla\cdot A(x/\varepsilon)\nabla$ and the oscillating domain $x_d>\varepsilon\psi(x'/\varepsilon)$, to the Green function $G^0=G^0(x,\tilde{x})$ associated to the homogenized operator with constant coefficients $-\nabla\cdot \overline{A}\nabla$ and the flat domain $x_d>0$.

\begin{rst}[Expansion, Theorem \ref{theogeps-g0}]\label{rst3}
There exists $C>0$, such that for all $\psi\in C^{1,\nu_0}_{M_0}$, for all $A\in \mathcal A^{0,\nu_0}$, for all $\varepsilon>0$, for all $x,\ \tilde{x}\in D^\varepsilon_+$
\begin{equation*}
|G^\varepsilon(x,\tilde{x})-G^0(x,\tilde{x})|\leq \frac{C\varepsilon}{|x-\tilde{x}|^{d-1}}.
\end{equation*}
Notice that $C$ depends on $d$, $N$, $M_0$, $\lambda$ and $\nu_0$.
\end{rst}

Along with these theorems, we prove among other things: estimates on Green and Poisson kernels (see Lemmas \ref{lemestgreend=3}, \ref{lemestgreend=2} and Propositions \ref{propestgeps} and \ref{propestpeps}), and a maximum principle for systems in a domain with oscillating boundary (Lemma \ref{lemLinftygLinfty}).

\subsection{Comments and strategy of proof}

We focus here on the Lipschitz estimate of Result \ref{rst2}. Let us make some comments:
\begin{description}
\item[boundary smoothness] Taking $\psi$ only Lipschitz would not be enough to get that $u^\varepsilon$ is Lipschitz close to the boundary. We need $C^{1,\nu_0}$ regularity on $\psi$.
\item[blow-up] For $\psi\in\mathcal C_{M_0}^{1,\nu_0}$, the norm
\begin{equation}\label{blowuppsi}
[\nabla(\varepsilon\psi(x'/\varepsilon))]_{C^{0,\nu_0}}=[\nabla\psi(x'/\varepsilon)]_{C^{0,\nu_0}}=O(\varepsilon^{-\nu_0})
\end{equation}
blows up in the limit $\varepsilon\rightarrow 0$. 
\item[lack of structure] Appart from taking $\psi$ in the class $\mathcal C_{M_0}^{1,\nu_0}$, the boundary has no structure. In particular, it is neither periodic, nor quasiperiodic.
\item[non periodic homogenization] Take $N=1$, $A=\Idd_d$ and $f=F=0$. In this case, $u^\varepsilon$ is a weak solution to
\begin{equation*}
\left\{
\begin{array}{rll}
-\Delta u^\varepsilon&=0,&x\in D^\varepsilon(0,1),\\
u^\varepsilon&=0,&x\in \Delta^\varepsilon(0,1).
\end{array}
\right.
\end{equation*}
If one maps the oscillating domain into the flat domain, using the (non unique) mapping
\begin{equation}\label{flatteningmap}
(x',x_d)\in\mathbb R^d_+\stackrel{\Psi^\varepsilon}{\longmapsto}(x',x_d+\varepsilon\psi(x'/\varepsilon)\vartheta(x_d/\varepsilon))\in D^\varepsilon_+,
\end{equation}
($\vartheta$ is a cut-off function in the vertical direction) we get that for all $(x',x_d)\in\mathbb R^d_+$
\begin{equation*}
\tilde{u}^\varepsilon(x',x_d):=u^\varepsilon(x',x_d+\varepsilon\psi(x'/\varepsilon)\vartheta(x_d/\varepsilon))
\end{equation*}
is a weak solution to
\begin{equation}\label{sysosctildeueps}
\left\{
\begin{array}{rll}
-\nabla\cdot \tilde{A}(x/\varepsilon)\nabla \tilde{u}^\varepsilon&=0,&x\in\tilde{D}^\varepsilon(0,1),\\
\tilde{u}^\varepsilon&=0,&x\in \tilde{\Delta}^\varepsilon(0,1),
\end{array}
\right.
\end{equation}
where $(x',x_d)\in\tilde{D}^\varepsilon(0,1)$ if $|x'|<1$ and $0<x_d<1+\varepsilon\psi(x'/\varepsilon)$, and for all $x\in\mathbb R^d$
\begin{equation*}
\tilde{A}(x/\varepsilon)=
\left(\begin{array}{c|c}
\Idd_{d-1}&-\vartheta(x_d/\varepsilon)\nabla_{x'}\psi(x'/\varepsilon)\\
\hline
0& 1-\vartheta'(x_d/\varepsilon)\psi(x'/\varepsilon)
\end{array}\right)^T
\left(\begin{array}{c|c}
\Idd_{d-1}&-\vartheta(x_d/\varepsilon)\nabla_{x'}\psi(x'/\varepsilon)\\
\hline
0 & 1-\vartheta'(x_d/\varepsilon)\psi(x'/\varepsilon)
\end{array}\right).
\end{equation*}
Notice that the oscillations are localized near the boundary. Yet the system \eqref{sysosctildeueps} has no structure and the homogenization may even not be possible.
\end{description}
To conclude, there is no way around dealing with the strong oscillations of the boundary.

\paragraph{Further remarks}

\begin{rem}
The H-convergence theory tells us that a subsequence of $\tilde{u}^\varepsilon$ solving \eqref{sysosctildeueps} converges weakly to a solution $\tilde{u}^0$ of an elliptic system with non oscillating coefficients $A^0=A^0(x)$. Notice that $A^0$ may depend on the subsequence, and is not unique. However, appart from $A^0\in L^\infty$, we do not have any information on the regularity of $A^0$.
\end{rem}

\begin{rem}[large scales]
In all what follows, the main issue when proving estimates uniform in $\varepsilon$ comes from the large scales $O(1)$ when $\varepsilon\rightarrow 0$. For the small scales $O(\varepsilon)$ we can always rely on classical estimates.
\end{rem}

\begin{rem}[boundedness of the boundary]
For our proof of the Lipschitz estimate, it is crucial that $\psi$ is bounded in $L^\infty(\mathbb R^{d-1})$. This ensures the convergence of the oscillating half-space to a flat one. If instead, we have a graph given by $\psi\in C^{1,\nu_0}(\mathbb R^{d-1})$, such that for instance
\begin{equation*}
\psi(0)=0\ \mbox{and}\ \|\nabla\psi\|_{L^\infty(\mathbb R^{d-1})}\leq M_0,
\end{equation*}
then for all $x'\in\mathbb R^{d-1}$
\begin{equation*}
|\varepsilon\psi(x'/\varepsilon)|=|\varepsilon\psi(x'/\varepsilon)-\varepsilon\psi(0)|\leq M_0|x'|.
\end{equation*}
This implies that the graph of $\psi^\varepsilon=\varepsilon\psi(\cdot/\varepsilon)$ is squeezed is the complement of a cone of $\mathbb R^d$. Furthermore, since
\begin{equation*}
\|\nabla\psi^\varepsilon\|_{L^\infty(\mathbb R^{d-1})}=O(1)\ \mbox{and}\ \psi^\varepsilon(0)=0,
\end{equation*}
for all $0<\mu<1$ by Ascoli and Arzela's theorem, we may extract a subsequence converging strongly in $C^{0,\mu}(\mathbb R^{d-1})$ to $\psi^0$. Of course, the limit $\psi^0$ of the subsequence is far from being unique. This smoothness is enough to get H\"older regularity, but is too weak to get Lipschitz regularity. 
\end{rem}

\begin{rem}[position of the boundary]
So as to avoid pointless technicalities, we work in the whole paper with a boundary lying above $x_d=0$, i.e. $0\leq \psi$. This condition can be always achieved by translating in the vertical direction. Doing so may change the coefficients of the operator, but the resulting coefficients $\tilde{A}$ still belong to $\mathcal A^{0,\nu_0}$ and the estimates of the paper apply.
\end{rem}

\begin{rem}[boundedness of the boundary and flattening]
Notice also that using the mapping \eqref{flatteningmap} in order to flatten the boundary leads to a loss of the crucial information $\psi\in L^{\infty}(\mathbb R^{d-1})$ in the system \eqref{sysosctildeueps}. 
\end{rem}

\paragraph*{Strategy of proof}

At first, the blow-up \eqref{blowuppsi} seems to be a huge obstruction to uniform Lipschitz estimates. Indeed, we have to rely on Schauder estimates. Applying these estimates directly leads to bounds depending on the $C^{1,\nu_0}$ semi-norm of $\varepsilon\psi(\cdot/\varepsilon)$, which blows-up. The key is that we only need the $C^{1,\mu}$ regularity estimate at small scale $O(\varepsilon)$. At large scale, we only see the Lipschitz regularity of $\psi$ and
\begin{equation*}
\|\nabla(\varepsilon\psi(x'/\varepsilon))\|_{L^\infty}=O(1).
\end{equation*}
Moreover, what exempts us from having to deal with the homogenization of \eqref{sysosctildeueps}, is that the amplitude of the boundary is of the order of $\varepsilon$, so that when $\varepsilon\rightarrow 0$, the boundary tends to a flat one. This argument is the key to the so-called ``improvement lemmas'', in which the compactness analysis is carried out.

\subsection{Overview of the paper}

Section \ref{secprelim} is concerned with a review of classical Schauder estimates, as well as a complete description of the interior estimates to be found in \cite{alin}. In section \ref{secbdrayholder}, we give a proof of Result \ref{rst1}, the H\"older estimate uniform in $\varepsilon$ in the oscillating domain. In section \ref{secbdarycor}, we carry out the analysis of a boundary corrector, which is crucial in order to get the boundary Lipschitz estimate. The boundary Lipschitz estimate of Result \ref{rst2} is proved first for equations with constant coefficients in section \ref{secliplap}, then for general elliptic systems with periodically oscillating coefficients in section \ref{secbdrliposc}. Pointwise estimates on Green and Poisson kernels, as well as the first-order expansion for Green's kernel in the oscillating domain (Result \ref{rst3}) are established in section \ref{secasygreen}. In section \ref{sec2scales}, we tackle the generalization of the uniform boundary estimates to systems where the coefficients and the boundary oscillate at two different scales. The last part \ref{secgen} is devoted to generalizations of our estimates to boundaries with a macroscopic behavior and to inclined half-spaces.

\subsection{Further notations}

For $\varepsilon>0$, $r>0$, for $x_0'\in\mathbb R^{d-1}$, let
\begin{equation*}
\begin{aligned}
D^\varepsilon_{\psi}(0,r)=D^\varepsilon(0,r)&:=\left\{(x',x_d),\ |x'|<r,\ \varepsilon\psi(x'/\varepsilon)<x_d<\varepsilon\psi(x'/\varepsilon)+r\right\},\\
\Delta^\varepsilon_{\psi}(0,r)=\Delta^\varepsilon(0,r)&:=\left\{(x',x_d),\ |x'|<r,\ x_d=\varepsilon\psi(x'/\varepsilon)\right\},\\
D^\varepsilon_{\psi,+}=D^\varepsilon_+&:=\left\{(x',x_d),\ \varepsilon\psi(x'/\varepsilon)<x_d\right\},\\
D^\varepsilon_{\psi,-}=D^\varepsilon_-&:=\left\{(x',x_d),\ \varepsilon\psi(x'/\varepsilon)>x_d\right\},\\
\Delta^\varepsilon_{\psi}=\Delta^\varepsilon&:=\left\{(x',x_d),\ x_d=\varepsilon\psi(x'/\varepsilon)\right\},\\
D^\varepsilon_{\psi}(x_0',r)=D^\varepsilon(x_0',r)&:=\left\{(x',x_d),\ |x'-x_0'|<r,\ \varepsilon\psi(x'/\varepsilon)<x_d<\varepsilon\psi(x'/\varepsilon)+r\right\},\\
\Delta^\varepsilon_{\psi}(x_0',r)=\Delta^\varepsilon(x_0',r)&:=\left\{(x',x_d),\ |x'-x_0'|<r,\ x_d=\varepsilon\psi(x'/\varepsilon)\right\},\\
D^0(0,r)&:=\left\{(x',x_d),\ |x'|<r,\ 0<x_d<r\right\},\\
\Delta^0(0,r)&:=\left\{(x',0),\ |x'|<r\right\},\\
D^{-1}(0,r)&:=\left\{(x',x_d),\ |x'|<r,\ -1<x_d<r\right\},
\end{aligned}
\end{equation*}
where $|x'|=\max_{i=1,\ldots\ d}|x_i|$. Notice that $x_0:=(x_0',\varepsilon\psi(x_0'/\varepsilon))\in \Delta(x_0',r)$. For an arbitrary point $x_0\in \mathbb R^d$, $B(x_0,r)$ is simply the ball of center $x_0$ and radius $r$. We usually drop the subscripts $\psi$, except at very few places. For $\varepsilon>0$, $x\in D^\varepsilon_+$, 
\begin{equation*}
\delta^\varepsilon_{\psi}(x)=\delta(x):=x_d-\varepsilon\psi(x'/\varepsilon).
\end{equation*}
We write $\delta(x)$ when no confusion is possible, even if this number depends on $\varepsilon$ and $\psi$.

Let also
\begin{equation*}
(\overline{u})_{D^\varepsilon(0,r)}:=\intbar_{D^\varepsilon(0,r)}u=\frac{1}{|D^\varepsilon(0,r)|}\int_{D^\varepsilon(0,r)}u.
\end{equation*}
We will write $(\overline{u})_{0,r}$ in short when it does not lead to any confusion. In general, for a point $x_0=(x_0',\varepsilon\psi(x_0'/\varepsilon))\in \Delta^\varepsilon_+$
\begin{equation*}
(\overline{u})_{D^\varepsilon(x_0',r)}=(\overline{u})_{x_0',r}:=\intbar_{D^\varepsilon(x_0',r)}u=\frac{1}{|D^\varepsilon(x_0',r)|}\int_{D^\varepsilon(x_0',r)}u.
\end{equation*}
The Lebesgue measure of a set is denoted by $|\cdot|$. In the sequel, $C>0$ is always a constant uniform in $\varepsilon$ which may change from line to line. For a positive integer $m$, let also $\Idd_m$ denote the identity matrix $M_m(\mathbb R)$. 

\section{Preliminaries}
\label{secprelim}

\subsection{Classical Schauder regularity}

Let $u=u(y)$ be a weak solution to
\begin{equation*}
\left\{
\begin{array}{rll}
-\nabla\cdot A(y)\nabla u&=f+\nabla\cdot F,&y\in D^1_\psi(0,1),\\
u&=0,&y\in \Delta^1_\psi(0,1).
\end{array}
\right.
\end{equation*}
For classical Schauder estimates, we refer to \cite[Chapter III]{giaquintaBook}, \cite[Chapter 5]{giaquinta12} and in a slightly different context (hydrodynamics) to \cite{GiaMod82}.

Following the method of Campanato, the Schauder estimates can be obtained from purely energetical considerations (Cacciopoli and Poincar\'e inequalities), without relying on potential theory. The central result in this theory is the characterization of H\"older continuity in terms of Campanato spaces \cite[Theorem 1.2]{giaquintaBook} and \cite[Theorem 5.5]{giaquinta12}: when $\Omega\subset\mathbb R^d$ is a Lipschitz domain and $u\in C^{0,\mu}(\overline{\Omega})$ for $0<\mu<1$, then 
\begin{equation*}
[u]_{C^{0,\mu}(\overline{\Omega})}^2\sim\sup_{x_0\in\overline{\Omega},\ \rho>0}\rho^{-2\mu}\intbar_{B(x_0,\rho)\cap\Omega}|u-(\overline{u})_{B(x_0,\rho)\cap\Omega}|^2dx,
\end{equation*}
where $\sim$ means that the semi-norms on the left and right hand sides are equivalent.

\begin{theo}[classical H\"older regularity]\label{theoclassholder}
Let $\kappa,\ \kappa'>0$. Assume that $\psi\in C^1(\mathbb R^{d-1})$,
\begin{equation*}
\|\nabla\psi\|_{L^\infty(\mathbb R^{d-1})}\leq M_0,\ |\nabla\psi(y')-\nabla\psi(\hat{y}')|\leq \omega(|y'-\hat{y}'|),\ \forall\ y',\ \hat{y}'\in\mathbb R^{d-1}
\end{equation*}
and $A\in\mathcal A^{0,\nu_0}$. Assume furthermore that $f\in L^{d/2+\kappa'}(D^1_\psi(0,1))$, $F\in L^{d+\kappa}(D^1_\psi(0,1))$ and $u\in L^2(D^1_\psi(0,1))$. Then $u\in C^{0,\sigma}(D^1_\psi(0,1/2))$ and
\begin{equation}\label{holderclassestu}
[u]_{C^{0,\sigma}(D^1_\psi(0,1/2))}\leq C\left\{\|u\|_{L^2(D^1_\psi(0,1))}+\|f\|_{L^{d/2+\kappa'}(D^1_\psi(0,1))}+\|F\|_{L^{d+\kappa}(D^1_\psi(0,1))}\right\},
\end{equation}
for $\sigma:=\min\left(1-d/(d+\kappa),2-d/(d/2+\kappa')\right)$.
\end{theo}

For elements of proof, we refer to Theorems 5.17, 5.21 and Corollary 5.18 in \cite{giaquinta12}, as well as Theorem 2.8 in \cite{GiaMod82}.

\begin{rem}[regularity on the coefficients]
Notice that this estimate is true for 
\begin{equation}\label{smoothA1}
\begin{aligned}
&A\ \mbox{in the class}\ C^0\ \mbox{such that}\ \|A\|_{L^\infty(\mathbb R^{d})}\leq M_0\\
&\mbox{and}\ |A(y)-A(\hat{y})|\leq \omega(|y-\hat{y}|),\ \forall\ y,\ \hat{y}\in\mathbb R^{d}.
\end{aligned}
\end{equation}
However, we do not use this optimal regularity in our work, since our focus is on Lipschitz estimates, for which $A$ has to be $C^{0,\nu_0}$.
\end{rem}

\begin{rem}[regularity on the source term]
Notice that $F\in L^{d+\kappa}(B(0,1))\subset L^{2,d-2+2\kappa/(d+\kappa)}$, which is the Morrey space (see \cite[Definition 5.1]{giaquinta12}). 
\end{rem}

\begin{rem}
The constant $C$ in \eqref{holderclassestu} depends on the $C^{0,\nu_0}$ norm of $A$, and on the $L^\infty$ norm and on the modulus of continuity of $\nabla\psi$ since the theorem is proved by flattening the boundary. In particular, $C$ does not depend on $\|\psi\|_{L^\infty}$. 
\end{rem}

\begin{theo}[classical Lipschitz and $C^{1,\sigma}$ regularity]\label{theoclasslip}
Let $\kappa>0$ and $0<\mu<1$. Assume that $\psi\in C^{1,\nu_0}(\mathbb R^{d-1})$,
\begin{equation*}
\|\nabla\psi\|_{L^\infty(\mathbb R^{d-1})}+[\nabla\psi]_{C^{0,\nu_0}(\mathbb R^{d-1})}\leq M_0,
\end{equation*}
and $A\in\mathcal A^{0,\nu_0}$. Assume furthermore that $f\in L^{d+\kappa}(D^1_\psi(0,1))$, $F\in C^{0,\mu}(D^1_\psi(0,1))$ and $u\in L^2(D^1_\psi(0,1))$. Then $u\in W^{1,\infty}(D^1_\psi(0,1/2))$ and
\begin{equation}\label{lipclassestu}
\|\nabla u\|_{L^\infty(D^1_\psi(0,1/2))}+[\nabla u]_{C^{1,\sigma}(D^1_\psi(0,1/2))}\leq C\{\|u\|_{L^2(D^1_\psi(0,1))}+\|f\|_{L^{d+\kappa}(D^1_\psi(0,1))}+\|F\|_{C^{0,\mu}(D^1_\psi(0,1))}\},
\end{equation}
where $\sigma:=\min\left(1-d/(d+\kappa),\mu,\nu_0\right)$.
\end{theo}

For elements of proof, we refer to Theorems 5.19, 5.20 and 5.21 in \cite{giaquinta12}, as well as Theorem 2.8 in \cite{GiaMod82}.

\begin{rem}
The constant $C$ in \eqref{lipclassestu} depends on $\|A\|_{C^{0,\nu_0}}$ and on $\|\nabla\psi\|_{C^{0,\nu_0}}$ since the theorem is proved by flattening the boundary. Again, the constant does not depend on the $L^\infty$ norm of $\psi$.
\end{rem}

\subsection{Homogenization and weak convergence}

We recall the standard weak convergence result in periodic homogenization for a fixed domain $\Omega$. As usual, the constant homogenized matrix $\overline{A}=\overline{A}^{\alpha\beta}\in M_N(\mathbb R)$ is given by
\begin{equation}\label{defA0}
\overline{A}^{\alpha\beta}:=\int_{\mathbb T^d}A^{\alpha\beta}(y)dy+\int_{\mathbb T^d}A^{\alpha\gamma}(y)\partial_{y_\gamma}\chi^\beta(y)dy,
\end{equation}
where the family $\chi=\chi^\gamma(y)\in M_N(\mathbb R)$, $y\in\mathbb T^d$, solves the cell problems
\begin{equation}\label{eqdefchi}
-\nabla_y \cdot A(y)\nabla_y \chi^\gamma=\partial_{y_\alpha}A^{\alpha\gamma},\ y\in \mathbb T^d\qquad \mbox{and}\qquad \int_{\mathbb T^d}\chi^\gamma(y)dy=0.
\end{equation}

\begin{theo}[weak convergence]\label{theoweakcvhomo}
Let $\Omega$ be a bounded Lipschitz domain in $\mathbb R^d$ and let $u_k\in W^{1,2}(\Omega)$ be a sequence of weak solutions to
\begin{equation*}
-\nabla\cdot A_k(x/\varepsilon_k)\nabla u_k=f_k\in (W^{1,2}(\Omega))',
\end{equation*}
where $\varepsilon_k\rightarrow 0$ and the matrices $A_k=A_k(y)\in L^\infty$ satisfy \eqref{elliptA} and \eqref{perA}. Assume that there exist $f\in (W^{1,2}(\Omega))'$ and $u_k\in W^{1,2}(\Omega)$, such that $f_k\longrightarrow f$ strongly in $(W^{1,2}(\Omega))'$, $u_k\rightarrow u_0$ strongly in $L^2(\Omega)$ and $\nabla u_k\rightharpoonup\nabla u^0$ weakly in $L^2(\Omega)$. Also assume that the constant matrix $\overline{A_k}$ defined by \eqref{defA0} with $A$ replaced by $A_k$ converges to a constant matrix $A^0$. Then
\begin{equation*}
A_k(x/\varepsilon_k)\nabla u_k\rightharpoonup A^0\nabla u^0\quad\mbox{weakly in}\ L^2(\Omega) 
\end{equation*}
and 
\begin{equation*}
\nabla\cdot A^0\nabla u^0=f\in (W^{1,2}(\Omega))'.
\end{equation*}
\end{theo}

For a proof, which relies on the classical oscillating test function argument, we refer for instance to \cite[Lemma 2.1]{KLSNeumann}. This is an interior convergence result, since no boundary condition is prescribed on $u_k$.

\subsection{Interior estimates in homogenization}
\label{subsecintestalin}

We recall here two interior estimates proved by Avellaneda and Lin in \cite{alin}.

\begin{theo}[H\"older estimate, {\cite[Lemma 9]{alin}}]
For all $\kappa>0$, there exists $C>0$ such that for all $A\in\mathcal A^{0,\nu_0}$, for all $\varepsilon>0$, for all $F\in L^{d+\kappa}(B(0,1))$, for all $u^\varepsilon\in L^2(B(0,1))$ weak solution to
\begin{equation*}
-\nabla\cdot A(x/\varepsilon)\nabla u^\varepsilon=\nabla\cdot F\quad\mbox{in}\ B(0,1),
\end{equation*}
the following estimate holds
\begin{equation*}
[u^\varepsilon]_{C^{0,\mu}(B(0,1/2))}\leq C\left\{\|u^\varepsilon\|_{L^2(B(0,1))}+\|F\|_{L^{d+\kappa}(B(0,1))}\right\},
\end{equation*}
where $\mu:=1-d/(d+\kappa)$. Notice that $C$ depends on $d$, $N$, on $\|A\|_{C^{0,\nu_0}}$ i.e. on $M_0$, on $\lambda$ and $\kappa$.
\end{theo}

\begin{rem}[rescaled estimate]
Assume 
\begin{equation*}
-\nabla\cdot A(x/\varepsilon)\nabla u^\varepsilon=\nabla\cdot F\quad\mbox{in}\ B(0,r),
\end{equation*}
for $r>0$. Then,
\begin{equation}\label{intestrescaled}
[u^\varepsilon]_{C^{0,\mu}(B(0,r/2))}\leq C\bigl\{r^{-d/2-\mu}\|u^\varepsilon\|_{L^2(B(0,r))}+\underbrace{r^{1-\mu-d/(d+\kappa)}}_{=1}\|F\|_{L^{d+\kappa}(B(0,r))}\bigr\}.
\end{equation}
\end{rem}

\begin{rem}[more general source terms]
A slight generalization of this H\"older estimate is the following: for $\kappa,\ \kappa'>0$, for $u^\varepsilon$ weak solution to
\begin{equation*}
-\nabla\cdot A(x/\varepsilon)\nabla u^\varepsilon=f+\nabla\cdot F\quad\mbox{in}\ B(0,1),
\end{equation*}
we have
\begin{equation*}
[u^\varepsilon]_{C^{0,\mu}(B(0,1/2))}\leq C\left\{\|u^\varepsilon\|_{L^2(B(0,1))}+\|f\|_{L^{d/2+\kappa'}(B(0,1))}+\|F\|_{L^{d+\kappa}(B(0,1))}\right\},
\end{equation*}
with $\mu:=\min\left(1-d/(d+\kappa),2-d/(d/2+\kappa')\right)$. The rescaled estimate reads
\begin{multline}\label{intestrescaledbis}
[u^\varepsilon]_{C^{0,\mu}(B(0,r/2))}\leq C\bigl\{r^{-d/2-\mu}\|u^\varepsilon\|_{L^2(B(0,r))}+r^{2-\mu-2d/(d+2\kappa')}\|f\|_{L^{d/2+\kappa'}(B(0,r))}\bigr.\\
\bigl.
+r^{1-\mu-d/(d+\kappa)}\|F\|_{L^{d+\kappa}(B(0,r))}\bigr\}.
\end{multline}
\end{rem}

\begin{rem}[regularity on the coefficients]
Notice again that this estimate is true for $A\in C^0$ satisfying \eqref{smoothA1}. However, we do not use this fact here, since our focus is on Lipschitz estimates, for which $A$ has to be $C^{0,\nu_0}$.
\end{rem}

\begin{rem}
In view of Theorem 5.17 and Corollary 5.18 in \cite{giaquinta12}, the classical Schauder regularity applies and gives $u^\varepsilon\in C^{0,\mu}(B(0,1/2))$, with $\mu:=1-d/(d+\kappa)=\kappa/(d+\kappa)$. Of course, this classical estimate is not uniform in $\varepsilon$, and the contribution of \cite{alin} is to show that there is a way, via homogenization, to get uniform H\"older estimates in $\varepsilon$ when the coefficients are periodically oscillating.
\end{rem}

\begin{theo}[Lipschitz estimate, {\cite[Lemma 16]{alin}}]
For all $\kappa>0$, there exists $C>0$ such that for all $A\in\mathcal A^{0,\nu_0}$, for all $\varepsilon>0$, for all $f\in L^{d+\kappa}(B(0,1))$, for all $u^\varepsilon\in L^\infty(B(0,1))$ weak solution to
\begin{equation*}
-\nabla\cdot A(x/\varepsilon)\nabla u^\varepsilon=f\quad\mbox{in}\ B(0,1),
\end{equation*}
the following estimate holds
\begin{equation*}
\|\nabla u^\varepsilon\|_{L^\infty(B(0,1/2))}\leq C\left\{\|u^\varepsilon\|_{L^\infty(B(0,1))}+\|f\|_{L^{d+\kappa}(B(0,1))}\right\}.
\end{equation*}
Notice that $C$ depends on $d$, $N$, $M_0$, $\lambda$ and $\kappa$.
\end{theo}

\begin{rem}[$L^\infty$ control of $u^\varepsilon$]
The control of $\|u^\varepsilon\|_{L^\infty(B(0,1))}$ is used in a compactness argument of Ascoli-Arzela type.
\end{rem}

\begin{rem}[rescaled estimate]
Assume 
\begin{equation*}
-\nabla\cdot A(x/\varepsilon)\nabla u^\varepsilon=f\quad\mbox{in}\ B(0,r),
\end{equation*}
for $r>0$. Then,
\begin{equation}\label{intgradestrescaled}
\|\nabla u^\varepsilon\|_{L^\infty(B(0,r/2))}\leq C\left\{r^{-1}\|u^\varepsilon\|_{L^\infty(B(0,r))}+r^{1-d/(d+\kappa)}\|f\|_{L^{d+\kappa}(B(0,r))}\right\}.
\end{equation}
\end{rem}

\begin{rem}[more general source terms]
A slight generalization of this Lipschitz estimate is the following: for $\kappa>0$, for $0<\mu<1$, for $u^\varepsilon$ weak solution to
\begin{equation*}
-\nabla\cdot A(x/\varepsilon)\nabla u^\varepsilon=f+\nabla\cdot F\quad\mbox{in}\ B(0,1),
\end{equation*}
we have
\begin{equation*}
\|\nabla u^\varepsilon\|_{L^\infty(B(0,1/2))}\leq C\left\{\|u^\varepsilon\|_{L^\infty(B(0,1))}+\|f\|_{L^{d+\kappa}(B(0,1))}+\|F\|_{C^{0,\mu}(B(0,1))}\right\}.
\end{equation*}
The rescaled estimate then reads
\begin{equation}\label{intgradestrescaledbis}
\|\nabla u^\varepsilon\|_{L^\infty(B(0,r/2))}\leq C\left\{r^{-1}\|u^\varepsilon\|_{L^\infty(B(0,r))}+r^{1-d/(d+\kappa)}\|f\|_{L^{d+\kappa}(B(0,r))}+r^{\mu}\|F\|_{C^{0,\mu}(B(0,r))}\right\}.
\end{equation}
\end{rem}

\section{Boundary H\"older estimate}
\label{secbdrayholder}

The following proposition is a generalization to oscillating boun\-daries of Lemma 12 in \cite{alin}. A similar estimate, in the case of the Stokes system with oscillating boundary, is to be found in \cite[estimate (5.4)]{David09}.

\begin{prop}\label{propboundaryholder}
Let $\kappa,\ \kappa'>0$. There exist $C>0$, $\varepsilon_0>0$ such that for all 
\begin{equation*}
0<\mu<\min\left(1-d/(d+\kappa),2-d/(d/2+\kappa')\right), 
\end{equation*}
for all $\psi\in\mathcal C_{M_0}^{1,\omega}$, for all $A\in\mathcal A^{0,\nu_0}$, for all $\varepsilon>0$, for all $f\in L^{d/2+\kappa'}(D^\varepsilon(0,1))$, for all $F\in L^{d+\kappa}(D^\varepsilon(0,1))$, for all $u^\varepsilon$ weak solution to \eqref{sysoscueps} the bounds
\begin{equation*}
\|u^\varepsilon\|_{L^2(D^\varepsilon(0,1))}\leq 1,\quad\|f\|_{L^{d/2+\kappa'}(D^\varepsilon(0,1))}\leq \varepsilon_0,\quad\|F\|_{L^{d+\kappa}(D^\varepsilon(0,1))}\leq \varepsilon_0
\end{equation*}
imply
\begin{equation*}
[u^\varepsilon]_{C^{0,\mu}(\overline{D^\varepsilon(0,1/2)})}\leq C.
\end{equation*}
Notice that $C$ and $\varepsilon_0$ depend on $d$, $N$, $M_0$, on the modulus of continuity $\omega$ of $\nabla\psi$, $\lambda$, $\kappa$ and $\kappa'$.
\end{prop}

\begin{rem}
Of course, for all $f\in L^{d/2+\kappa'}(D^\varepsilon(0,1))$, for all $F\in L^{d+\kappa}(D^\varepsilon(0,1))$, for all $u^\varepsilon$ weak solution to \eqref{sysoscueps} such that
\begin{equation*}
\|u^\varepsilon\|_{L^2(D^\varepsilon(0,1))}<\infty,\quad\|f\|_{L^{d/2+\kappa'}(D^\varepsilon(0,1))}<\infty,\quad\|F\|_{L^{d+\kappa}(D^\varepsilon(0,1))}<\infty,
\end{equation*}
we have the estimate
\begin{equation}\label{bdaryholdest}
[u^\varepsilon]_{C^{0,\mu}(\overline{D^\varepsilon(0,1/2)})}\leq C\left\{\|u^\varepsilon\|_{L^2(D^\varepsilon(0,1))}+\|f\|_{L^{d/2+\kappa'}(D^\varepsilon(0,1))}+\|F\|_{L^{d+\kappa}(D^\varepsilon(0,1))}\right\},
\end{equation} 
with $C>0$ uniform in $\varepsilon$. In order to see that \eqref{bdaryholdest} boils down to Proposition \ref{propboundaryholder} we divide $u^\varepsilon$, $f$ and $F$ by the quantity 
\begin{equation*}
J:=\intbar_{D^\varepsilon(0,1)}|u^\varepsilon|^2+\frac{1}{\varepsilon_0}\|f\|_{L^{d/2+\kappa'}(D^\varepsilon(0,1))}+\frac{1}{\varepsilon_0}\|F\|_{L^{d+\kappa}(D^\varepsilon(0,1))}.
\end{equation*}
\end{rem}

\begin{rem}[without source terms]
If $f=F=0$, then for all $0<\mu<1$, 
\begin{equation*}
\|u^\varepsilon\|_{L^2(D^\varepsilon(0,1))}\leq 1\ \mbox{implies}\ [u^\varepsilon]_{C^{0,\mu}(\overline{D^\varepsilon(0,1/2)})}\leq C.
\end{equation*}
\end{rem}

\begin{rem}[control of the H\"older norm]
Let us emphasize three arguments showing that the estimate on the semi-norm of $u^\varepsilon$ is enough to control the H\"older norm of $u^\varepsilon$: namely, if
\begin{equation*}
\|u^\varepsilon\|_{L^2(D^\varepsilon(0,1))}\leq 1,
\end{equation*}
then
\begin{equation*}
\|u^\varepsilon\|_{C^{0,\mu}(\overline{D^\varepsilon(0,1/2)})}\leq C.
\end{equation*} 
Assume for simplicity that $f=F=0$. Fix $\varepsilon>0$. 

The first argument reads as follows. It is enough to show that we can control $|u^\varepsilon(x_\varepsilon)|$ at one point $x_\varepsilon\in D^\varepsilon(0,1/2)$:
\begin{equation}\label{findxepstche}
|u^\varepsilon(x_\varepsilon)|\leq C\left(\int_{D^\varepsilon(0,1)}|u^\varepsilon(x)|^2dx\right)^{1/2}.
\end{equation}
Indeed, for every $x\in D^\varepsilon(0,1/2)$, there is an $\tilde{x}\in D^{\varepsilon}(0,1/2)$ such that the segments $[x,\tilde{x}]$ and $[\tilde{x},x_\varepsilon]$ are contained in $D^\varepsilon(0,1/2)$ (of course, $\tilde{x}$ is introduced to compensate for the lack of convexity of $D^\varepsilon(0,1/2)$). Then
\begin{equation*}
\begin{aligned}
|u^\varepsilon(x)|&\leq |u^\varepsilon(x)-u^\varepsilon(\tilde{x})|+|u^\varepsilon(\tilde{x})-u^\varepsilon(x_\varepsilon)|+|u^\varepsilon(x_\varepsilon)|\\
&\leq C\left\{[u^\varepsilon]_{C^{0,\mu}(\overline{D^\varepsilon(0,1/2)})}+\left(\int_{D^\varepsilon(0,1)}|u^\varepsilon(x)|^2dx\right)^{1/2}\right\},
\end{aligned}
\end{equation*}
with $C>0$ uniform in $\varepsilon$. Note that \eqref{findxepstche} is a consequence of Tchebychev's inequality, since
\begin{equation*}
|\{u^\varepsilon>t\}|\leq \|u^\varepsilon\|_{L^2(D^\varepsilon(0,1))}/t^2\stackrel{t\rightarrow 0}{\longrightarrow} 0.
\end{equation*}

The second argument starts from the inequality
\begin{equation*}
|u(x)|\leq |u(y)|+|u(x)-u(y)|\leq |u(y)|+[u]_{C^{0,\mu}(D^\varepsilon(0,1/2))}|x-y|^\mu
\end{equation*}
for all $x,\ y\in D^\varepsilon(0,1/2)$. Integrating the latter with respect to $y$ yields
\begin{equation*}
\|u\|_{L^\infty(D^\varepsilon(0,1/2))}\leq C\left(\|u\|_{L^2(D^\varepsilon(0,1))}+[u]_{C^{0,\mu}(D^\varepsilon(0,1/2))}\right)\leq C\|u\|_{L^2(D^\varepsilon(0,1))}.
\end{equation*}

Finally, an alternative (simpler) argument uses directly the fact that $u^\varepsilon(x)=0$ for $x\in\Delta^\varepsilon(0,1/2)$: for $\bar{x}:=(x',\varepsilon\psi(x'/\varepsilon))\in\Delta^\varepsilon(0,1/2)$
\begin{equation*}
|u^\varepsilon(x)|=|u^\varepsilon(x)-u^\varepsilon(\bar{x})|\leq [u^\varepsilon]_{C^{0,\mu}(\overline{D^\varepsilon(0,1/2)})}\delta(x)^\mu\leq [u^\varepsilon]_{C^{0,\mu}(\overline{D^\varepsilon(0,1/2)})}\leq C\|u^\varepsilon\|_{L^2(D^\varepsilon(0,1))}.
\end{equation*}
\end{rem}

The following corollary is a generalization of Theorem 13 in \cite{alin} to oscillating boundaries. Similar estimates for the Green function associated to the Stokes operator in an oscillating domain have been showed in \cite[section 5.2]{David09}.

Let $\widetilde{G}^\varepsilon=\widetilde{G}^\varepsilon(x,\tilde{x})$ be the Green kernel associated to the operator $-\nabla\cdot A(x/\varepsilon)\nabla$ and the oscillating domain $D^\varepsilon(0,2)$. We recall that for all $\tilde{x}\in D^\varepsilon(0,2)$, $\widetilde{G}^\varepsilon(\cdot,\tilde{x})$ is a weak solution of
\begin{equation*}
\left\{
\begin{array}{rll}
-\nabla\cdot A(x/\varepsilon)\nabla \widetilde{G}^\varepsilon(x,\tilde{x})&=\delta(x-\tilde{x})\Idd_N,&x\in D^\varepsilon(0,2),\\
\widetilde{G}^\varepsilon(x,\tilde{x})&=0,&x\in \partial D^\varepsilon(0,2),
\end{array}
\right.
\end{equation*}
where here $\delta(\cdot)$ stands for the Dirac measure supported at the point $0$.

\begin{cor}\label{coralingreen}
There exists $C>0$, for all $\psi\in\mathcal C_{M_0}^{1,\omega}$, for all $A\in\mathcal A^{0,\nu_0}$, for all $\varepsilon>0$, for all $x,\ \tilde{x}\in D^\varepsilon(0,15/8)$,
\begin{align}
|\widetilde{G}^\varepsilon(x,\tilde{x})|\leq \frac{C}{|x-\tilde{x}|^{d-2}},\qquad \mbox{if}\quad d\geq 3\label{estgreend-23},\\
|\widetilde{G}^\varepsilon(x,\tilde{x})|\leq C(|\log|x-\tilde{x}||+1),\qquad\mbox{if}\quad d=2\label{estgreend-22}.
\end{align}
\end{cor}

\begin{proof}[Proof of Corollary \ref{coralingreen}]
The proof of this corollary follows the lines of \cite[Theorem 13]{alin}. The key is the boundary H\"older estimate of Proposition \ref{propboundaryholder} uniform in $\varepsilon$ for $f=F=0$.
\end{proof}

\subsection*{Proof of Proposition \ref{propboundaryholder}}

Let $\kappa,\ \kappa'>0$ be fixed for the whole proof. The proof follows the scheme of the three-step compactness method introduced by Avellaneda and Lin \cite{alin} in the context of homogenization:
\begin{enumerate}
\item improvement lemma,
\item iteration lemma,
\item proof of Proposition \ref{propboundaryholder}.
\end{enumerate}
We will redo the same type of three steps argument in the proof of the boundary Lipschitz estimates. For the latter, however, the proof is much more involved, as it requires the introduction of Dirichlet and boundary correctors, meant to correct the oscillations of the coefficients and of the boundary.

The estimate of the $C^{0,\mu}$ semi-norm of $u^\varepsilon$ is based on the characterization of H\"older spaces by Campanato (see \cite[Theorem 1.2]{giaquintaBook} or \cite[Theorem 5.5]{giaquinta12}):
\begin{equation*}
[u^\varepsilon]_{C^{0,\mu}(\overline{D^\varepsilon(0,1/2)})}^2\sim\sup_{x_0\in \overline{D^\varepsilon(0,1/2)},\ \rho>0}\rho^{-2\mu}\intbar_{D^\varepsilon(x_0,\rho)}|u^\varepsilon-(\overline{u^\varepsilon})_{x_0,\rho}|^2dx,
\end{equation*}
where $\sim$ means that the semi-norms on the left and right hand sides are equivalent.

\subsubsection*{First step}

\begin{lem}[improvement lemma]\label{lem1holder}
For all $0<\mu<\min\left(1-d/(d+\kappa),2-d/(d/2+\kappa')\right)$, there exist $\varepsilon_0>0,\ 0<\theta<1/8$, such that for all $\psi\in\mathcal C_{M_0}^{1,\omega}$, for all $0<\varepsilon<\varepsilon_0$, for all $A\in\mathcal A^{0,\nu_0}$, for all $f\in L^{d/2+\kappa'}(D^\varepsilon(0,1))$, for all $F\in L^{d+\kappa}(D^\varepsilon(0,1))$, for all $u^\varepsilon$ weak solution to 
\begin{equation}\label{elliptosc1/2}
\left\{\begin{array}{rll}
-\nabla\cdot A(x/\varepsilon)\nabla u^\varepsilon&=f+\nabla\cdot F
,&x\in D^\varepsilon(0,1),\\
u^\varepsilon&=0,&x\in\Delta^\varepsilon(0,1),
\end{array}
\right. 
\end{equation}
if
\begin{equation*}
\intbar_{D^\varepsilon(0,1)}|u^\varepsilon|^2\leq 1,\ \left\|f\right\|_{L^{d/2+\kappa'}(D^\varepsilon(0,1))}\leq \varepsilon_0,\ \left\|F\right\|_{L^{d+\kappa}(D^\varepsilon(0,1))}\leq \varepsilon_0
\end{equation*}
then 
\begin{equation}\label{estlem1holder}
\intbar_{D^\varepsilon(0,\theta)}|u^\varepsilon|^2\leq\theta^{2\mu}.
\end{equation}
\end{lem}

\begin{rem}
It is a classical fact (see for example \cite[Exercise 5.10]{giaquinta12}) that
\begin{equation*}
\int_{D^\varepsilon(0,\theta)}|u^\varepsilon-(\overline{u^\varepsilon})_{0,\theta}|^2=\inf_{\bar{u}\in\mathbb R^N}\int_{D^\varepsilon(0,\theta)}|u^\varepsilon-\bar{u}|^2.
\end{equation*}
This implies that
\begin{equation*}
\intbar_{D^\varepsilon(0,\theta)}|u^\varepsilon-(\overline{u^\varepsilon})_{0,\theta}|^2\leq\intbar_{D^\varepsilon(0,\theta)}|u^\varepsilon|^2\leq\theta^{2\mu}.
\end{equation*}
\end{rem}

\begin{rem}
The reason why we can prove the improved bound \eqref{estlem1holder} lies in the fact that $u^{\varepsilon}$ vanishes on $\Delta^{\varepsilon}(0,1)$. The iteration argument of the second step of the proof is easier to carry out with this bound on $u^{\varepsilon}$, rather than the bound on $u^{\varepsilon}-(\overline{u^\varepsilon})_{0,\theta}$. Indeed, the former vanishes on $\Delta^{\varepsilon}(0,1)$, which is not true for the latter.
\end{rem}

\begin{proof}[Proof of Lemma \ref{lem1holder}]
Let $0<\theta<1/8$, $0<\mu'<\min\left(1-d/(d+\kappa),2-d/(d/2+\kappa')\right)$, $A^0$ be any constant coefficients matrix satisfying \eqref{elliptA} and $u^0\in W^{1,2}(D^0(0,1/4))$ be a weak solution to
\begin{equation}\label{eqlimu0}
\left\{
\begin{array}{rll}
-\nabla\cdot A^0\nabla u^0&=0,&x\in D^0(0,1/4),\\
u^0&=0,&x\in\Delta^0(0,1/4),
\end{array}
\right.
\end{equation}
such that 
\begin{equation*}
\intbar_{D^0(0,1/4)}|u^0|^2\leq 4^{2d+1}. 
\end{equation*}
By classical Schauder regularity (see Theorem \ref{theoclassholder}),
\begin{equation}\label{estu^0holderbis}
[u^0]_{C^{0,\mu'}(\overline{D^0(0,1/8)})}\leq C\left(\intbar_{D^0(0,1/4)}|u^0|^2\right)^{1/2}.
\end{equation}
We have for $\theta$ small, using $u^0(0)=0$,
\begin{align*}
\intbar_{D^0(0,\theta)}|u^0|^2\leq C\sup_{x\in D^0(0,\theta)}|u^0(x)|^2&\leq C\sup_{x\in D^0(0,\theta)}|u^0(x)-u^0(0)|^2\\
&\leq C\theta^{2\mu'}[u^0]_{C^{0,\mu'}(\overline{D^0(0,1/8)})}^2\leq C\theta^{2\mu'}\intbar_{D^0(0,1/4)}|u^0|^2,
\end{align*}
so that
\begin{equation}\label{estu^0holder}
\intbar_{D^0(0,\theta)}|u^0|^2dx\leq C4^{2d+1}\theta^{2\mu'}
\end{equation}
where $C$ is uniform in $\theta$. Fix $0<\mu<\mu'<1$, and choose $0<\theta<1/8$ such that $\theta^{2\mu}>C4^{2d+1}\theta^{2\mu'}$. 

The goal is now to show (uniformly in $\psi$ and $A$ in their respective classes) that for this $\theta$, there exists $\varepsilon_0$ such that for all $0<\varepsilon<\varepsilon_0$,
\begin{equation*}
\intbar_{D^\varepsilon(0,\theta)}|u^\varepsilon|^2\leq\theta^{2\mu}.
\end{equation*}
Let us show this by contradiction. Assume that there exist a sequence $\varepsilon_k\stackrel{k\rightarrow\infty}{\longrightarrow}0$, $\psi_k\in\mathcal C_{M_0}^{1,\omega}$, $A_k\in\mathcal A^{0,\nu_0}$, $f^{\varepsilon_k}\in L^{d/2+\kappa'}(D^{\varepsilon_k}(0,1/2))$, $F^{\varepsilon_k}\in L^{d+\kappa}(D^{\varepsilon_k}(0,1/2))$ and $u^{\varepsilon_k}$ weak solution to \eqref{elliptosc1/2} such that
\begin{equation*}
\intbar_{D^{\varepsilon_k}(0,1)}|u^{\varepsilon_k}|^2\leq 1,\ \|f^{\varepsilon_k}\|_{L^{d/2+\kappa'}(D^{\varepsilon_k}(0,1))}\leq \varepsilon_k,\ \|F^{\varepsilon_k}\|_{L^{d+\kappa}(D^{\varepsilon_k}(0,1))}\leq \varepsilon_k
\end{equation*}
and
\begin{equation}\label{contraholder}
\intbar_{D^{\varepsilon_k}(0,\theta)}|u^{\varepsilon_k}|^2>\theta^{2\mu}.
\end{equation}
Thanks to the uniform control of $u^{\varepsilon_k}$ in $L^2$, we can rely on weak compactness in $W^{1,2}$ using Cacciopoli's inequality, and on strong compactness in $L^2$ using Rellich's compact embedding. Doing so we have to be careful, since the $L^2$ bound on $u^\varepsilon$ is on the domain with oscillating boundary $D^{\varepsilon_k}(0,1)$. Since $u^{\varepsilon_k}=0$ on $\Delta^{\varepsilon_k}(0,1)$, we control by Cacciopoli's inequality
\begin{equation*}
\|\nabla u^{\varepsilon_k}\|_{L^2(D^{\varepsilon_k}(0,1/2))}\leq C,
\end{equation*}
with $C>0$ uniform in $\varepsilon_k$. In order to deal with the oscillating boundary, we extend $f^{\varepsilon_k}$, $F^{\varepsilon_k}$ and $u^{\varepsilon_k}$ by $0$ below the oscillating boundary $x_d=\varepsilon_k\psi_k(x'/\varepsilon_k)$. We then have for $k$ large
\begin{equation*}
\|u^{\varepsilon_k}\|_{L^2(D^{-1}(0,1/4))}\leq C\quad\mbox{and}\quad\|\nabla u^{\varepsilon_k}\|_{L^2(D^{-1}(0,1/4))}\leq C.
\end{equation*}
Therefore, up to extracting subsequences,
\begin{equation*}
\overline{A_k}\longrightarrow A^0,
\end{equation*}
with $\overline{A_k}$ defined by \eqref{defA0} with $A$ replaced by $A_k$, and $A^0$ is a constant coefficients matrix satisfying the ellipticity condition \eqref{elliptA}, and there exists $u^0\in H^1(D^{-1}(0,1/4))$,
\begin{equation}\label{weakCVholderest}
\begin{aligned}
u^{\varepsilon_k}\rightharpoonup u^0\quad\mbox{weakly in}\quad L^2(D^{-1}(0,1/4)),\\
\nabla u^{\varepsilon_k}\rightharpoonup \nabla u^0\quad\mbox{weakly in}\quad L^2(D^{-1}(0,1/4)),
\end{aligned}
\end{equation}
and by Rellich's compact embedding theorem,
\begin{equation}\label{strongCVholderest}
u^{\varepsilon_k}\longrightarrow u^0\quad\mbox{strongly in}\quad L^2(D^{-1}(0,1/4)).
\end{equation}
Furthermore,
\begin{equation}\label{weakCVsource}
\begin{aligned}
f^{\varepsilon_k}\longrightarrow 0\quad\mbox{strongly in}\quad L^{d/2+\kappa'}(D^{-1}(0,1/4)),\\
F^{\varepsilon_k}\longrightarrow 0\quad\mbox{strongly in}\quad L^{d+\kappa}(D^{-1}(0,1/4)).
\end{aligned}
\end{equation}
We have to show that $u^0$ solves \eqref{eqlimu0}. That $u^0$ is a weak solution to
\begin{equation}\label{u0limweaksol}
-\nabla\cdot A^0\nabla u^0=0\quad\mbox{in}\quad D^0(0,1/4)
\end{equation}
follows from Theorem \ref{theoweakcvhomo}: for $\iota>0$, taking $\Omega_\iota=(-1/4,1/4)^{d-1}\times(\iota,1/4)$ yields 
\begin{equation*}
-\nabla\cdot A^0\nabla u^0=0\quad\mbox{in}\quad \Omega_\iota,
\end{equation*}
thus \eqref{u0limweaksol} holds. It remains to see that $u^0=0$ on $\Delta^0(0,1/4)$. Take a test function $\varphi\in C^\infty_c((-1/4,1/4)^{d-1}\times(-1,0))$. For all $k$, $\supp(\varphi)\subset D^{\varepsilon_k}_-$ since $0\leq \psi_k$. Then, since $u^{\varepsilon_k}=0$ on $D^{\varepsilon_k}_-$ and by \eqref{weakCVholderest}, we get for $k$ sufficiently large, 
\begin{equation*}
0=\int_{D^{-1}(0,1/4)} u^{\varepsilon_k}\varphi\stackrel{k\rightarrow\infty}{\longrightarrow}\int_{D^{-1}(0,1/4)} u^0\varphi.
\end{equation*}
Consequently $u^0=0$ in $\mathcal D'((-1/4,1/4)^{d-1}\times(-1,0))$, thus $u^0|_{(-1/4,1/4)^{d-1}\times(-1,0)}=0$, and by the trace theorem $u^0=0\in W^{1/2,2}(\Delta^0(0,1/4))$. It follows that $u^0$ satisfies the estimate \eqref{estu^0holderbis}. Moreover,  
\begin{equation*}
\intbar_{D^0(0,1/4)}|u^0|^2\leq 4\intbar_{D^0(0,1/4)}|u^{\varepsilon_k}|^2\leq 4^{2d+1}\intbar_{D^0(0,1)}|u^{\varepsilon_k}|^2\leq 4^{2d+1}, 
\end{equation*}
so that \eqref{estu^0holder} holds.

The last step of the proof consists in passing to the limit in \eqref{contraholder} in order to get a contradiction. Since $|D^{\varepsilon_k}(0,\theta)|=|D^0(0,\theta)|$, we have
\begin{equation}\label{decompCVmeanueps^2}
\begin{aligned}
\intbar_{D^{\varepsilon_k}(0,\theta)}|u^{\varepsilon_k}|^2
=&\frac{1}{|D^0(0,\theta)|}\int_{D^{\varepsilon_k}(0,\theta)}|u^{\varepsilon_k}|^2\\
=&\frac{1}{|D^0(0,\theta)|}\left\{\int_{D^0(0,\theta)}|u^{\varepsilon_k}|^2+\int_{D^{\varepsilon_k}(0,\theta)\setminus D^0(0,\theta)}|u^{\varepsilon_k}|^2-\int_{D^0(0,\theta)\setminus D^{\varepsilon_k}(0,\theta)}|u^{\varepsilon_k}|^2\right\}.
\end{aligned}
\end{equation}
By the strong convergence \eqref{strongCVholderest} of $u^{\varepsilon_k}$
\begin{equation*}
\int_{D^0(0,\theta)}|u^{\varepsilon_k}|^2\stackrel{k\rightarrow\infty}{\longrightarrow}\int_{D^0(0,\theta)}|u^0|^2,
\end{equation*}
and for $k$ large
\begin{align*}
\int_{D^{\varepsilon_k}(0,\theta)\setminus D^0(0,\theta)}|u^{\varepsilon_k}|^2&=\int_{D^{\varepsilon_k}(0,\theta)\setminus D^0(0,\theta)}|u^0|^2+\int_{D^{\varepsilon_k}(0,\theta)\setminus D^0(0,\theta)}(|u^{\varepsilon_k}|^2-|u^0|^2)\\
&\leq 2\int_{D^{\varepsilon_k}(0,\theta)\setminus D^0(0,\theta)}|u^0|^2+2\int_{D^{\varepsilon_k}(0,\theta)\setminus D^0(0,\theta)}|u^{\varepsilon_k}-u^0|^2\\
&\leq 2\int_{D^{-1}(0,1/4)}|u^0|^21_{D^{\varepsilon_k}(0,\theta)\setminus D^0(0,\theta)}+2\int_{D^{-1}(0,1/4)}|u^{\varepsilon_k}-u^0|^2\stackrel{k\rightarrow\infty}{\longrightarrow}0,
\end{align*}
by dominated convergence on the one hand and the strong convergence \eqref{strongCVholderest} on the other hand. In a similar fashion, the third term in the right hand side of \eqref{decompCVmeanueps^2} is shown to tend to $0$ when $k\rightarrow\infty$. Finally,
\begin{equation*}
C4^{2d+1}\theta^{2\mu'}<\theta^{2\mu}\leq\intbar_{D^{\varepsilon_k}(0,\theta)}|u^{\varepsilon_k}|^2\stackrel{k\rightarrow\infty}{\longrightarrow}\intbar_{D^0(0,\theta)}|u^0|^2
\end{equation*}
which is in contradiction with \eqref{estu^0holder}.
\end{proof}

\subsubsection*{Second step}

\begin{lem}[iteration lemma]\label{lem2holder}
For $0<\mu<\min\left(1-d/(d+\kappa),2-d/(d/2+\kappa')\right)$ fixed, let $\varepsilon_0>0$ and $\theta>0$ as given by Lemma \ref{lem1holder}. For all $k\in\mathbb N$, $k\geq 1$, for all $\varepsilon<\theta^{k-1}\varepsilon_0$, for all $\psi\in\mathcal C_{M_0}^{1,\omega}$, for all $A\in\mathcal A^{0,\nu_0}$, for all $f\in L^{d/2+\kappa'}(D^\varepsilon(0,1))$, for all $F\in L^{d+\kappa}(D^\varepsilon(0,1))$, for all $u^\varepsilon$ weak solution to \eqref{elliptosc1/2}, if
\begin{equation*}
\intbar_{D^\varepsilon(0,1)}|u^\varepsilon|^2\leq 1,\ \left\|f\right\|_{L^{d/2+\kappa'}(D^\varepsilon(0,1))}\leq \varepsilon_0,\ \left\|F\right\|_{L^{d+\kappa}(D^\varepsilon(0,1))}\leq \varepsilon_0
\end{equation*}
then
\begin{equation*}
\intbar_{D^\varepsilon(0,\theta^k)}|u^\varepsilon|^2\leq\theta^{2k\mu}.
\end{equation*}
\end{lem}

\begin{proof}[Proof of Lemma \ref{lem2holder}]
Let us do the proof by induction. The case $k=1$ is Lemma \ref{lem1holder}. Let $k\geq 1$ and assume that for all $\varepsilon<\theta^{k-1}\varepsilon_0$, for all $\psi\in\mathcal C_{M_0}^{1,\omega}$, for all $A\in\mathcal A^{0,\nu_0}$, for all $u^\varepsilon$ weak solution to \eqref{elliptosc1/2}, if
\begin{equation*}
\intbar_{D^\varepsilon(0,1)}|u^\varepsilon|^2\leq 1,\ \left\|f\right\|_{L^{d/2+\kappa'}(D^\varepsilon(0,1))}\leq \varepsilon_0,\ \left\|F\right\|_{L^{d+\kappa}(D^\varepsilon(0,1))}\leq \varepsilon_0
\end{equation*}
then
\begin{equation}\label{estlem2holderHR}
\intbar_{D^\varepsilon(0,\theta^k)}|u^\varepsilon|^2\leq\theta^{2k\mu}.
\end{equation}
For all $x\in D^{\varepsilon/\theta^k}(0,1)$, let $U^\varepsilon$ be defined by
\begin{equation*}
U^\varepsilon(x):=u^\varepsilon(\theta^kx)/\theta^{k\mu}.
\end{equation*}
Since
\begin{equation*}
\nabla\cdot\left(A(\theta^kx/\varepsilon)\nabla u^\varepsilon(\theta^kx)\right)=\theta^{2k}\left[\nabla\cdot A(\cdot/\varepsilon)\nabla u^\varepsilon\right](\theta^kx),
\end{equation*}
$U^\varepsilon$ solves
\begin{equation*}
\left\{\begin{array}{rll}
-\nabla\cdot A(\theta^k x/\varepsilon)\nabla U^\varepsilon&=\frac{1}{\theta^{k\mu}}\left[\theta^{2k}f(\theta^kx)+\theta^k\nabla\cdot(F(\theta^kx))\right],&x\in D^{\varepsilon/\theta^k}(0,1)\\
U^\varepsilon&=0,&x\in \Delta^{\varepsilon/\theta^k}(0,1)
\end{array}
\right.
\end{equation*}
and from \eqref{estlem2holderHR}
\begin{equation*}
\intbar_{D^{\varepsilon/\theta^k}(0,1)}|U^\varepsilon(x)|^2\leq 1.
\end{equation*}
Notice that since $\mu<\min\left(1-d/(d+\kappa),2-d/(d/2+\kappa')\right)$,
\begin{equation*}
\|\theta^{2k-k\mu}f(\theta^k x)\|_{L^{d/2+\kappa'}(D^{\varepsilon/\theta^k}(0,1))}=\theta^{k(2-d/(d/2+\kappa')-\mu)}\|f\|_{L^{d/2+\kappa'}(D^{\varepsilon}(0,1))}\leq \|f\|_{L^{d/2+\kappa'}(D^{\varepsilon}(0,1))}\leq \varepsilon_0
\end{equation*}
and
\begin{equation*}
\|\theta^{k-k\mu}F(\theta^k x)\|_{L^{d+\kappa}(D^{\varepsilon/\theta^k}(0,1))}=\theta^{k(1-d/(d+\kappa)-\mu)}\|F\|_{L^{d+\kappa}(D^{\varepsilon}(0,1))}\leq \|F\|_{L^{d+\kappa}(D^{\varepsilon}(0,1))}\leq \varepsilon_0.
\end{equation*}
Therefore, applying Lemma \ref{lem1holder}, for $\varepsilon/\theta^k<\varepsilon_0$,
\begin{equation*}
\intbar_{D^{\varepsilon/\theta^k}(0,\theta)}|U^\varepsilon(x)|^2\leq \theta^{2\mu},
\end{equation*}
which boils down to
\begin{equation*}
\intbar_{D^{\varepsilon}(0,\theta^{k+1})}|u^\varepsilon(x)|^2\leq \theta^{2(k+1)\mu}
\end{equation*}
and concludes the proof.
\end{proof}

\subsubsection*{Third step: proof of Proposition \ref{propboundaryholder}}

If $\varepsilon\geq\varepsilon_0$, then Proposition \ref{propboundaryholder} follows from the classical H\"older estimates. Let $\varepsilon<\varepsilon_0$, $A\in\mathcal A^{0,\nu_0}$ and $\psi\in\mathcal C_{M_0}^{1,\omega}$ be fixed for the rest of the proof. There exists a unique integer $k\geq 1$ such that $\theta^k\leq\varepsilon/\varepsilon_0<\theta^{k-1}$. Let $x_0:=(0,x_{0,d})\in D^{\varepsilon}(0,1/2)$. Recall that $\delta(x_0)=x_{0,d}-\varepsilon\psi(0)$. We distinguish between two cases: either $\delta(x_0)>\theta^{k+1}/2$ (case ``far'' from the boundary) or $\delta(x_0)\leq\theta^{k+1}/2$ (case ``close'' to the boundary).

Assume that $\delta(x_0)>\theta^{k+1}/2$. In that case, the idea is to rely on the rescaled interior H\"older estimate \eqref{intestrescaledbis}
\begin{multline*}
[u^\varepsilon]_{C^{0,\mu}(B(x_0,\delta(x_0)/4))}\leq C\bigl\{(\delta(x_0))^{-d/2-\mu}\|u^\varepsilon\|_{L^2(B(x_0,\delta(x_0)/2))}\bigr.\\
\bigl.+(\delta(x_0))^{2-\mu-d/(d/2+\kappa')}\|f\|_{L^{d/2+\kappa'}(B(x_0,\delta(x_0)/2))}+(\delta(x_0))^{1-\mu-d/(d+\kappa)}\|F\|_{L^{d+\kappa}(B(x_0,\delta(x_0)/2))}\bigr\}.
\end{multline*}
There exists $0\leq l\leq k$ such that $\theta^{l+1}/2<\delta(x_0)\leq\theta^l/2$. If $l=0$, then
\begin{equation*}
\|u^\varepsilon\|_{L^2(B(x_0,\delta(x_0)/2))}\leq \|u^\varepsilon\|_{L^2(D^\varepsilon(0,1))}\leq |D^\varepsilon(0,1)|^{1/2}<C(2\delta(x_0)/\theta)^{d/2+\mu},
\end{equation*}
where $C>0$ only depends on the dimension $d$. If $l\geq 1$, then by Lemma \ref{lem2holder}
\begin{equation*}
\|u^\varepsilon\|_{L^2(B(x_0,\delta(x_0)/2))}\leq \|u^\varepsilon\|_{L^2(D^\varepsilon(0,\theta^l))}\leq C\theta^{(d/2+\mu)l}<C(2\delta(x_0)/\theta)^{d/2+\mu},
\end{equation*}
where $C>0$ only depends on $d$ and in particular not on $\theta$. Moreover, for all $l\in\mathbb N$, since 
\begin{equation*}
1-\mu-d/(d+\kappa)>0,\quad 2-\mu-d/(d/2+\kappa')>0,
\end{equation*}
we get
\begin{equation*}
\begin{aligned}
&(\delta(x_0))^{2-\mu-2d/(d+2\kappa')}\|f\|_{L^{d/2+\kappa'}(B(x_0,\delta(x_0)/2))}\leq \|f\|_{L^{d/2+\kappa'}(D^\varepsilon(0,1))}\leq \varepsilon_0,\\
&(\delta(x_0))^{1-\mu-d/(d+\kappa)}\|F\|_{L^{d+\kappa}(B(x_0,\delta(x_0)/2))}\leq \|F\|_{L^{d+\kappa}(D^\varepsilon(0,1))}\leq \varepsilon_0.
\end{aligned}
\end{equation*}
This boils down to
\begin{equation*}
[u^\varepsilon]_{C^{0,\mu}(B(x_0,\delta(x_0)/4))}\leq C
\end{equation*}
with $C>0$ uniform in $\varepsilon$, $k$, but which may depend on $d$, $\mu$ and $\theta$.

Assume that $\delta(x_0)\leq\theta^{k+1}/2\leq \theta\varepsilon/(2\varepsilon_0)$. In this case we rely on classical estimates near the boundary, which is relevant since we are at the small scale $O(\varepsilon)$. We proceed thanks to a blow-up argument, i.e. we consider $v^\varepsilon$ defined for $y\in D^1(0,1/\varepsilon)$ by
\begin{equation*}
U^\varepsilon(y):=\frac{1}{\varepsilon}u^\varepsilon(\varepsilon y).
\end{equation*}
It solves
\begin{equation*}
\left\{\begin{array}{rll}
-\nabla\cdot A(y)\nabla U^\varepsilon&=\varepsilon f(\varepsilon y)+\nabla\cdot(F(\varepsilon y)),&y\in D^1(0,1/\varepsilon),\\
U^\varepsilon&=0,&y\in \Delta^1(0,1/\varepsilon).
\end{array}
\right.
\end{equation*}
Applying the classical estimate \eqref{holderclassestu} and rescaling, we get
\begin{align*}
&\varepsilon^{-1+\mu}[u^\varepsilon(x)]_{C^{0,\mu}(D^\varepsilon(0,\varepsilon/(2\varepsilon_0)))}=[U^\varepsilon(y)]_{C^{0,\mu}(D^1(0,1/(2\varepsilon_0)))}\\
&\leq C\left\{\|U^\varepsilon(y)\|_{L^2(D^1(0,1/\varepsilon_0))}+\|\varepsilon f(\varepsilon\cdot)\|_{L^{d/2+\kappa'}(D^1(0,1/\varepsilon_0))}+\|F(\varepsilon\cdot)\|_{L^{d+\kappa}(D^1(0,1/\varepsilon_0))}\right\}\\
&=C\left\{\varepsilon^{-1-d/2}\|u^\varepsilon(x)\|_{L^2(D^\varepsilon(0,\varepsilon/\varepsilon_0))}+\varepsilon^{1-d/(d/2+\kappa')}\|f\|_{L^{d/2+\kappa'}(D^\varepsilon(0,\varepsilon/\varepsilon_0))}+\varepsilon^{-d/(d+\kappa)}\|F\|_{L^{d+\kappa}(D^\varepsilon(0,\varepsilon/\varepsilon_0))}\right\}.
\end{align*}
Therefore,
\begin{align*}
[u^\varepsilon]_{C^{0,\mu}(D^\varepsilon(0,\varepsilon/(2\varepsilon_0)))}&\leq C\left\{\varepsilon^{-d/2-\mu}\|u^\varepsilon\|_{L^2(D^\varepsilon(0,\varepsilon/\varepsilon_0))}\right.\\
&\left.\qquad+\varepsilon^{2-\mu-d/(d/2+\kappa')}\|f\|_{L^{d/2+\kappa'}(D^\varepsilon(0,\varepsilon/\varepsilon_0))}+\varepsilon^{1-\mu-d/(d+\kappa)}\|F\|_{L^{d+\kappa}(D^\varepsilon(0,\varepsilon/\varepsilon_0))}\right\}\\
&\leq C\left\{1/(\varepsilon_0\theta^k)^{d/2+\mu}\|u^\varepsilon\|_{L^2(D^\varepsilon(0,\theta^{k-1}))}+\|f\|_{L^{d/2+\kappa'}(D^\varepsilon(0,1))}+\|F\|_{L^{d+\kappa}(D^\varepsilon(0,1))}\right\}\\
&\leq C\left\{1/(\varepsilon_0\theta^k)^{d/2+\mu}\theta^{(k-1)(d/2+\mu)}+2\varepsilon_0\right\}\leq C_{\varepsilon_0,\theta}.
\end{align*}

One case is left. We have to deal with arbitrary points $x_0\in D^\varepsilon(0,1/2)$. Remember that $\varepsilon$, $\psi$ and $A$ are fixed. We consider $\tilde{\psi}$ defined for $y'\in\mathbb R^{d-1}$ by
\begin{equation*}
\tilde{\psi}(y'):=\psi(y'/2+x_0'/\varepsilon),
\end{equation*}
and for $y\in\mathbb R^d$
\begin{equation*}
\tilde{A}(y',y_d):=A(y'/2+x_0'/\varepsilon,y_d).
\end{equation*}
Notice that $\tilde{\psi}\in \mathcal C_{M_0}^{1}$ and $\tilde{A}\in\mathcal A^{0,\nu_0}$. For all $|x'|<1$, $2\varepsilon\tilde{\psi}(x'/(2\varepsilon))<x_d<2\varepsilon\tilde{\psi}(x'/(2\varepsilon))+1$, let $\tilde{u}^\varepsilon$ be defined by
\begin{equation*}
\tilde{u}^\varepsilon(x',x_d):=u^\varepsilon(x'/2+x_0',x_d),
\end{equation*}
$\tilde{f}$ be defined by
\begin{equation*}
\tilde{f}(x',x_d):=f(x'/2+x_0',x_d),
\end{equation*}
and $\tilde{F}$ be defined by
\begin{equation*}
\tilde{F}(x',x_d):=F(x'/2+x_0',x_d).
\end{equation*}
Let 
\begin{equation*}
J:=\intbar_{D^{2\varepsilon}(0,1)}|\tilde{u}^\varepsilon|^2+\frac{1}{\varepsilon_0}\|\tilde{f}\|_{L^{d/2+\kappa'}(D^{2\varepsilon}_{\tilde{\psi}}(0,1))}+\frac{1}{\varepsilon_0}\|\tilde{F}\|_{L^{d+\kappa}(D^{2\varepsilon}_{\tilde{\psi}}(0,1))}\leq\varepsilon_0.
\end{equation*}
We have that 
\begin{equation*}
\|\tilde{u}^\varepsilon\|_{L^2(D^{2\varepsilon}_{\tilde{\psi}}(0,1))}\leq J,\ \|\tilde{f}\|_{L^{d/2+\kappa'}(D^{2\varepsilon}_{\tilde{\psi}}(0,1))}\leq \varepsilon_0J,\ \|\tilde{F}\|_{L^{d+\kappa}(D^{2\varepsilon}_{\tilde{\psi}}(0,1))}\leq \varepsilon_0J
\end{equation*}
and $\tilde{u}^\varepsilon$ solves
\begin{equation*}
\left\{\begin{array}{rll}
-\nabla\cdot \tilde{A}(x/(2\varepsilon))\nabla\tilde{u}^\varepsilon&=\tilde{f}+\nabla\cdot\tilde{F},&x\in D^{2\varepsilon}_{\tilde{\psi}}(0,1),\\
\tilde{u}^\varepsilon&=0,&x\in \Delta^{2\varepsilon}_{\tilde{\psi}}(0,1).
\end{array}
\right.
\end{equation*}
We can thus apply Lemma \ref{lem2holder} and argue exactly as above for $x_0=(0,x_{0,d})$. 

\section{Boundary corrector}
\label{secbdarycor}

In the following lemma we build a corrector term for the oscillating boundary. This term is crucial for the construction of an expansion for $u^\varepsilon$ in the proof of the uniform Lipschitz estimate.

For fixed $M_0>0$, let $\vartheta'\in C^\infty_c(\mathbb R^{d-1})$ (resp. $\vartheta_{d}\in C^\infty_c(\mathbb R)$) a cut-off function compactly supported in $(-3/2,3/2)^{d-1}$ (resp. in $(-3M_0/2,3M_0/2)$), identically equal to $1$ on $(-1,1)^{d-1}$ (resp. on $(-M_0,M_0)$). We define the cut-off function $\Theta\in C^\infty_c(\mathbb R^d)$ by for all $x'\in\mathbb R^{d-1}$, $y_d\in\mathbb R$,
\begin{equation*}
\Theta(x',y_d):=\vartheta'(x')\vartheta_d(y_d).
\end{equation*}
Notice that $\Theta$ is compactly supported in $(-3/2,3/2)^{d-1}\times(-3M_0/2,3M_0/2)$, identically equal to $1$ on $(-1,1)^{d-1}\times(-M_0,M_0)$ and that for all $x',\ \hat{x}'\in\mathbb R^{d-1}$, $y_d\in\mathbb R$,
\begin{equation}\label{condTheta}
|x'|,\ |\hat{x}'|\leq 1\qquad\mbox{implies}\qquad\Theta(x',y_d)=\Theta(\hat{x}',y_d).
\end{equation}
This property of $\Theta$ is just meant to give a nice form to the expansion of $u^\varepsilon$ (see Lemma \ref{lem2} below).

\begin{lem}[boundary corrector]\label{bdarycorlem}
For all $1/2<\tau<1$, there exists $C>0$ such that for all $\psi\in\mathcal C_{M_0}^{1,\omega}$, for all $A\in\mathcal A^{0,\nu_0}$, for all $0<\varepsilon<1$, the unique weak solution $v^\varepsilon\in W^{1,2}(D^\varepsilon(0,2))$ of
\begin{equation}\label{sysbdarycorr}
\left\{
\begin{array}{rll}
-\nabla\cdot A(x/\varepsilon)\nabla v^\varepsilon&=\nabla\cdot A(x/\varepsilon)\nabla(\psi(x'/\varepsilon)\Theta(x',x_d/\varepsilon)),&x\in D^\varepsilon(0,2),\\
v^\varepsilon&=0,&x\in\partial D^\varepsilon(0,2),
\end{array}
\right. 
\end{equation}
satisfies the following estimate: for all $x\in D^\varepsilon(0,3/2)$,
\begin{equation}\label{estbdarycorlem}
|v^\varepsilon(x)|\leq \frac{C_0\delta(x)^\tau}{\varepsilon^\tau},
\end{equation}
where $\delta(x):=x_d-\varepsilon\psi(x'/\varepsilon)$.
\end{lem}

\begin{rem}
Notice that $v^\varepsilon$ depends on $\psi$ even if the dependence is not explicitly written. However, the constant $C_0>0$ in the above inequality is uniform for $\psi\in\mathcal C_{M_0}^{1,\omega}$.
\end{rem}

The proof of Lemma \ref{bdarycorlem} follows from the representation of $v^\varepsilon$ thanks to Green's kernel $\widetilde{G}^\varepsilon=\widetilde{G}^\varepsilon(x,\tilde{x})$ associated to the operator $-\nabla\cdot A(x/\varepsilon)\nabla$ and to the oscillating domain $D^\varepsilon(0,2)$. An estimate of $\nabla_2\widetilde{G}^\varepsilon$ is the key.

\begin{lem}[estimate of Green's kernel, $d\geq 3$]\label{lemestgreend=3}
For all $0<\tau<1$, there exists $C>0$ such that for all $\psi\in\mathcal C_{M_0}^{1,\omega}$, for all $A\in\mathcal A^{0,\nu_0}$, for all $0<\varepsilon<1$:
\begin{enumerate}[label=(\arabic*)]
\item for all $x,\ \tilde{x}\in D^\varepsilon(0,7/4)$,
\begin{align}
|\widetilde{G}^\varepsilon(x,\tilde{x})|\leq \frac{C\delta(x)^\tau}{|x-\tilde{x}|^{d-2+\tau}},\label{est3gepsa}\\
|\widetilde{G}^\varepsilon(x,\tilde{x})|\leq \frac{C\delta(x)^\tau\delta(\tilde{x})^\tau}{|x-\tilde{x}|^{d-2+2\tau}},\label{est3gepsb}
\end{align}
\item for all $x,\ \tilde{x}\in D^\varepsilon(0,3/2)$,
\begin{align}
|\nabla_2\widetilde{G}^\varepsilon(x,\tilde{x})|\leq \frac{C\delta(x)^\tau}{|x-\tilde{x}|^{d-1+\tau}},\quad\mbox{for}\quad |x-\tilde{x}|\leq\varepsilon,\label{est3nablagepsa}\\
|\nabla_2\widetilde{G}^\varepsilon(x,\tilde{x})|\leq \frac{C\delta(x)^\tau\delta(\tilde{x})^\tau}{\varepsilon|x-\tilde{x}|^{d-2+2\tau}},\quad\mbox{for}\quad |x-\tilde{x}|>\varepsilon.\label{est3nablagepsb}
\end{align}
\end{enumerate}
\end{lem}

\begin{rem}
Although the gradient estimate holds for all $x,\ \tilde{x}\in D^\varepsilon(0,3/2)$, thanks to the cut-off $\Theta(x',x_d/\varepsilon)$ we essentially need it in a small layer of size $O(\varepsilon)$ located near the oscillating boundary.
\end{rem}

\begin{rem}
Notice that the decay at large scales of $\nabla_2\widetilde{G}^\varepsilon$ is not better than the decay of $\widetilde{G}^\varepsilon$. This comes from the fact that the bound \eqref{est3nablagepsb} is obtained from \eqref{est3gepsb} by applying Lipschitz estimates at small scale $O(\varepsilon)$.
\end{rem}

\begin{rem}
Notice that the estimates \eqref{est3gepsa} and \eqref{est3gepsb} only hold for all $x,\ \tilde{x}\in D^\varepsilon(0,7/4)$, not all $x,\ \tilde{x}\in D^\varepsilon(0,2)$. This comes from the fact that the boundary $\partial D^\varepsilon(0,2)$ is not smooth at the points lying on $\partial D^\varepsilon(0,2)\cap \Delta^\varepsilon\cap\{|x'|=2\}$.
\end{rem}

\begin{lem}[estimate of Green's kernel, $d=2$]\label{lemestgreend=2}
For all $0<\tau<1$, there exists $C>0$ such that for all $\psi\in\mathcal C_{M_0}^{1,\omega}$, for all $A\in\mathcal A^{0,\nu_0}$, for all $0<\varepsilon<1$:
\begin{enumerate}[label=(\arabic*)]
\item for all $x,\ \tilde{x}\in D^\varepsilon(0,7/4)$,
\begin{align}
|\widetilde{G}^\varepsilon(x,\tilde{x})|\leq C(|\log|x-\tilde{x}||+1)\frac{\delta(x)^\tau}{|x-\tilde{x}|^{\tau}},\label{est3gepsad=2}\\
|\widetilde{G}^\varepsilon(x,\tilde{x})|\leq C(|\log|x-\tilde{x}||+1)\frac{\delta(x)^\tau\delta(\tilde{x})^\tau}{|x-\tilde{x}|^{2\tau}},\label{est3gepsbd=2}
\end{align}
\item for all $x,\ \tilde{x}\in D^\varepsilon(0,3/2)$,
\begin{align}
|\nabla_2\widetilde{G}^\varepsilon(x,\tilde{x})|\leq C(|\log|x-\tilde{x}|+1)\frac{\delta(x)^\tau}{|x-\tilde{x}|^{1+\tau}},\quad\mbox{for}\quad |x-\tilde{x}|\leq\varepsilon,\label{est3nablagepsad=2}\\
|\nabla_2\widetilde{G}^\varepsilon(x,\tilde{x})|\leq C(|\log|x-\tilde{x}||+1)\frac{\delta(x)^\tau\delta(\tilde{x})^\tau}{\varepsilon|x-\tilde{x}|^{2\tau}},\quad\mbox{for}\quad |x-\tilde{x}|>\varepsilon.\label{est3nablagepsbd=2}
\end{align}
\end{enumerate}
\end{lem}

Below we only address the proof of Lemma \ref{lemestgreend=3}. The case $d=2$ is handled by the same approach, using the estimate \eqref{estgreend-22} on Green's kernel.

\subsection{Proof of Lemma \ref{bdarycorlem}}

Let us assume that $d\geq 3$. The case $d=2$ is handled similarly using the estimates \eqref{est3nablagepsad=2} and \eqref{est3nablagepsbd=2}. We can represent $v^\varepsilon$ thanks to Green's kernel: for all $x\in D^\varepsilon(0,3/2)$,
\begin{equation*}
v^\varepsilon(x)=\int_{D^\varepsilon(0,2)}\widetilde{G}^\varepsilon(x,\tilde{x})\nabla\cdot A(\tilde{x}/\varepsilon)\nabla(\psi(\tilde{x}'/\varepsilon)\Theta(\tilde{x}',\tilde{x}_d/\varepsilon))d\tilde{x}.
\end{equation*}
Integrating by parts, we get
\begin{equation*}
v^\varepsilon(x)=-\int_{D^\varepsilon(0,2)}\nabla_2 \widetilde{G}^\varepsilon(x,\tilde{x})A(\tilde{x}/\varepsilon)\nabla(\psi(\tilde{x}'/\varepsilon)\Theta(\tilde{x}',\tilde{x}_d/\varepsilon))d\tilde{x}.
\end{equation*}
We now use the fact that the cut-off $\Theta(\tilde{x}',\tilde{x}_d/\varepsilon)$ is supported in $[-3/2,3/2]^{d-1}\times[-3M_0\varepsilon/2,3M_0\varepsilon/2]$, and split the latter integral:
\begin{align*}
v^\varepsilon(x)&=-\int_{D^\varepsilon(0,2)\cap[-3/2,3/2]^{d-1}\times[-3M_0\varepsilon/2,3M_0\varepsilon/2]\cap\{|x-\tilde{x}|\leq\varepsilon\}}\nabla_2 \widetilde{G}^\varepsilon(x,\tilde{x})A(\tilde{x}/\varepsilon)\nabla(\psi(\tilde{x}'/\varepsilon)\Theta(\tilde{x}',\tilde{x}_d/\varepsilon))d\tilde{x}\\
&\qquad-\int_{D^\varepsilon(0,2)\cap[-3/2,3/2]^{d-1}\times[-3M_0\varepsilon/2,3M_0\varepsilon/2]\cap\{|x-\tilde{x}|>\varepsilon\}}\nabla_2 \widetilde{G}^\varepsilon(x,\tilde{x})A(\tilde{x}/\varepsilon)\nabla(\psi(\tilde{x}'/\varepsilon)\Theta(\tilde{x}',\tilde{x}_d/\varepsilon))d\tilde{x}\\
&=I_1+I_2.
\end{align*}
Using $\psi\in\mathcal C_{M_0}^{1,\omega}$ we get on the one hand by \eqref{est3nablagepsa}
\begin{align*}
|I_1|&\leq C/\varepsilon\int_{D^\varepsilon(0,2)\cap[-3/2,3/2]^{d-1}\times[-3M_0\varepsilon/2,3M_0\varepsilon/2]\cap\{|x-\tilde{x}|\leq\varepsilon\}}\frac{\delta(x)^\tau}{|x-\tilde{x}|^{d-1+\tau}}d\tilde{x}\\
&\leq C\delta(x)^\tau/\varepsilon\int_{|x-\tilde{x}|\leq\varepsilon}\frac{1}{|x-\tilde{x}|^{d-1+\tau}}d\tilde{x}\\
&\leq C\delta(x)^\tau/\varepsilon^\tau\int_{|\tilde{y}|\leq 1}\frac{1}{|\tilde{y}|^{d-1+\tau}}d\tilde{y}\leq C\delta(x)^\tau/\varepsilon^\tau,
\end{align*}
and on the other hand by \eqref{est3nablagepsb}
\begin{align*}
|I_2|&\leq C/\varepsilon^2\int_{D^\varepsilon(0,2)\cap[-3/2,3/2]^{d-1}\times[-3M_0\varepsilon/2,3M_0\varepsilon/2]\cap\{|x-\tilde{x}|>\varepsilon\}}\frac{\delta(x)^\tau\delta(\tilde{x})^\tau}{|x-\tilde{x}|^{d-2+2\tau}}d\tilde{x}\\
&\leq C\delta(x)^\tau/\varepsilon^2\int_{\mathbb R^{d-1}\times[-3M_0\varepsilon/2,3M_0\varepsilon/2]\cap\{|x-\tilde{x}|>\varepsilon\}}\frac{\delta(\tilde{x})^\tau}{|x-\tilde{x}|^{d-2+2\tau}}d\tilde{x}\\
&\leq C\delta(x)^\tau\varepsilon^{-2+\tau}\Biggl\{\int_{|x'-\tilde{x}'|>\varepsilon}\frac{1}{|x'-\tilde{x}'|^{d-2+2\tau}}\int_{-3M_0\varepsilon/2}^{3M_0\varepsilon/2}\frac{1}{\left\{1+\frac{(x_d-\tilde{x}_d)^2}{|x'-\tilde{x}'|^{2}}\right\}^{(d-2+2\tau)/2}}d\tilde{x}_dd\tilde{x}'\Biggr.\\
&\qquad\Biggl.+\int_{\{|x-\tilde{x}|>\varepsilon\}\cap\{|x'-\tilde{x}'|<\varepsilon\}\times[-3M_0\varepsilon/2,3M_0\varepsilon/2]}\frac{1}{|x-\tilde{x}|^{d-2+2\tau}}d\tilde{x}\Biggr\}\\
&\leq C\delta(x)^\tau\varepsilon^{-2+\tau}\Biggl\{\varepsilon\int_{|x'-\tilde{x}'|>\varepsilon}\frac{1}{|x'-\tilde{x}'|^{d-2+2\tau}}d\tilde{x}'+\int_{\{|x-\tilde{x}|>\varepsilon\}\cap\{|x'-\tilde{x}'|<\varepsilon\}\times[-3M_0\varepsilon/2,3M_0\varepsilon/2]}\frac{1}{\varepsilon^{d-2+2\tau}}d\tilde{x}\Biggr\}\\
&\leq C\delta(x)^\tau\varepsilon^{-2+\tau}\Biggl\{\varepsilon^{2-2\tau}\int_{|\tilde{y}'|>1}\frac{1}{|\tilde{y}'|^{d-2+2\tau}}d\tilde{y}'+\varepsilon^{2-2\tau}\Biggr\}\leq C\delta(x)^\tau/\varepsilon^{\tau},
\end{align*}
for $1/2<\tau<1$.

\subsection{Proof of Lemma \ref{lemestgreend=3}}

The first observation is that the estimates \eqref{est3gepsa} and \eqref{est3gepsb} follow directly from the boundary H\"older regularity uniform in $\varepsilon$ of Proposition \ref{propboundaryholder}. Let us sketch the proof.

Fix $0<\varepsilon<1$. Let 
\begin{equation*}
\eta:=7/4\left\{4+(1+4/7M_0)^2\right\}^{1/2};
\end{equation*}
it is an upper bound for the diameter of $D^\varepsilon(0,7/4)$ when $0<\varepsilon<1$. Let $0<\tau<1$, $x,\ \tilde{x}\in D^\varepsilon(0,7/4)$, $r:=|x-\tilde{x}|$ and $\bar{x}:=(x',\varepsilon\psi(x'/\varepsilon))\in\Delta^\varepsilon(0,7/4)$. Notice that $|x-\bar{x}|=x_d-\varepsilon\psi(x'/\varepsilon)=\delta(x)$. Either $\delta(x)=|x-\bar{x}|\geq r/(8\eta+1)$ or $\delta(x)=|x-\bar{x}|<r/(8\eta+1)$. In the former case, \eqref{est3gepsa} follows directly from \eqref{estgreend-23} since $(\delta(x)/r)^\tau\geq (1/(8\eta+1))^\tau$ so that
\begin{equation*}
|\widetilde{G}^\varepsilon(x,\tilde{x})|\leq \frac{C}{|x-\tilde{x}|^{d-2}}\leq \frac{C\delta(x)^\tau}{|x-\tilde{x}|^{d-2+\tau}}.
\end{equation*}
In the latter case $\delta(x)=|x-\bar{x}|<r/(8\eta+1)$, we use the fact that $\widetilde{G}^\varepsilon(\cdot,\tilde{x})$ is a weak solution of 
\begin{equation*}
\left\{
\begin{array}{rll}
-\nabla\cdot A(x/\varepsilon)\nabla \widetilde{G}^\varepsilon(x,\tilde{x})&=0,&x\in D^\varepsilon(\bar{x},r/(8\eta+1)),\\
\widetilde{G}^\varepsilon(x,\tilde{x})&=0,&x\in \Delta^\varepsilon(\bar{x},r/(8\eta+1)).
\end{array}
\right.
\end{equation*}
Notice that
\begin{equation*}
D^\varepsilon(\bar{x},r/(8\eta+1))\subset D^\varepsilon(0,15/8),
\end{equation*}
that $\tilde{x}\notin D^\varepsilon(\bar{x},r/(8\eta+1))$, and that $x\in D^\varepsilon(\bar{x},1/(8\eta+1))$. We want to apply the boundary H\"older estimate \eqref{bdaryholdest} properly rescaled. Consider for $z\in D^{\varepsilon'}_{\psi'}(0,1)$
\begin{equation*}
u^{\varepsilon'}(z):=\widetilde{G}^\varepsilon((r/(8\eta+1))z+\bar{x},\tilde{x}),\quad \varepsilon':=(8\eta+1)\varepsilon/r,\quad\psi':=\psi(\cdot+\bar{x}'/\varepsilon)
\end{equation*}
which solves
\begin{equation*}
\left\{
\begin{array}{rll}
-\nabla\cdot A(z/\varepsilon'+\bar{x}/\varepsilon)\nabla u^{\varepsilon'}&=0,&z\in D^{\varepsilon'}_{\psi'}(0,1),\\
u^{\varepsilon'}&=0,&z\in \Delta^{\varepsilon'}_{\psi'}(0,1).
\end{array}
\right.
\end{equation*}
Now, $\psi'\in\mathcal C_{M_0}^1$, $A(\cdot+\bar{x}/\varepsilon)\in\mathcal A^{0,\nu_0}$ and estimate \eqref{bdaryholdest} implies
\begin{multline*}
(r/(8\eta+1))^\tau[\widetilde{G}^\varepsilon(\cdot,\tilde{x})]_{C^{0,\tau}(\overline{D^\varepsilon(\bar{x},r/(16\eta+2))})}=[u^{\varepsilon'}]_{C^{0,\tau}(\overline{D^{\varepsilon'}_{\psi'}(0,1/2)})}\\
\leq C\|u^{\varepsilon'}\|_{L^2(D^{\varepsilon'}_{\psi'}(0,1))}=C\|\widetilde{G}^\varepsilon(\cdot,\tilde{x})\|_{L^2(D^\varepsilon(\bar{x},r/(8\eta+1)))}\left(r/(8\eta+1)\right)^{-d/2},
\end{multline*}
so that using \eqref{estgreend-23} 
\begin{equation*}
\begin{aligned}
|\widetilde{G}^\varepsilon(x,\tilde{x})|&=|\widetilde{G}^\varepsilon(x,\tilde{x})-\widetilde{G}^\varepsilon(\bar{x},\tilde{x})|\leq [\widetilde{G}^\varepsilon(\cdot,\tilde{x})]_{C^{0,\tau}(\overline{D^\varepsilon(\bar{x},r/(16\eta+2))})}\delta(x)^\tau\\
&\leq \frac{C\delta(x)^\tau}{r^{d/2+\tau}}\|\widetilde{G}^\varepsilon(\cdot,\tilde{x})\|_{L^2(D^\varepsilon(\bar{x},r/(8\eta+1)))}\leq \frac{C\delta(x)^\tau}{r^{d-2+\tau}}=\frac{C\delta(x)^\tau}{|x-\tilde{x}|^{d-2+\tau}},
\end{aligned}
\end{equation*}
which is \eqref{est3gepsa}. Estimate \eqref{est3gepsb} is obtained in a similar manner by relying on the bound \eqref{est3gepsa} instead of \eqref{estgreend-23} and by considering $\widetilde{G}^{*,\varepsilon}=\widetilde{G}^{*,\varepsilon}(\tilde{x},x)$ the Green kernel associated to the operator $-\nabla\cdot A^{*}(\tilde{x}/\varepsilon)\nabla$ and the domain $D^\varepsilon(0,2)$, where $(A^*)^{\alpha\beta}_{ij}:=A^{\beta\alpha}_{ji}$. Notice that $\widetilde{G}^{*,\varepsilon}(\tilde{x},x)=\widetilde{G}^{\varepsilon}(x,\tilde{x})^T$ for all $x,\ \tilde{x}\in D^{\varepsilon}(0,2)$. 

The gradient estimates \eqref{est3nablagepsa} and \eqref{est3nablagepsb} now follow from classical Lipschitz estimates applied at small scale. Let $x,\ \tilde{x}\in D^\varepsilon(0,3/2)$ and $r:=|x-\tilde{x}|$. Notice that 
\begin{equation*}
\widetilde{G}^\varepsilon(x,\cdot)^T=\widetilde{G}^{*,\varepsilon}(\cdot,x)\quad\mbox{so that}\quad\nabla_2\widetilde{G}^\varepsilon(x,\cdot)^T=\nabla_1 \widetilde{G}^{*,\varepsilon}(\cdot,x),
\end{equation*}
and
\begin{equation*}
\left\{
\begin{array}{rll}
-\nabla\cdot A^*(\hat{x}/\varepsilon)\nabla \widetilde{G}^{*,\varepsilon}(\hat{x},x)&=0,&\hat{x}\in D^\varepsilon(\tilde{x},r/2),\\
\widetilde{G}^{*,\varepsilon}(\hat{x},x)&=0,&\hat{x}\in \Delta^\varepsilon(\tilde{x},r/2).
\end{array}
\right.
\end{equation*}
If $r\leq\varepsilon$, consider for $z\in D^{\varepsilon'}_{\psi'}(\tilde{x},1)$
\begin{equation*}
u^{\varepsilon'}(z)=\widetilde{G}^{*,\varepsilon}(r/2(z-\tilde{x})+\tilde{x},x),\quad \varepsilon':=2\varepsilon/r,\quad \psi':=\psi(\cdot+\tilde{x}'(1/\varepsilon-1/\varepsilon'))
\end{equation*}
which solves
\begin{equation*}
\left\{
\begin{array}{rll}
-\nabla\cdot A^*(z/\varepsilon'+\tilde{x}(1/\varepsilon-1/\varepsilon'))\nabla u^{\varepsilon'}&=0,&z\in D^{\varepsilon'}_{\psi'}(\tilde{x}',1),\\
u^{\varepsilon'}&=0,&z\in \Delta^{\varepsilon'}_{\psi'}(\tilde{x}',1).
\end{array}
\right.
\end{equation*}
Now, $A^*(\cdot+\tilde{x}(1/\varepsilon-1/\varepsilon'))\in\mathcal A^{0,\nu_0}$, and the classical Lipschitz estimate yields
\begin{equation}\label{classlipestueps'}
\|\nabla u^{\varepsilon'}\|_{L^\infty(D^{\varepsilon'}_{\psi'}(\tilde{x}',1/2))}\leq C'\|u^{\varepsilon'}\|_{L^\infty(D^{\varepsilon'}_{\psi'}(\tilde{x}',1))},
\end{equation}
where $C'$ in the previous inequality depends a priori on $\varepsilon'$ but only through $\|A^*(\cdot+\tilde{x}(1/\varepsilon-1/\varepsilon'))\|_{C^{0,\nu_0}}$ and $\|\nabla(\varepsilon'\psi'(\cdot/\varepsilon'))\|_{C^{0,\nu_0}}$: since $\varepsilon'\geq 2$,
\begin{equation*}
\|A^*(\cdot/\varepsilon'+\tilde{x}(1/\varepsilon-1/\varepsilon'))\|_{C^{0,\nu_0}}=O(1)\ \mbox{and}\ \|\nabla(\varepsilon'\psi'(\cdot/\varepsilon'))\|_{C^{0,\mu}}=O(1).
\end{equation*}
Thus $C'$ can be taken uniform in $\varepsilon'$. Rescaling and applying \eqref{est3gepsa} finally gives
\begin{equation*}
\|\nabla_2\widetilde{G}^\varepsilon(x,\cdot)\|_{L^\infty(D^\varepsilon(\tilde{x},r/4))}\leq \frac{C}{r}\|\widetilde{G}^\varepsilon(x,\cdot)\|_{L^\infty(D^\varepsilon(\tilde{x},r/2))}\leq \frac{C}{r}\sup_{\hat{x}\in D^\varepsilon(\tilde{x},r/2)}\frac{\delta(x)^\tau}{|x-\hat{x}|^{d-2+\tau}}\leq \frac{C\delta(x)^\tau}{r^{d-1+\tau}},
\end{equation*}
since $|x-\hat{x}|\geq |x-\tilde{x}|-|\tilde{x}-\hat{x}|\geq r/2$ for all $\hat{x}\in D^\varepsilon(\tilde{x},r/2)$. Thus
\begin{equation*}
|\nabla_2\widetilde{G}^\varepsilon(x,\tilde{x})|\leq \frac{C\delta(x)^\tau}{|x-\tilde{x}|^{d-1+\tau}}.
\end{equation*}
If $r>\varepsilon$, $D^\varepsilon(\tilde{x},\varepsilon/2)\subset D^\varepsilon(\tilde{x},r/2)$, so that we may directly apply the classical estimate at small scale in combination with \eqref{est3gepsb}
\begin{multline*}
\|\nabla_2\widetilde{G}^\varepsilon(x,\cdot)\|_{L^\infty(D^\varepsilon(\tilde{x},\varepsilon/4))}\leq \frac{C}{\varepsilon}\|\widetilde{G}^\varepsilon(x,\cdot)\|_{L^\infty(D^\varepsilon(\tilde{x},\varepsilon/2))}\\
\leq \frac{C}{\varepsilon}\sup_{\hat{x}\in D^\varepsilon(\tilde{x},\varepsilon/2)}\frac{\delta(x)^\tau\delta(\hat{x})^\tau}{|x-\hat{x}|^{d-2+2\tau}}\leq \frac{C\delta(x)^\tau\delta(\tilde{x})^\tau}{\varepsilon|x-\tilde{x}|^{d-2+2\tau}},
\end{multline*}
which implies \eqref{est3nablagepsb}.

\section{Boundary Lipschitz estimate for Poisson's equation}
\label{secliplap}

We consider the problem
\begin{equation}\label{laposc}
\left\{\begin{array}{rll}
-\Delta u^\varepsilon&=0
,&x\in D^\varepsilon(0,1),\\
u^\varepsilon&=0,&x\in\Delta^\varepsilon(0,1).
\end{array}
\right. 
\end{equation}
Our goal is to show the following proposition:
\begin{prop}\label{proplap}
There exists a constant $C>0$ such that for all $\psi\in\mathcal C_{M_0}^{1,\nu_0}$, for all $\varepsilon>0$, for all $u^\varepsilon$ weak solution to \eqref{laposc},
\begin{equation*}
\left\|\nabla u^\varepsilon\right\|_{L^\infty(D^\varepsilon(0,1/2))}\leq C\left\|u^\varepsilon\right\|_{L^\infty(D^\varepsilon(0,1))}.
\end{equation*}
Notice that $C$ depends on $d$, $N$, $M_0$, $\lambda$ and $\nu_0$.
\end{prop}

The result relies on the classical Schauder estimates for elliptic systems with non-oscillating coefficients. The main point here is to get an estimate uniform in $\varepsilon$. Following Avellaneda and Lin \cite{alin}, we apply the compactness method and prove Proposition \ref{proplap} in three steps: improvement, iteration, blow-up.

\begin{rem}[$\varepsilon$ large]
The Lipschitz estimate for $\varepsilon$ large, say bigger than $\varepsilon_0$, follows from the classical Schauder theory. Therefore, our proof is focused on $\varepsilon<\varepsilon_0$, where the true issues due to the highly oscillating boundary arise.
\end{rem}

The purpose of the lemmas \ref{lem1} and \ref{lem2} below is to show that the homogeneous boundary condition $u^\varepsilon=0$ on $\Delta^\varepsilon(0,1)$ implies that $u^\varepsilon$ is small in a boundary layer of size $O(\varepsilon)$.

\begin{lem}[improvement lemma]\label{lem1}
For all $0<\mu<1$, there exist $0<\varepsilon_0< 1,\ 0<\theta<1/8$, such that for all $\psi\in\mathcal C_{M_0}^{1,\nu_0}$, for all $0<\varepsilon<\varepsilon_0$, for all $u^\varepsilon$ weak solution to 
\begin{equation}\label{laposc1/2}
\left\{\begin{array}{rll}
-\Delta u^\varepsilon&=0
,&x\in D^\varepsilon(0,1/2),\\
u^\varepsilon&=0,&x\in\Delta^\varepsilon(0,1/2),
\end{array}
\right. 
\end{equation}
if
\begin{equation*}
\left\|u^\varepsilon\right\|_{L^\infty(D^\varepsilon(0,1/2))}\leq 1,
\end{equation*}
then 
\begin{equation}\label{estlem1}
\bigl\|u^\varepsilon(x)-\left(\overline{\partial_{x_d} u^\varepsilon}\right)_{0,\theta}\left\{x_d-\varepsilon\psi(x'/\varepsilon)\Theta(x',x_d/\varepsilon)-\varepsilon v^\varepsilon(x)\right\}\bigr\|_{L^\infty(D^\varepsilon(0,\theta))}\leq\theta^{1+\mu},
\end{equation}
where $v^\varepsilon$ is the boundary corrector solving \eqref{sysbdarycorr} with $A=\Idd_d$.
\end{lem}

\begin{lem}[iteration lemma]\label{lem2}
For $0<\mu<1$ fixed, let $0<\varepsilon_0<1$ and $\theta>0$ as given by Lemma \ref{lem1}. There exists $C_1>0$, for all $k\in\mathbb N$, $k\geq 1$, for all $\varepsilon<\theta^{k-1}\varepsilon_0$, for all $\psi\in\mathcal C_{M_0}^{1,\nu_0}$, for all $u^\varepsilon$ weak solution to \eqref{laposc1/2}, if
\begin{equation*}
\left\|u^\varepsilon\right\|_{L^\infty(D^\varepsilon(0,1/2))}\leq 1,
\end{equation*}
then there exist $a^\varepsilon_k\in \mathbb R$ and $V^\varepsilon_k=V^\varepsilon_k(x)$ such that
\begin{equation*}
\begin{aligned}
|a^\varepsilon_k|&\leq (C_1/\theta)[1+\theta^\mu+\ldots\ \theta^{(k-1)\mu}],\\
|V^\varepsilon_k(x)|&\leq (C_0C_1/\theta)[1+\theta^\mu+\ldots\ \theta^{(k-1)\mu}]\frac{\delta(x)^\tau}{\varepsilon^\tau},
\end{aligned}
\end{equation*}
where $C_0$ is the constant appearing in \eqref{estbdarycorlem} and
\begin{equation*}
\bigl\|u^\varepsilon(x)-a^\varepsilon_k\left\{x_d-\varepsilon\psi(x'/\varepsilon)\Theta(x',x_d/\varepsilon)\right\}-\varepsilon V^\varepsilon_k(x)\bigr\|_{L^\infty(D^\varepsilon(0,\theta^k))}\leq\theta^{k(1+\mu)}.
\end{equation*}
\end{lem}

\begin{rem}[Lemma \ref{lem1}] We cannot prove a bound like
\begin{equation*}
\intbar_{D^\varepsilon(0,\theta)}\bigl|\nabla u^\varepsilon-\left(\overline{\nabla u^\varepsilon}\right)_{0,\theta}\bigr|^2\leq\theta^{1+2\mu}
\end{equation*}
because we lack strong convergence of $\nabla u^\varepsilon$. Furthermore, as a computation in dimension $1$ shows, we cannot prove boundedness in $\varepsilon$ of more than $\nabla u^\varepsilon$ in $L^\infty$.
\end{rem}

\begin{rem}[Lemma \ref{lem1}]
We may also prove the estimate
\begin{equation*}
\bigl\|u^\varepsilon-\left(\overline{\nabla u^\varepsilon}\right)_{0,\theta}\cdot x\bigr\|_{L^\infty(D^\varepsilon(0,\theta))}\leq\theta^{1+\mu},
\end{equation*}
but the iteration does not go through since $\left(\overline{\nabla u^\varepsilon}\right)_{0,\theta}\cdot x$ is highly oscillating on the boundary.
\end{rem}

\begin{rem}[Necessity of the boundary corrector]
We could also get an estimate (Lemma \ref{lem1}) on 
\begin{equation*}
u^\varepsilon-\left(\overline{\partial_{x_d} u^\varepsilon}\right)_{0,\theta}\left(x_d-\varepsilon\psi(x'/\varepsilon)\right),
\end{equation*}
which vanishes on the oscillating boundary (good for the iteration argument of Lemma \ref{lem2}). However, this would require to deal with source terms of the form $\nabla\cdot f^\varepsilon$ in the system \eqref{laposc} when proving an analogue of Lemma \ref{lem2}:
\begin{equation*}
\begin{array}{rll}
-\Delta\left(u^\varepsilon(\theta x)-\left(\overline{\partial_{x_d} u^\varepsilon}\right)_{0,\theta}\left(\theta x_d-\varepsilon\psi(\theta x'/\varepsilon)\right)\right)&=\theta\left(\overline{\partial_{x_d} u^\varepsilon}\right)_{0,\theta}\nabla\cdot\left(\nabla\psi(\theta x'/\varepsilon)\right),&x\in D^{\varepsilon/\theta}(0,1/2),\\
u^\varepsilon(\theta x)-\left(\overline{\partial_{x_d} u^\varepsilon}\right)_{0,\theta}\left(\theta x_d-\varepsilon\psi(\theta x'/\varepsilon)\right)&=0,&x\in \Delta^{\varepsilon/\theta}(0,1/2).
\end{array}
\end{equation*} 
The issue here is that $f^\varepsilon=\theta\left(\overline{\partial_{x_d} u^\varepsilon}\right)_{0,\theta}\nabla\psi(\theta x'/\varepsilon)$ is not compact in $C^{0,\alpha}$, but only in $L^{d+\delta}$ which is not enough to get H\"older regularity on $\nabla u^0$ (see Theorem \ref{theoclasslip}) in the compactness argument of Lemma \ref{lem1}. This motivates the introduction of the boundary corrector $v^\varepsilon$ in the expansion for $u^\varepsilon$.
\end{rem}

\begin{rem}
The bound $\left\|u^\varepsilon\right\|_{L^\infty(D^\varepsilon(0,1/2))}\leq 1$ and the divergence theorem yield
\begin{equation}\label{majmoy}
\begin{aligned}
\left|(\overline{\partial_{x_d}u^\varepsilon})_{0,\theta}\right|&=\left|\intbar_{D^\varepsilon(0,\theta)}\partial_{x_d}u^\varepsilon\right|=\frac{1}{|D^\varepsilon(0,\theta)|}\left|\int_{\Delta^0(0,\theta)}u^\varepsilon(x',\varepsilon\psi(x'/\varepsilon)+\theta)dx'\right|\\
&\leq \frac{C}{\theta}\left\|u^\varepsilon\right\|_{L^\infty(D^\varepsilon(0,1/2))}\leq C_1/\theta.
\end{aligned}
\end{equation}
This bound will be used extensively in the iteration procedure of Lemma \ref{lem2}.
\end{rem}

\subsection{Proof of Lemma \ref{lem1}}

Let $0<\mu<1$ be fixed for the whole proof. This lemma contains the compactness argument and its proof is done by contradiction. It follows the scheme of the proof of Lemma \ref{lem1holder}. Roughly speaking the idea is:
\begin{itemize}
\item to assume that estimate \eqref{estlem1} is false for a subsequence $u^{\varepsilon_k}$ bounded in $L^\infty(D^{\varepsilon_k}(0,1))$,
\item and to show that $u^{\varepsilon_k}$ (or a subsequence) converges to $u^0$ solving Poisson's equation in a flat domain, for which classical elliptic estimates \cite{adn1,adn2} apply.
\end{itemize}
Again, some technicalities result from the fact that $u^\varepsilon$ lives in the oscillating upper half-space, whereas its limit $u^0$ lives in the flat one.

\paragraph*{Estimate in the flat domain}
Let $0<\theta<1/8$ and $u^0\in W^{1,2}(D^0(0,1/4))$ be a weak solution of 
\begin{equation}\label{lapu0}
\left\{
\begin{array}{rll}
-\Delta u^0&=0,&x\in D^0(0,1/4),\\
u^0&=0,&x\in \Delta^0(0,1/4),
\end{array}
\right. 
\end{equation}
such that $\left\|u^0\right\|_{L^2(D^0(0,1/4))}\leq 1$. The classical regularity theory of \cite{adn1,adn2} yields $u^0\in C^2(\overline{D^0(0,1/8)})$. Using that for all $x\in D^0(0,\theta)$
\begin{equation*}
\begin{aligned}
u^0(x)-\left(\overline{\partial_{x_d} u^0}\right)_{0,\theta}x_d&=u^0(x)-u^0(x',0)-\left(\overline{\partial_{x_d} u^0}\right)_{0,\theta}x_d\\
&=\int_0^1\partial_{x_d}u^0(x',tx_d)x_ddt-\frac{1}{|D^0(0,\theta)|}\int_{D^0(0,\theta)}\partial_{x_d}u^0(y)x_ddy\\
&=\frac{1}{|D^0(0,\theta)|}\int_0^1\int_{D^0(0,\theta)}\left(\partial_{x_d}u^0(x',tx_d)-\partial_{x_d}u^0(y)\right)x_ddydt.
\end{aligned}
\end{equation*}
we get
\begin{equation}\label{classSchauderu^0}
\bigl\|u^0-\left(\overline{\partial_{x_d} u^0}\right)_{0,\theta}x_d\bigr\|_{L^\infty(D^0(0,\theta))}\leq C_2\theta^2.
\end{equation}
The constant $C_2$ appearing in the former inequality is uniform in $\theta$, so take now for the rest of the proof $0<\theta<1/8$ such that $\theta^{1+\mu}>C_2\theta^2$.

\paragraph*{Contradiction argument: extraction of subsequences}
We carry out the contradiction argument assuming that there exist $\varepsilon_k\rightarrow 0$ ($0<\varepsilon_k<\theta$), $\psi_k\in\mathcal C_{M_0}^{1,\nu_0}$ and $u^{\varepsilon_k}$ solving \eqref{laposc} such that 
\begin{equation}\label{Hcontra1}
\left\|u^{\varepsilon_k}\right\|_{L^\infty(D^{\varepsilon_k}(0,1/2))}\leq 1,
\end{equation}
and
\begin{equation}\label{Hcontra2}
\bigl\|u^{\varepsilon_k}-\left(\overline{\partial_{x_d} u^{\varepsilon_k}}\right)_{0,\theta}\left\{x_d-\varepsilon_k\psi_k(x'/\varepsilon_k)\Theta(x',x_d/\varepsilon_k)-\varepsilon_kv^{\varepsilon_k}(x)\right\}\bigr\|_{L^\infty(D^{\varepsilon_k}(0,\theta))}>\theta^{1+\mu}.
\end{equation}
We recall that $v^{\varepsilon_k}$ is the boundary corrector associated to the boundary graph $\varepsilon_k\psi_k(\cdot/\varepsilon_k)$ and solving \eqref{sysbdarycorr} with $A=\Idd_d$.
The idea is to use the bound 
\begin{equation}\label{bounduepslinfty}
\left\|u^{\varepsilon_k}\right\|_{L^\infty(D^{\varepsilon_k}(0,1/2))}\leq 1,
\end{equation}
together with H\"older and Cacciopoli estimates to extract subsequences in $C^0$ and $W^{1,2}$. One thing one has to take care of is the fact that the bounds hold in the oscillating domain  $x_d>\varepsilon_k\psi_k(x'/\varepsilon_k)$.

First, it follows from the $L^\infty$ bound \eqref{bounduepslinfty} and the boundary H\"older estimate of Proposition \ref{propboundaryholder} that for $0<\mu<1$
\begin{equation*}
\left[u^{\varepsilon_k}\right]_{C^{0,\mu}(\overline{D^{\varepsilon_k}(0,3/8))}}\leq C.
\end{equation*}
Let $\tilde{u}^{\varepsilon_k}$ be defined by: for all $(x',x_d)\in \overline{D^0(0,3/8)}$,
\begin{equation*}
\tilde{u}^{\varepsilon_k}(x',x_d)=u^{\varepsilon_k}(x',x_d+\varepsilon_k\psi_k(x'/\varepsilon_k)).
\end{equation*}
We have, for all $x,\ \hat{x}\in \overline{D^0(0,3/8)}$,
\begin{multline*}
\left|\tilde{u}^{\varepsilon_k}(x',x_d)-\tilde{u}^{\varepsilon_k}(\hat{x}',\hat{x}_d)\right|\leq \left|u^{\varepsilon_k}(x',x_d+\varepsilon_k\psi_k(x'/\varepsilon_k))-u^{\varepsilon_k}(\hat{x}',\hat{x}_d+\varepsilon_k\psi_k(\hat{x}'/\varepsilon_k))\right|\\
\leq C\left(|x'-\hat{x}'|+|x_d+\varepsilon_k\psi_k(x'/\varepsilon_k)-\hat{x}_d-\varepsilon_k\psi_k(\hat{x}'/\varepsilon_k)|\right)^\mu\\
\leq C\left(|x'-\hat{x}'|+|x_d-\hat{x}_d|+M_0|x'-\hat{x}'|\right)^\mu\leq C|x-\hat{x}|^\mu,
\end{multline*}
which means that
\begin{equation*}
\left[\tilde{u}^{\varepsilon_k}\right]_{C^{0,\mu}(\overline{D^0(0,3/8)})}\leq C.
\end{equation*}
The boundedness of $\tilde{u}^{\varepsilon_k}$ in $C^{0,\mu}(\overline{D^0(0,3/8)})$ and Ascoli-Arzela's theorem makes it possible to extract a subsequence again denoted by $\tilde{u}^{\varepsilon_k}$ such that
\begin{equation}\label{CVC0}
\tilde{u}^{\varepsilon_k}\stackrel{k\rightarrow\infty}{\longrightarrow}u^0\quad\mbox{strongly in}\quad C^0(\overline{D^0(0,3/8)}).
\end{equation}

Second, the $L^\infty$ bound \eqref{bounduepslinfty}, the homogeneous Dirichlet boundary condition $u^{\varepsilon_k}=0$ on $\Delta^{\varepsilon_k}(0,1/2)$ and Cacciopoli's inequality imply that 
\begin{equation}\label{boundnablauepsl2}
\left\|\nabla u^{\varepsilon_k}\right\|_{L^2(D^{\varepsilon_k}(0,3/8))}\leq C.
\end{equation}
In order to take care of the fact that the latter bound is for $u^{\varepsilon_k}$ defined in the oscillating domain, we extend $u^{\varepsilon_k}$ on $D^{-1}(0,1/4)$ by $0$ on $\left\{|x'|<\frac{1}{4},\ -1<x_d<\varepsilon_k\psi_k\left(x'/\varepsilon_k\right)\right\}$. Then $u^{\varepsilon_k}\in W^{1,2}(D^{-1}(0,1/4))$ and \eqref{boundnablauepsl2} yields
\begin{equation*}
\left\|\nabla u^{\varepsilon_k}\right\|_{L^2(D^{-1}(0,1/4))}\leq C.
\end{equation*}

Up to extracting a subsequence, we have the following convergences
\begin{equation}\label{CVH1}
\begin{aligned}
u^{\varepsilon_k}\rightharpoonup \bar{u}^0\quad\mbox{weakly in}\quad L^2(D^{-1}(0,1/4)),\\
\nabla u^{\varepsilon_k}\rightharpoonup \nabla \bar{u}^0\quad\mbox{weakly in}\quad L^2(D^{-1}(0,1/4)).
\end{aligned}
\end{equation}

Let us briefly show that $u_0=\bar{u}^0$ in $D^0(0,1/4)$, where $u^0$ is the limit in \eqref{CVC0} and $\bar{u}^0$ is the limit in \eqref{CVH1}. It is sufficient to show that $u^{\varepsilon_k}$ converges toward $u^0$ weakly in $L^2(D^0(0,1/4))$. Let $\varphi\in L^2(D^0(0,1/4))$. We have 
\begin{align*}
&\left|\int_{D^0(0,1/4)}(u^{\varepsilon_k}-u^0)\varphi\right|=\left|\int_{D^0(0,1/4)}(u^{\varepsilon_k}-\tilde{u}^{\varepsilon_k})\varphi+\int_{D^0(0,1/4)}(\tilde{u}^{\varepsilon_k}-u^0)\varphi\right|\\
&\leq\int_{D^0(0,1/4)}|(u^{\varepsilon_k}(x',x_d)-u^{\varepsilon_k}(x',x_d+\varepsilon_k\psi_k(x'/\varepsilon_k))\varphi|+\int_{D^0(0,1/4)}|(\tilde{u}^{\varepsilon_k}-u^0)\varphi|\\
&\leq\int_{|x'|<1/4,\ \varepsilon_k\psi_k(x'/\varepsilon_k)<x_d<1/4}|(u^{\varepsilon_k}(x',x_d)-u^{\varepsilon_k}(x',x_d+\varepsilon_k\psi_k(x'/\varepsilon_k))\varphi|\\
&\quad+\int_{D^0(0,1/4)}|(\tilde{u}^{\varepsilon_k}-u^0)\varphi|+\int_{|x'|<1/4,\ 0<x_d<\varepsilon_k\psi_k(x'/\varepsilon_k)}|u^{\varepsilon_k}(x',x_d+\varepsilon_k\psi_k(x'/\varepsilon_k))\varphi|\\
&\leq \left(C\varepsilon_k^\mu M_0^\mu+\|\tilde{u}^{\varepsilon_k}-u^0\|_{L^\infty(D^0(0,1/4))}\right)\int_{D^0(0,1/4)}|\varphi|+o(1)\stackrel{k\rightarrow\infty}{\longrightarrow}0,
\end{align*}
where
\begin{equation*}
\int_{|x'|<1/4,\ 0<x_d<\varepsilon_k\psi_k(x'/\varepsilon_k)}|u^{\varepsilon_k}(x',x_d+\varepsilon_k\psi_k(x'/\varepsilon_k))\varphi|
\end{equation*}
tends to zero by dominated convergence using the $L^\infty$ bound on $u^{\varepsilon_k}$.

It remains to show that $u^0$ satisfies \eqref{lapu0} and to pass to the limit in \eqref{Hcontra2} in order to get a contradiction. Let $\varphi\in C^\infty_c(D^0(0,1/4))$. For all $k$ sufficiently large, $\supp\varphi\Subset D^{\varepsilon_k}(0,1/4)$ so that 
\begin{equation*}
0=\int_{D^{\varepsilon_k}(0,1/4)}\nabla u^{\varepsilon_k}\cdot\nabla\varphi=\int_{D^0(0,1/4)}\nabla u^{\varepsilon_k}\cdot\nabla\varphi\stackrel{k\rightarrow\infty}{\longrightarrow}\int_{D^0(0,1/4)}\nabla u^0\cdot\nabla\varphi.
\end{equation*}
Therefore, $u^0$ is a weak solution of 
\begin{equation*}
-\Delta u^0=0\quad\mbox{in}\quad D^0(0,1/4).
\end{equation*}
In order to see that $u^0=0$ in $\mathcal D'(\Delta^0(0,1/4))$, we may use the strong convergence of $\tilde{u}^{\varepsilon_k}$ in $C^0(\overline{D^0(0,1/4)})$: for all $|x'|<1/4$,
\begin{equation*}
0=u^{\varepsilon_k}(x',\varepsilon_k\psi_k(x'/\varepsilon_k))=\tilde{u}^{\varepsilon_k}(x',0)\stackrel{k\rightarrow\infty}{\longrightarrow}u^0(x',0).
\end{equation*}
Therefore, $u^0=0$ in $\mathcal D'(\Delta^0(0,1/4))$, which implies $u^0=0$ in $W^{1/2,2}\left(\Delta^0(0,1/4)\right)$.

\paragraph{Contradiction argument: final step}
The final step in this lemma is to pass to the limit in \eqref{Hcontra2} and get a contradiction with \eqref{classSchauderu^0}. Let us first rewrite \eqref{Hcontra2} as a bound in the non-oscillating domain $D^0(0,\theta)$. We have, for all $(x',x_d)\in D^0(0,\theta)$
\begin{multline*}
u^{\varepsilon_k}(x',x_d+\varepsilon_k\psi_k(x'/\varepsilon_k))-\left(\overline{\partial_{x_d} u^{\varepsilon_k}}\right)_{0,\theta}\left\{x_d+\varepsilon_k\psi_k(x'/\varepsilon_k)(1-\Theta(x',x_d/\varepsilon_k+\psi_k(x'/\varepsilon_k)))\right.\\
\left.-\varepsilon_kv^{\varepsilon_k}(x',x_d+\varepsilon_k\psi_k(x'/\varepsilon_k))\right\}\\
=\tilde{u}^{\varepsilon_k}(x',x_d)-\left(\overline{\partial_{x_d} u^{\varepsilon_k}}\right)_{0,\theta}\left\{x_d+\varepsilon_k\psi_k(x'/\varepsilon_k)(1-\Theta(x',x_d/\varepsilon_k+\psi_k(x'/\varepsilon_k)))\right.\\
\left.-\varepsilon_kv^{\varepsilon_k}(x',x_d+\varepsilon_k\psi_k(x'/\varepsilon_k))\right\},
\end{multline*}
so that
\begin{multline*}
\bigl\|u^{\varepsilon_k}-\left(\overline{\partial_{x_d} u^{\varepsilon_k}}\right)_{0,\theta}\left\{x_d-\varepsilon_k\psi_k(x'/\varepsilon_k)\Theta(x',x_d/\varepsilon_k)-\varepsilon_kv^{\varepsilon_k}(x)\right\}\bigr\|_{L^\infty(D^{\varepsilon_k}(0,\theta))}\\
=\bigl\|\tilde{u}^{\varepsilon_k}(x',x_d)-\left(\overline{\partial_{x_d} u^{\varepsilon_k}}\right)_{0,\theta}\left\{x_d+\varepsilon_k\psi_k(x'/\varepsilon_k)(1-\Theta(x',x_d/\varepsilon_k+\psi_k(x'/\varepsilon_k)))\bigr.\right.\\
\bigl.\left.-\varepsilon_kv^{\varepsilon_k}(x',x_d+\varepsilon_k\psi_k(x'/\varepsilon_k))\right\}\bigr\|_{L^\infty(D^0(0,\theta))}
\end{multline*}

Let us show that
\begin{equation*}
(\overline{\partial_{x_d}u^{\varepsilon_k}})_{D^{\varepsilon_k}(0,\theta)}\stackrel{k\rightarrow\infty}{\longrightarrow}(\overline{\partial_{x_d}u^0})_{D^0(0,\theta)}.
\end{equation*}
Indeed, we have
\begin{align*}
(\overline{\partial_{x_d}u^{\varepsilon_k}})_{D^{\varepsilon_k}(0,\theta)}&=\frac{1}{|D^{\varepsilon_k}(0,\theta)|}\int_{D^{\varepsilon_k}(0,\theta)}\partial_{x_d}u^{\varepsilon_k}=\frac{1}{|D^0(0,\theta)|}\int_{D^{\varepsilon_k}(0,\theta)}\partial_{x_d}u^{\varepsilon_k}\\
&=\frac{1}{|D^0(0,\theta)|}\left(\int_{D^0(0,\theta)}\partial_{x_d}u^{\varepsilon_k}+\int_{D^{\varepsilon_k}(0,\theta)\setminus D^0(0,\theta)}\partial_{x_d}u^{\varepsilon_k}-\int_{D^0(0,\theta)\setminus D^{\varepsilon_k}(0,\theta)}\partial_{x_d}u^{\varepsilon_k}\right).
\end{align*}
The first integral above converges to $(\overline{\partial_{x_d}u^0})_{D^0(0,\theta)}$ because of the weak convergence \eqref{CVH1} (take the caracteristic function of $D^0(0,\theta)$ as a test function), the second and the third integrals converge to zero by Cauchy-Schwarz inequality, the uniform $L^2$ bound on $\nabla u^{\varepsilon_k}$ and the fact that the Lebesgue measure of $D^{\varepsilon_k}(0,\theta)\setminus D^0(0,\theta)$ and of $D^0(0,\theta)\setminus D^{\varepsilon_k}(0,\theta)$ tends to zero as $k\rightarrow\infty$. Now, using among other things the boundedness of $\psi_k$, we have the following convergences 
\begin{align*}
\tilde{u}^{\varepsilon_k}&\longrightarrow u^0\quad\mbox{strongly in}\quad C^0(\overline{D^0(0,1/4)}),\\
\varepsilon_k\psi_k(x'/\varepsilon_k)(1-\Theta(x',x_d/\varepsilon_k+\psi_k(x'/\varepsilon_k)))&\longrightarrow 0\quad\mbox{strongly in}\quad L^\infty(D^0(0,1/4)),\\
\varepsilon_kv^{\varepsilon_k}(x',x_d+\varepsilon_k\psi_k(x'/\varepsilon_k))&\longrightarrow 0\quad\mbox{strongly in}\quad L^\infty(D^0(0,1/4)),
\end{align*}
the last line being a consequence of Lemma \ref{bdarycorlem}. Passing to the limit in \eqref{Hcontra2} we get
\begin{equation*}
\bigl\|u^0-\left(\overline{\partial_{x_d} u^0}\right)_{0,\theta}x_d\bigr\|_{L^\infty(D^0(0,\theta))}\geq \theta^{1+\mu}>C_2\theta^2,
\end{equation*}
which is a contradiction.

\subsection{Proof of Lemma \ref{lem2}}

The idea is to iterate using Lemma \ref{lem1}. The case $k=1$ corresponds to Lemma \ref{lem1}. Let us show the estimate for $k=2$ in order to figure out what the expansion for $u^\varepsilon$ should look like. Let 
\begin{equation*}
U^\varepsilon(x):=\frac{1}{\theta^{1+\mu}}\left[u^\varepsilon(\theta x)-a^\varepsilon_1\left\{\theta x_d-\varepsilon\psi(\theta x'/\varepsilon)\Theta(\theta x',\theta x_d/\varepsilon)\right\}-\varepsilon V^\varepsilon_1(\theta x)\right],
\end{equation*}
with 
\begin{align}
a^\varepsilon_1&:=(\overline{\partial_{x_d}u^\varepsilon})_{0,\theta},\\
V^\varepsilon_1(x)&:=(\overline{\partial_{x_d}u^\varepsilon})_{0,\theta}v^\varepsilon(x).
\end{align}
It solves 
\begin{equation*}
\left\{\begin{array}{rll}
-\Delta U^\varepsilon&=0
,&x\in D^{\varepsilon/\theta}(0,1/2),\\
U^\varepsilon&=0,&x\in\Delta^{\varepsilon/\theta}(0,1/2).
\end{array}
\right. 
\end{equation*}
By Lemma \ref{lem1},
\begin{equation}\label{LinftyboundUeps}
\left\|U^\varepsilon\right\|_{L^\infty(D^{\varepsilon/\theta}(0,1/2))}\leq 1.
\end{equation}
Applying Lemma \ref{lem1} to $U^\varepsilon$, we get for $\varepsilon/\theta<\varepsilon_0$
\begin{equation*}
\left\|U^\varepsilon-(\overline{\partial_{x_d}U^\varepsilon})_{0,\theta}\left\{x_d-\varepsilon/\theta\psi(\theta x'/\varepsilon)\Theta(x',\theta x_d/\varepsilon)-\varepsilon/\theta v^{\varepsilon/\theta}(x)\right\}\right\|_{L^\infty(D^{\varepsilon/\theta}(0,\theta))}\leq \theta^{1+\mu}.
\end{equation*}
The bound \eqref{LinftyboundUeps} yields again (see \eqref{majmoy})
\begin{equation*}
\left|(\overline{\partial_{x_d}U^\varepsilon})_{0,\theta}\right|\leq C_1/\theta.
\end{equation*}
We then get
\begin{multline*}
\bigl\|u^\varepsilon(\theta x)-a^\varepsilon_1\left\{\theta x_d-\varepsilon\psi(\theta x'/\varepsilon)\Theta(\theta x',\theta x_d/\varepsilon)\right\}-\varepsilon V^\varepsilon_1(\theta x)\bigr.\\
\bigl.-\theta^\mu(\overline{\partial_{x_d}U^\varepsilon})_{0,\theta}\left\{\theta x_d-\varepsilon\psi(\theta x'/\varepsilon)\Theta(x',\theta x_d/\varepsilon)-\varepsilon v^{\varepsilon/\theta}(x)\right\}\bigr\|_{L^\infty(D^{\varepsilon/\theta}(0,\theta))}\leq \theta^{2(1+\mu)}.
\end{multline*}
Notice that by \eqref{condTheta}, for all $|x'|<\theta<1/8\leq 1$, $\Theta(x',\theta x_d/\varepsilon)=\Theta(\theta x',\theta x_d/\varepsilon)$ so that the former estimate is equivalent to
\begin{equation*}
\left\|u^\varepsilon(x)-a^\varepsilon_2\left\{x_d-\varepsilon\psi(x'/\varepsilon)\Theta(x',x_d/\varepsilon)\right\}-\varepsilon V^\varepsilon_2(x)\right\|_{L^\infty(D^{\varepsilon}(0,\theta^2))}\leq \theta^{2(1+\mu)}
\end{equation*}
with
\begin{align*}
a^\varepsilon_2&:=a^\varepsilon_1+\theta^\mu(\overline{\partial_{x_d}U^\varepsilon})_{0,\theta},\\
V^\varepsilon_2&:=V^\varepsilon_1(x)+\theta^\mu(\overline{\partial_{x_d}U^\varepsilon})_{0,\theta}v^{\varepsilon/\theta}(x/\theta).
\end{align*}
We make now two important observations which make it possible to iterate. The first one is that 
\begin{equation*}
\left\{
\begin{array}{rll}
-\Delta V^\varepsilon_2&=a^\varepsilon_2\Delta(\psi(x'/\varepsilon)\Theta(x',x_d/\varepsilon)),&x\in D^\varepsilon(0,\theta^2),\\
V^\varepsilon_2&=0,&x\in\Delta^\varepsilon(0,\theta^2).
\end{array}
\right. 
\end{equation*}
Secondly, on the one hand it follows from \eqref{majmoy} that
\begin{equation*}
|a^\varepsilon_2|\leq |a^\varepsilon_1|+\theta^\mu|(\overline{\partial_{x_d}U^\varepsilon})_{0,\theta}|\leq (C_1/\theta)\left[1+\theta^\mu\right]
\end{equation*}
and on the other hand, it follows from Lemma \ref{bdarycorlem} that for all $x\in D^\varepsilon(0,\theta^2)$,
\begin{equation*}
|V^\varepsilon_2(x)|\leq |V^\varepsilon_1(x)|+\theta^\mu|(\overline{\partial_{x_d}U^\varepsilon})_{0,\theta}||v^{\varepsilon/\theta}(x/\theta)|\leq (C_0C_1/\theta)\left[1+\theta^{\mu}\right]\frac{\delta(x)^\tau}{\varepsilon^\tau}.
\end{equation*}

We now carry out the induction argument. Let $k\in\mathbb N$, $k\geq 1$. Assume that for all $\varepsilon<\theta^{k-1}\varepsilon_0$, for all $\psi\in \mathcal C_{M_0}^{1,\nu_0}$, for all $u^\varepsilon$ weak solution to \eqref{laposc1/2}, if
\begin{equation*}
\left\|u^\varepsilon\right\|_{L^\infty(D^\varepsilon(0,1/2))}\leq 1,
\end{equation*}
then there exist $a^\varepsilon_k\in \mathbb R$ and $V^\varepsilon_k=V^\varepsilon_k(x)$ such that
\begin{equation*}
\begin{aligned}
|a^\varepsilon_k|&\leq (C_0/\theta)[1+\theta^\mu+\ldots\ \theta^{(k-1)\mu}],\\
|V^\varepsilon_k(x)|&\leq (C_0C_1/\theta)[1+\theta^\mu+\ldots\ \theta^{(k-1)\mu}]\frac{\delta(x)^\tau}{\varepsilon^\tau},
\end{aligned}
\end{equation*}
and
\begin{equation*}
\left\{
\begin{array}{rll}
-\Delta V^\varepsilon_k&=a^\varepsilon_k\Delta(\psi(x'/\varepsilon)\Theta(x',x_d/\varepsilon)),&x\in D^\varepsilon(0,\theta^{k+1}),\\
V^\varepsilon_k&=0,&x\in\Delta^\varepsilon(0,\theta^{k+1}),
\end{array}
\right. 
\end{equation*}
and such that the following estimate holds
\begin{equation}\label{itlem2k}
\bigl\|u^\varepsilon(x)-a^\varepsilon_k\left\{x_d-\varepsilon\psi(x'/\varepsilon)\Theta(x',x_d/\varepsilon)\right\}-\varepsilon V^\varepsilon_k(x)\bigr\|_{L^\infty(D^\varepsilon(0,\theta^k))}\leq\theta^{k(1+\mu)}.
\end{equation}
Fix now $\varepsilon<\theta^{k}\varepsilon_0$, $\psi$ and $u^\varepsilon$ solving \eqref{laposc}. Let 
\begin{equation*}
U^\varepsilon(x):=\frac{1}{\theta^{k(1+\mu)}}\left[u^\varepsilon(\theta^k x)-a^\varepsilon_k\left\{\theta^k x_d-\varepsilon\psi(\theta^k x'/\varepsilon)\Theta(\theta^k x',\theta^k x_d/\varepsilon)\right\}-\varepsilon V^\varepsilon_k(\theta^k x)\right].
\end{equation*}
It solves 
\begin{equation*}
\left\{\begin{array}{rll}
-\Delta U^\varepsilon&=0
,&x\in D^{\varepsilon/\theta^k}(0,1/2),\\
U^\varepsilon&=0,&x\in\Delta^{\varepsilon/\theta^k}(0,1/2).
\end{array}
\right. 
\end{equation*}
By our recursive assumption \eqref{itlem2k},
\begin{equation}\label{LinftyboundUepsk}
\left\|U^\varepsilon\right\|_{L^\infty(D^{\varepsilon/\theta^k}(0,1/2))}\leq 1.
\end{equation}
Applying Lemma \ref{lem1} to $U^\varepsilon$, we get for $\varepsilon/\theta^k<\varepsilon_0$
\begin{equation*}
\left\|U^\varepsilon-(\overline{\partial_{x_d}U^\varepsilon})_{0,\theta}\left\{x_d-\varepsilon/\theta^k\psi(\theta^k x'/\varepsilon)\Theta(x',\theta^k x_d/\varepsilon)-\varepsilon/\theta^k v^{\varepsilon/\theta^k}(x)\right\}\right\|_{L^\infty(D^{\varepsilon/\theta^k}(0,\theta))}\leq \theta^{1+\mu}.
\end{equation*}
The bound \eqref{LinftyboundUepsk} yields again (see \eqref{majmoy})
\begin{equation*}
\left|(\overline{\partial_{x_d}U^\varepsilon})_{0,\theta}\right|\leq C_1/\theta.
\end{equation*}
We get
\begin{multline*}
\bigl\|u^\varepsilon(\theta^kx)-a^\varepsilon_k\left\{\theta^kx_d-\varepsilon\psi(\theta^kx'/\varepsilon)\Theta(\theta^kx',\theta x_d/\varepsilon)\right\}-\varepsilon V^\varepsilon_k(\theta^k x)\bigr.\\
\bigl.-\theta^{k\mu}(\overline{\partial_{x_d}U^\varepsilon})_{0,\theta}\left\{\theta^k x_d-\varepsilon\psi(\theta^k x'/\varepsilon)\Theta(x',\theta^k x_d/\varepsilon)-\varepsilon v^{\varepsilon/\theta^k}(x)\right\}\bigr\|_{L^\infty(D^{\varepsilon/\theta^k}(0,\theta))}\leq \theta^{(k+1)(1+\mu)}.
\end{multline*}
Notice that by \eqref{condTheta}, for all $|x'|<\theta^k\leq\theta<1/8\leq 1$, $\Theta(x',\theta^k x_d/\varepsilon)=\Theta(\theta^k x',\theta^k x_d/\varepsilon)$ so that the former estimate is equivalent to
\begin{equation*}
\left\|u^\varepsilon(x)-a^\varepsilon_{k+1}\left\{x_d-\varepsilon\psi(x'/\varepsilon)\Theta(x',x_d/\varepsilon)\right\}-\varepsilon V^\varepsilon_{k+1}(x)\right\|_{L^\infty(D^{\varepsilon}(0,\theta^{k+1}))}\leq \theta^{(k+1)(1+\mu)}
\end{equation*}
with
\begin{align*}
a^\varepsilon_{k+1}&:=a^\varepsilon_k+\theta^{k\mu}(\overline{\partial_{x_d}U^\varepsilon})_{0,\theta},\\
V^\varepsilon_{k+1}&:=V^\varepsilon_k(x)+\theta^{k\mu}(\overline{\partial_{x_d}U^\varepsilon})_{0,\theta}v^{\varepsilon/\theta^k}(x/\theta^k).
\end{align*}
Firstly, we have that 
\begin{equation*}
\left\{
\begin{array}{rll}
-\Delta V^\varepsilon_{k+1}&=a^\varepsilon_{k+1}\Delta(\psi(x'/\varepsilon)\Theta(x',x_d/\varepsilon)),&x\in D^\varepsilon(0,\theta^{k+1}),\\
V^\varepsilon_{k+1}&=0,&x\in\Delta^\varepsilon(0,\theta^{k+1}).
\end{array}
\right. 
\end{equation*}
Secondly, on the one hand it follows from \eqref{majmoy} that
\begin{equation*}
|a^\varepsilon_{k+1}|\leq |a^\varepsilon_k|+\theta^{k\mu}|(\overline{\partial_{x_d}U^\varepsilon})_{0,\theta^k}|\leq (C_0/\theta)\left[1+\theta^\mu+\ldots\ \theta^{k\mu}\right]
\end{equation*}
and on the other hand, it follows from Lemma \ref{bdarycorlem} and from our iterative assumption that for all $x\in D^\varepsilon(0,\theta^{k+1})$,
\begin{equation*}
|V^\varepsilon_{k+1}(x)|\leq |V^\varepsilon_k(x)|+\theta^{k\mu}|(\overline{\partial_{x_d}U^\varepsilon})_{0,\theta}||v^{\varepsilon/\theta^k}(x/\theta^k)|\leq (C_0C_1/\theta)\left[1+\theta^{\mu}+\ldots\ \theta^{k\mu}\right]\frac{\delta(x)^\tau}{\varepsilon^\tau}.
\end{equation*}
This concludes the iteration.

\subsection{Proof of Proposition \ref{proplap}}

If $\varepsilon\geq\varepsilon_0$, Proposition \ref{proplap} folows from the classical Lipschitz estimate. Let $0<\varepsilon<\varepsilon_0$ and $\psi\in\mathcal C_{M_0}^{1,\nu_0}$ be fixed for the whole proof. We start with $x_0:=(0,x_{0,d})\in D^\varepsilon(0,1/2)$. Recall that $\delta(x_0):=x_{0,d}-\varepsilon\psi(0)$. Let $k$ be the unique integer $k\geq 1$ such that $\theta^k\leq\varepsilon/\varepsilon_0<\theta^{k-1}$. As in the proof of Proposition \ref{propboundaryholder}, the idea is to consider two cases:
\begin{itemize}
\item ``far'' from the boundary: $\delta(x_0)>\theta^{k+1}/2$, in which case we rely on interior Lipschitz gradient estimates;
\item ``close'' to the boundary: $\delta(x_0)\leq \theta^{k+1}/2$, in which case we rely on a blow-up argument and the classical Lipschitz estimate of Theorem \ref{theoclasslip} near the boundary.
\end{itemize}
In either case, we use Lemma \ref{lem2} to bound $u^\varepsilon$ in $L^\infty$ and show that $u^\varepsilon$ is not too big close to the oscillating boundary. We assume that $\|u^\varepsilon\|_{L^\infty(D^\varepsilon(0,1/2))}\leq 1$.

\subsubsection*{First case: ``far'' from the boundary}

There exist $0\leq l\leq k$ such that 
\begin{equation*}
\theta^{l+1}/2<\delta(x_0)\leq\theta^l/2.
\end{equation*}
Applying the interior gradient estimate \eqref{intgradestrescaled}, we get
\begin{equation*}
\|\nabla u^\varepsilon\|_{L^\infty(B(x_0,\delta(x_0)/4))}\leq \frac{C}{\delta(x_0)}\|u^\varepsilon\|_{L^\infty(B(x_0,\delta(x_0)/2))}
\end{equation*}
where $C$ is uniform in $\varepsilon$. Of course,  the boundedness of $\psi$ implies that for $\varepsilon$ sufficiently small, $B(x_0,\delta(x_0)/4)$ does not intersect the oscillating boundary. We now use Lemma \ref{lem2} to bound $\|u^\varepsilon\|_{L^\infty(B(x_0,\delta(x_0)/2))}$. We get if $l=0$
\begin{equation*}
\|u^\varepsilon\|_{L^\infty(B(x_0,\delta(x_0)/2))}\leq \|u^\varepsilon\|_{L^\infty(D^\varepsilon(0,1))}\leq 1<2\delta(x_0)/\theta, 
\end{equation*}
and if $l\geq 1$ 
\begin{multline}\label{majuepsLinftycas1}
\|u^\varepsilon\|_{L^\infty(B(x_0,\delta(x_0)/2))}\leq \|u^\varepsilon\|_{L^\infty(D^\varepsilon(0,\theta^l))}\leq \theta^{l(1+\mu)}+|a^\varepsilon_l|\theta^l+\varepsilon \|V^\varepsilon_k(x)\|_{L^\infty(D^\varepsilon(0,\theta^l))}\\
\leq\theta^{l(1+\mu)}+\frac{C_0\max(1,C_1)}{\theta(1-\theta^\mu)}\left[\theta^l+\varepsilon^{1-\tau}\theta^{\tau l}\right]\leq \theta^{l(1+\mu)}+\frac{C_0\max(1,C_1)}{\theta(1-\theta^\mu)}\theta^l\left[1+\frac{\varepsilon_0^{1-\tau}}{\theta^{1-\tau}}\right]\\
\leq C_{\varepsilon_0,\theta}\theta^l\leq C_{\varepsilon_0,\theta}\delta(x_0),
\end{multline}
so that 
\begin{equation*}
\|\nabla u^\varepsilon\|_{L^\infty(B(x_0,\delta(x_0)/4))}\leq C.
\end{equation*}
Notice that we have used the inequality
\begin{equation*}
\|x_d-\varepsilon\psi(x'/\varepsilon)\Theta(x',x_d/\varepsilon)\|_{L^\infty(D^\varepsilon(0,\theta^l))}\leq \|x_d-\varepsilon\psi(x'/\varepsilon)\|_{L^\infty(D^\varepsilon(0,\theta^l))}+M_0\varepsilon.
\end{equation*}
\begin{center}
\includegraphics[width=9.78cm,height=6.60cm]{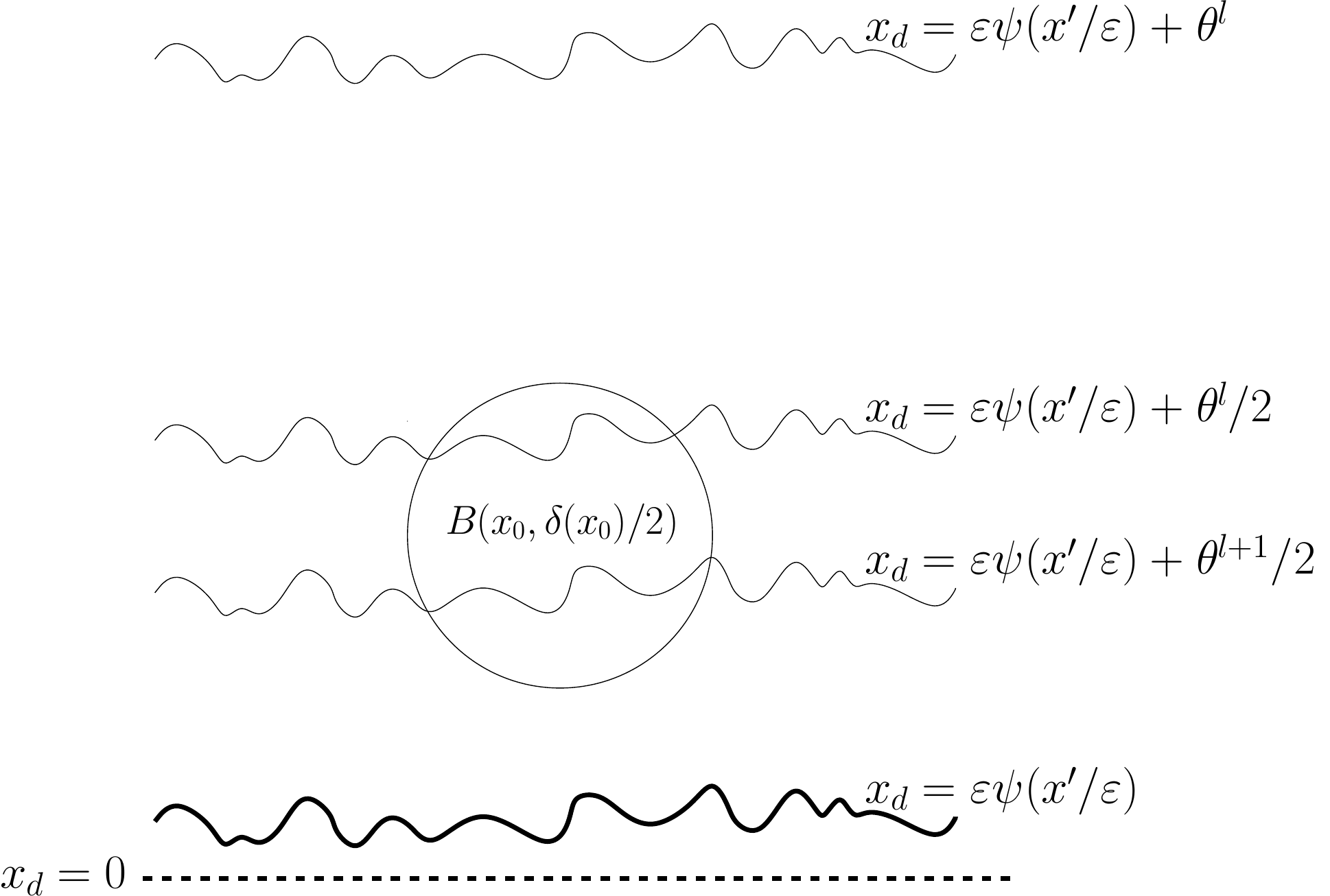}
\end{center}

\subsubsection*{Second case: ``close'' to the boundary}
We blow-up and apply classical estimates. This works here, since we are at the microscale. Let us consider 
\begin{equation*}
U^\varepsilon(y):=\frac{1}{\varepsilon}u^\varepsilon(\varepsilon y),
\end{equation*}
which solves
\begin{equation*}
\left\{\begin{array}{rll}
-\Delta U^\varepsilon&=0,&y\in D^1(0,1/\varepsilon),\\
U^\varepsilon&=0,&y\in\Delta^1(0,1/\varepsilon),
\end{array}
\right. 
\end{equation*}
where $D^1(0,1/\varepsilon)$ is the region bounded from below by the oscillating graph at scale $1$: $y_d=\psi(y')$. Applying the classical Lipschitz estimate \eqref{lipclassestu} near the boundary to $U^\varepsilon$, we get
\begin{equation*}
\|\nabla U^\varepsilon\|_{L^\infty(D^1(0,1/(2\varepsilon_0)))}\leq C\|U^\varepsilon\|_{L^\infty(D^1(0,1/\varepsilon_0))}.
\end{equation*}
On the one hand, 
\begin{equation*}
\|\nabla U^\varepsilon\|_{L^\infty(D^1(0,1/(2\varepsilon_0)))}=\|\nabla u^\varepsilon\|_{L^\infty(D^\varepsilon(0,\varepsilon/(2\varepsilon_0)))}
\end{equation*}
and on the other hand, rescaling and by Lemma \ref{lem2}
\begin{multline}\label{majuepsLinftycas2}
\|U^\varepsilon\|_{L^\infty(D^1(0,1/\varepsilon_0))}=\frac{1}{\varepsilon}\|U^\varepsilon\|_{L^\infty(D^\varepsilon(0,\varepsilon/\varepsilon_0))}\leq \frac{1}{\varepsilon}\|u^\varepsilon\|_{L^\infty(D^\varepsilon(0,\theta^{k-1}))}\\\leq \frac{\theta^k}{\varepsilon}\left(\theta^{k\mu}+\frac{C_0\max(1,C_1)}{\theta(1-\theta^\mu)}\left[1+\frac{\varepsilon^{1-\tau}}{\theta^{k(1-\tau)}}\right]\right)\leq C_{\varepsilon_0,\theta}\frac{\theta^k}{\varepsilon}\leq C.
\end{multline}

\subsubsection*{Arbitrary point $x_0$}
Take a point $x_0=(x_0',x_{0,d})\in D^\varepsilon(0,1/2)$. Remember that $\varepsilon$ is fixed. We consider $\tilde{\psi}$ defined for $y'\in\mathbb R^{d-1}$ by
\begin{equation*}
\tilde{\psi}(y'):=\psi(y'+x_0'/\varepsilon).
\end{equation*}
Notice that $\tilde{\psi}\in \mathcal C_{M_0}^{1,\nu_0}$. For all $|x'|<1/2$, $\varepsilon\tilde{\psi}(x'/\varepsilon)<x_d<\varepsilon\tilde{\psi}(x'/\varepsilon)+1/2$, let $\tilde{u}^\varepsilon$ be defined by
\begin{equation*}
\tilde{u}^\varepsilon(x',x_d):=u^\varepsilon(x'+x_0',x_d).
\end{equation*}
We have that 
\begin{equation*}
\|\tilde{u}^\varepsilon\|_{L^\infty(D^\varepsilon_{\tilde{\psi}}(0,1/2))}\leq 1
\end{equation*}
and $\tilde{u}^\varepsilon$ solves
\begin{equation*}
\left\{\begin{array}{rll}
\Delta \tilde{u}^\varepsilon&=0,&x\in D^\varepsilon_{\tilde{\psi}}(0,1/2),\\
\tilde{u}^\varepsilon&=0,&x\in \Delta^\varepsilon_{\tilde{\psi}}(0,1/2).
\end{array}
\right.
\end{equation*}
We can apply Lemma \ref{lem2} and argue exactly as above for $x_0=(0,x_{0,d})$. 

\section{Boundary Lipschitz estimate for systems with oscillating coefficients}
\label{secbdrliposc}

In this section, we address the Lipschitz estimates for weak solutions of elliptic systems such as \eqref{sysoscueps} with oscillating coefficients as well as oscillating boundary. Our main result is the following:

\begin{theo}\label{theoboundarylip}
Let $0<\mu<1$ and $\kappa>0$. There exist $C>0$, $\varepsilon_0>0$, such that for all $\psi\in\mathcal C_{M_0}^{1,\nu_0}$, for all $A\in\mathcal A^{0,\nu_0}$, for all $\varepsilon>0$, for all $f\in L^{d+\kappa}(D^\varepsilon(0,1))$, for all $F\in C^{0,\mu}(D^\varepsilon(0,1))$, for all $u^\varepsilon$ weak solution to \eqref{sysoscueps} the bounds
\begin{equation*}
\|u^\varepsilon\|_{L^\infty(D^\varepsilon(0,1))}\leq 1,\quad\|f\|_{L^{d+\kappa}(D^\varepsilon(0,1))}\leq \varepsilon_0,\quad\|F\|_{C^{0,\mu}(D^\varepsilon(0,1))}\leq \varepsilon_0
\end{equation*}
imply
\begin{equation*}
\|\nabla u^\varepsilon\|_{L^\infty(D^\varepsilon(0,1/2))}\leq C.
\end{equation*}
Notice that $C$ and $\varepsilon_0$ depend on $d$, $N$, $M_0$, $\lambda$, $\nu_0$, $\kappa$ and $\mu$.
\end{theo}

\begin{rem}
Of course, this boils down to the estimate
\begin{equation}\label{bdarylipest}
\|\nabla u^\varepsilon\|_{L^\infty(D^\varepsilon(0,1/2))}\leq C\left\{\|u^\varepsilon\|_{L^2(D^\varepsilon(0,1))}+\|f\|_{L^{d+\kappa}(D^\varepsilon(0,1))}+\|F\|_{C^{0,\mu}(D^\varepsilon(0,1))}\right\}.
\end{equation} 
with $C>0$ uniform in $\varepsilon$.
\end{rem}

Once again, the proof relies on the three-steps compactness method. However, with the expansion for $u^\varepsilon$ provided in Lemma \ref{lem2}, the iteration argument does not go through. Indeed, assuming for simplicity that $f=F=0$, we may prove as for Poisson's equation (cf. Lemma \ref{lem1}) the existence of $\varepsilon_0,\ \theta>0$ such that for all $0<\varepsilon<\varepsilon_0$
\begin{equation}\label{estlem1model}
\bigl\|u^\varepsilon(x)-\left(\overline{\partial_{x_d} u^\varepsilon}\right)_{0,\theta}\left\{x_d-\varepsilon\psi(x'/\varepsilon)\Theta(x',x_d/\varepsilon)-\varepsilon v^\varepsilon(x)\right\}\bigr\|_{L^\infty(D^\varepsilon(0,\theta))}\leq\theta^{1+\mu},
\end{equation}
where $v^\varepsilon$ is the boundary corrector solving \eqref{sysbdarycorr}. However,
\begin{equation*}
U^\varepsilon(x):=\frac{1}{\theta^{1+\mu}}\left[u^\varepsilon(\theta x)-a^\varepsilon_1\left\{\theta x_d-\varepsilon\psi(\theta x'/\varepsilon)\Theta(\theta x',\theta x_d/\varepsilon)\right\}-\varepsilon V^\varepsilon_1(\theta x)\right]
\end{equation*}
solves 
\begin{equation*}
\left\{
\begin{array}{rll}
-\nabla\cdot A(\theta x/\varepsilon)\nabla U^\varepsilon&=a^\varepsilon_1/\theta^{\mu}\partial_\alpha\left(A^{\alpha d}(\theta x/\varepsilon)\right),&x\in D^\varepsilon(0,1/2),\\
U^\varepsilon&=0,&x\in \Delta^\varepsilon(0,1/2).
\end{array}
\right. 
\end{equation*}
Yet, neither is this right hand side zero, nor is $A^{\alpha d}(\theta \cdot/\varepsilon)$ uniformly bounded in $C^{0,\mu}(D^\varepsilon(0,1/2))$. Therefore, there is no way to get an estimate of the type of \eqref{estlem1model}, without introducing an additional corrector for the oscillations of the operator. The standard corrector $\chi^d$ solving \eqref{eqdefchi} needs to be adjusted since it introduces annoying oscillations on the boundary. We therefore introduce $w^\varepsilon=w^\varepsilon(x)\in W^{1,2}(D^\varepsilon(0,2))$ unique weak solution of
\begin{equation}\label{sysbulkcorr}
\left\{
\begin{array}{rll}
-\nabla\cdot A(x/\varepsilon)\nabla w^\varepsilon&=\nabla\cdot A(x/\varepsilon)\nabla(\chi^d(x/\varepsilon)\Theta(x',x_d/\varepsilon)),&x\in D^\varepsilon(0,2),\\
w^\varepsilon&=0,&x\in\partial D^\varepsilon(0,2),
\end{array}
\right. 
\end{equation}
for which we have the following lemma:

\begin{lem}[Dirichlet corrector]\label{bulkcorlem}
For all $1/2<\tau<1$, there exists $C_0'>0$ such that for all $\psi\in\mathcal C_{M_0}^{1,\omega}$, for all $A\in\mathcal A^{0,\nu_0}$, for all $0<\varepsilon<1$, the unique weak solution $w^\varepsilon\in W^{1,2}(D^\varepsilon(0,2))$ of \eqref{sysbulkcorr} satisfies the following estimate: for all $x\in D^\varepsilon(0,3/2)$,
\begin{equation}\label{estdircorlem}
|w^\varepsilon(x)|\leq \frac{C_0'\delta(x)^\tau}{\varepsilon^{\tau}},
\end{equation}
where $\delta(x):=x_d-\varepsilon\psi(x'/\varepsilon)$.
\end{lem}

\begin{proof}
The proof follows exactly the lines of the proof of Lemma \ref{bdarycorlem} for the boundary corrector.
\end{proof}

Let $0<\mu<1$ and $\kappa>0$ be fixed in what follows. The two usual auxiliary lemmas (improvement and iteration) now read:
\begin{lem}[improvement lemma]\label{lem1lip}
There exist $0<\varepsilon_0< 1,\ 0<\theta<1/8$, such that for all $\psi\in\mathcal C_{M_0}^{1,\nu_0}$, for all $A\in\mathcal A^{0,\nu_0}$, for all $0<\varepsilon<\varepsilon_0$, for all $f\in L^{d+\kappa}(D^\varepsilon(0,1))$, for all $F\in C^{0,\mu}(D^\varepsilon(0,1))$, for all $u^\varepsilon$ weak solution to 
\begin{equation}\label{sysoscueps1/2}
\left\{
\begin{array}{rll}
-\nabla\cdot A(x/\varepsilon)\nabla u^\varepsilon&=f+\nabla\cdot F,&x\in D^\varepsilon(0,1/2),\\
u^\varepsilon&=0,&x\in \Delta^\varepsilon(0,1/2),
\end{array}
\right.
\end{equation}
if
\begin{equation}\label{boundsufF}
\left\|u^\varepsilon\right\|_{L^\infty(D^\varepsilon(0,1/2))}\leq 1,\quad\|f\|_{L^{d+\kappa}(D^\varepsilon(0,1/2))}\leq \varepsilon_0,\quad\|F\|_{C^{0,\mu}(D^\varepsilon(0,1/2))}\leq \varepsilon_0,
\end{equation}
then 
\begin{multline}\label{estlem1lip}
\bigl\|u^\varepsilon(x)-\left(\overline{\partial_{x_d} u^\varepsilon}\right)_{0,\theta}\left\{x_d-\varepsilon(1-\Theta(x',x_d/\varepsilon))\chi^d(x/\varepsilon)-\varepsilon w^\varepsilon(x)\right.\bigr.\\
\bigl.\left.-\varepsilon\psi(x'/\varepsilon)\Theta(x',x_d/\varepsilon)-\varepsilon v^\varepsilon(x)\right\}\bigr\|_{L^\infty(D^\varepsilon(0,\theta))}\leq\theta^{1+\mu},
\end{multline}
where $v^\varepsilon$ is the boundary corrector solving \eqref{sysbdarycorr}.
\end{lem}

\begin{lem}[iteration lemma]\label{lem2lip}
Let $0<\varepsilon_0<1$ and $\theta>0$ as given by Lemma \ref{lem1lip}. There exists $C_1>0$, for all $k\in\mathbb N$, $k\geq 1$, for all $\varepsilon<\theta^{k-1}\varepsilon_0$, for all $\psi\in\mathcal C_{M_0}^{1,\nu_0}$, for all $A\in\mathcal A^{0,\nu_0}$, for all $f\in L^{d+\kappa}(D^\varepsilon(0,1))$, for all $F\in C^{0,\mu}(D^\varepsilon(0,1))$, for all $u^\varepsilon$ weak solution to \eqref{sysoscueps1/2}, if
\begin{equation*}
\left\|u^\varepsilon\right\|_{L^\infty(D^\varepsilon(0,1/2))}\leq 1,\quad\|f\|_{L^{d+\kappa}(D^\varepsilon(0,1/2))}\leq \varepsilon_0,\quad\|F\|_{C^{0,\mu}(D^\varepsilon(0,1/2))}\leq \varepsilon_0,
\end{equation*}
then there exist $a^\varepsilon_k\in \mathbb R$, $V^\varepsilon_k=V^\varepsilon_k(x)$ and $W^\varepsilon_k=W^\varepsilon_k(x)$ such that
\begin{equation*}
\begin{aligned}
|a^\varepsilon_k|&\leq (C_0/\theta)[1+\theta^\mu+\ldots\ \theta^{(k-1)\mu}],\\
|V^\varepsilon_k(x)|&\leq (C_0C_1/\theta)[1+\theta^\mu+\ldots\ \theta^{(k-1)\mu}]\frac{\delta(x)^\tau}{\varepsilon^\tau},\\
|W^\varepsilon_k(x)|&\leq (C_0'C_1/\theta)[1+\theta^\mu+\ldots\ \theta^{(k-1)\mu}]\frac{\delta(x)^\tau}{\varepsilon^\tau},
\end{aligned}
\end{equation*}
where $C_0$ (resp. $C_0'$) is the constant appearing in \eqref{estbdarycorlem} (resp. \eqref{estdircorlem}) and
\begin{multline*}
\bigl\|u^\varepsilon(x)-a^\varepsilon_k\left\{x_d-\varepsilon(1-\Theta(x',x_d/\varepsilon))\chi^d(x/\varepsilon)-\varepsilon\psi(x'/\varepsilon)\Theta(x',x_d/\varepsilon)\right\}\bigr.\\
\bigl.-\varepsilon W^\varepsilon_k(x)-\varepsilon V^\varepsilon_k(x)\bigr\|_{L^\infty(D^\varepsilon(0,\theta^k))}\leq\theta^{k(1+\mu)}.
\end{multline*}
\end{lem}

The proofs of Lemma \ref{lem1lip} and Lemma \ref{lem2lip} are now completely standard. The core argument has been carried out in the proofs of Lemma \ref{lem1} and Lemma \ref{lem2}. The key argument for the improvement is the convergence of a subsequence to a system with constant coefficients in a flat space. This convergence is ensured 
\begin{itemize}
\item on the one hand by the bounds \eqref{boundsufF}, which give weak $H^1$ and strong $C^{0,\sigma}$ compactness, 
\item and on the other hand by the homogenization result of Theorem \ref{theoweakcvhomo}. 
\end{itemize}
As far as Lemma \ref{lem2lip} and the actual proof of Theorem \ref{theoboundarylip} are concerned, the Dirichlet corrector $w^\varepsilon$ leads to the exact same difficulties as the boundary corrector $v^\varepsilon$.

\section{Asymptotics of Green and Poisson kernels}
\label{secasygreen}

\subsection{Large scale pointwise estimates}

Let $G^\varepsilon=G^\varepsilon(x,\tilde{x})$ be the Green kernel associated to the operator $-\nabla\cdot A(x/\varepsilon)\nabla$ and the oscillating domain $D^\varepsilon_+$. Recall that for all $\tilde{x}\in D^\varepsilon_+$, $G^\varepsilon(\cdot,\tilde{x})$ is a weak solution of
\begin{equation*}
\left\{
\begin{array}{rll}
-\nabla\cdot A(x/\varepsilon)\nabla G^\varepsilon(x,\tilde{x})&=\delta(x-\tilde{x})\Idd_N,&x_d>\varepsilon\psi(x'/\varepsilon),\\
G^\varepsilon(x,\tilde{x})&=0,&x_d=\varepsilon\psi(x'/\varepsilon),
\end{array}
\right.
\end{equation*}
where $\delta(\cdot)$ stands here for the Dirac measure supported at the point $0$. Existence and uniqueness of the Green kernel $G^\varepsilon$ is ensured by the results of Dong and Kim \cite{DongKimJDE09} in dimension $d=2$ and those of Hofmann and Kim \cite{HofKim07} in dimension $d\geq 3$. The Poisson kernel $P^\varepsilon$ is defined by for all $i,\ j\in\{1,\ldots N\}$, for all $x\in D^\varepsilon_+$, for all $\tilde{x}\in \Delta^\varepsilon$,
\begin{equation*}
\begin{aligned}
P^\varepsilon_{ij}(x,\tilde{x})&=-A^{\alpha\beta}_{kj}(\tilde{x}/\varepsilon)\partial_{\tilde{x}_\alpha}G^\varepsilon_{ik}(x,\tilde{x})n_\beta\\
&=-\left(A^*\right)^{\beta\alpha}_{jk}(\tilde{x}/\varepsilon)\partial_{\tilde{x}_\alpha}G^{*,\varepsilon}_{ki}(\tilde{x},x)n_\beta\\
&=-\left[A^{*,\beta\alpha}(\tilde{x}/\varepsilon)\partial_{\tilde{x}_\alpha}G^{*,\varepsilon}(\tilde{x},x)n_\beta\right]_{ji}\\
&=-\left[A^*(\tilde{x}/\varepsilon)\nabla_{\tilde{x}}G^{*,\varepsilon}(\tilde{x},x)\cdot n\right]^T_{ij}
\end{aligned}
\end{equation*}
where $G^{*,\varepsilon}$ is the Green kernel associated to the operator $L^{*,\varepsilon}=-\nabla\cdot A^*(x/\varepsilon)\nabla$ and the domain $D^\varepsilon_+$.

Our focus is on getting pointwise estimates on $G^\varepsilon(x,\tilde{x})$ and $P^\varepsilon(x,\tilde{x})$ for $|x-\tilde{x}|\gg 1$. This is now a routine procedure given that the uniform local interior and boundary H\"older and Lipschitz estimates \eqref{intestrescaledbis}, \eqref{intgradestrescaledbis}, \eqref{bdaryholdest} and \eqref{bdarylipest} hold. The estimates of Green's kernel are summarized in the following proposition.

\begin{prop}\label{propestgeps}
For all $d\geq 2$, there exists $C>0$, such that for all $\psi\in\mathcal C^{1,\nu_0}_{M_0}$, for all $A\in \mathcal A^{0,\nu_0}$, for all $\varepsilon>0$, for all $x,\ \tilde{x}\in D^\varepsilon_+$, we have:
\begin{equation*}
\begin{aligned}
|G^\varepsilon(x,\tilde{x})|&\leq \frac{C}{|x-\tilde{x}|^{d-2}}\min\left\{\frac{\dist(x,\Delta^\varepsilon)}{|x-\tilde{x}|},\frac{\dist(\tilde{x},\Delta^\varepsilon)}{|x-\tilde{x}|},\frac{\dist(x,\Delta^\varepsilon)\dist(\tilde{x},\Delta^\varepsilon)}{|x-\tilde{x}|^2}\right\},\\
|\nabla_1 G^\varepsilon(x,\tilde{x})|&\leq \frac{C}{|x-\tilde{x}|^{d-1}}\min\left\{1,\frac{\dist(\tilde{x},\Delta^\varepsilon)}{|x-\tilde{x}|}\right\},\\
|\nabla_2 G^\varepsilon(x,\tilde{x})|&\leq \frac{C}{|x-\tilde{x}|^{d-1}}\min\left\{1,\frac{\dist(x,\Delta^\varepsilon)}{|x-\tilde{x}|}\right\},\\
|\nabla_1\nabla_2 G^\varepsilon(x,\tilde{x})|&\leq \frac{C}{|x-\tilde{x}|^d}.
\end{aligned}
\end{equation*}
Notice that $C$ depends on $d$, $N$, $M_0$, $\lambda$ and $\nu_0$.
\end{prop}

Of course, by definition $\dist(x,\Delta^\varepsilon)\leq \delta(x)=x_d-\varepsilon\psi(x'/\varepsilon)$. Estimates involving $\delta$ are often very useful. We now turn to the estimates on Poisson's kernel.

\begin{prop}\label{propestpeps}
For all $d\geq 2$, there exists $C>0$, such that for all $\psi\in\mathcal C^{1,\nu_0}_{M_0}$, for all $A\in \mathcal A^{0,\nu_0}$, for all $\varepsilon>0$, for all $x\in D^\varepsilon_+$, for all $\tilde{x}\in \Delta^\varepsilon$, we have:
\begin{equation*}
\begin{aligned}
|P^\varepsilon(x,\tilde{x})|&\leq\frac{C\dist(x,\Delta^\varepsilon)}{|x-\tilde{x}|^d},\\
|\nabla_1 P^\varepsilon(x,\tilde{x})|&\leq\frac{C}{|x-\tilde{x}|^d}\left(1+\frac{\dist(x,\Delta^\varepsilon)}{|x-\tilde{x}|}\right).
\end{aligned}
\end{equation*}
Notice that $C$ depends on $d$, $N$, $M_0$, $\lambda$ and $\nu_0$.
\end{prop}

The proofs of the bounds in Proposition \ref{propestgeps} and \ref{propestpeps} are standard. The arguments follow the same scheme as that of Lemma \ref{lemestgreend=3} (see also \cite{alin} and \cite[Appendix A]{dgvnm2}). The detailed proofs are omitted.

\subsection{Asymptotic expansion}

The goal of this section is to compare the Green function $G^\varepsilon$ associated to the operator $-\nabla\cdot A(x/\varepsilon)\nabla$ and the oscillating domain $D^\varepsilon_+$ to the Green function $G^0$ associated to the homogenized constant coefficient operator $-\nabla\cdot \overline{A}\nabla$ and the flat domain $\mathbb R^d_+$.

\begin{theo}\label{theogeps-g0}
There exists $C>0$, such that for all $\psi\in C^{1,\nu_0}_{M_0}$, for all $A\in \mathcal A^{0,\nu_0}$, for all $\varepsilon>0$, for all $x,\ \tilde{x}\in D^\varepsilon_+$
\begin{equation*}
|G^\varepsilon(x,\tilde{x})-G^0(x,\tilde{x})|\leq \frac{C\varepsilon}{|x-\tilde{x}|^{d-1}}.
\end{equation*}
Notice that $C$ depends on $d$, $N$, $M_0$, $\lambda$ and $\nu_0$.
\end{theo}

The proof follows the presentation by Kenig, Lin and Shen \cite{kls12}, but differs in some technical aspects. See also \cite{alinLp} (whole space) and \cite{BLtail} (flat half-space) for similar arguments and results. The bound of Theorem \ref{theogeps-g0} follows from a duality argument. For $x_0\in\mathbb R^d$ and $f\in C^\infty_c(B(x_0,1))$, the key is to compare the solution $u^\varepsilon=u^\varepsilon(x)\in H^1_0(D^\varepsilon_+)$ of the boundary value problem
\begin{equation}\label{sysuepsdual}
\left\{
\begin{array}{rll}
-\nabla\cdot A(x/\varepsilon)\nabla u^\varepsilon&=f,&x_d>\varepsilon\psi(x'/\varepsilon),\\
u^\varepsilon&=0,&x_d=\varepsilon\psi(x'/\varepsilon),
\end{array}
\right.
\end{equation}
to the solution $u^0=u^0(x)\in C^\infty(\overline{D^\varepsilon_+})\cap W^{2,p}(D^\varepsilon_+)$, $1\leq p\leq \infty$ of the boundary value problem
\begin{equation}\label{sysu0dual}
\left\{
\begin{array}{rll}
-\nabla\cdot \overline{A}\nabla u^0&=f,&x_d>0,\\
u^0&=0,&x_d=0.
\end{array}
\right.
\end{equation}

The first of two lemmas is an Agmond-Miranda type of maximum principle for systems and is of independant interest. Of course, this result is well-known for scalar equations i.e. when $N=1$.

\begin{lem}[maximum principle for systems]\label{lemLinftygLinfty}
There exists $C>0$, such that for all $\psi\in C^{1,\nu_0}_{M_0}$, for all $A\in \mathcal A^{0,\nu_0}$, for all $\varepsilon>0$, for all $g\in L^\infty(\Delta^\varepsilon(0,2))$, for all weak solution $v^\varepsilon=v^\varepsilon(x)\in L^2(D^\varepsilon(0,1))$ to 
\begin{equation*}
\left\{\begin{array}{rll}
-\nabla\cdot A(x/\varepsilon)\nabla v^\varepsilon&=0,&x\in D^\varepsilon(0,1),\\
v^\varepsilon&=g,&x\in \Delta^\varepsilon(0,1),
\end{array}
\right.
\end{equation*}
we have
\begin{equation*}
\|v^\varepsilon\|_{L^\infty(D^\varepsilon(0,1/2))}\leq C\left\{\|v^\varepsilon\|_{L^2(D^\varepsilon(0,1))}+\|g\|_{L^\infty(\Delta^\varepsilon(0,2))}\right\}.
\end{equation*}
Notice that $C$ depends on $d$, $N$, $M_0$, $\lambda$ and $\nu_0$.
\end{lem}

\begin{proof}[Proof of Lemma \ref{lemLinftygLinfty}]
Let $\vartheta\in C^\infty_c((-3/2,3/2)^{d-1})$ be a cut-off function such that $0\leq \vartheta\leq 1$ and $\vartheta\equiv 1$ on $(-1,1)^{d-1}$. Let $w^\varepsilon\in H^1_0(D^\varepsilon(0,2))$ be the unique weak solution to
\begin{equation*}
\left\{\begin{array}{rll}
-\nabla\cdot A(x/\varepsilon)\nabla w^\varepsilon&=0,&x\in D^\varepsilon(0,2),\\
w^\varepsilon&=g\vartheta(x'),&x\in \Delta^\varepsilon(0,2),\\
w^\varepsilon&=0,&x\in\partial D^\varepsilon(0,2)\setminus \Delta^\varepsilon(0,2).
\end{array}
\right.
\end{equation*}
Of course, $z^\varepsilon:=v^\varepsilon-w^\varepsilon$ solves 
\begin{equation*}
\left\{\begin{array}{rll}
-\nabla\cdot A(x/\varepsilon)\nabla z^\varepsilon&=0,&x\in D^\varepsilon(0,1),\\
z^\varepsilon&=0,&x\in \Delta^\varepsilon(0,1),
\end{array}
\right.
\end{equation*}
so that by \eqref{bdaryholdest}, we get
\begin{equation*}
\|z^\varepsilon\|_{L^\infty(D^\varepsilon(0,1/2))}\leq C\|z^\varepsilon\|_{L^2(D^\varepsilon(0,1))}.
\end{equation*}
Furthermore, letting $\widetilde{P}^\varepsilon$ denote Poisson's kernel associated to the operator $-\nabla\cdot A(x/\varepsilon)\nabla$ and the domain $D^\varepsilon(0,2)$, we have for all $x\in D^\varepsilon(0,7/4)$ and $\tilde{x}\in \Delta^\varepsilon(0,7/4)$,
\begin{equation*}
|\widetilde{P}^\varepsilon(x,\tilde{x})|\leq \frac{C\delta(x)}{|x-\tilde{x}|^d}
\end{equation*}
with a constant $C>0$ uniform in $\varepsilon$. This estimate for Poisson's kernel in the bounded domain $D^\varepsilon(0,2)$ is proved in the same way as the estimates of Proposition \ref{propestpeps} for Poisson's kernel in the half-space. It follows from the use of local Lipschitz estimates, in a fashion similar to Lemma \ref{lemestgreend=3} where we rely on local H\"older estimates. For all $x\in D^\varepsilon(0,1)$
\begin{equation*}
\begin{aligned}
|w^\varepsilon(x)|&=\left|\int_{\Delta^\varepsilon(0,3/2)}\widetilde{P}^\varepsilon(x,\tilde{x})g(x)\vartheta(x')dx\right|\\
&\leq C\|g\|_{L^\infty(D^\varepsilon(0,2))}\int_{\mathbb R^{d-1}}\frac{x_d}{(x_d^2+|\tilde{x}'|^2)^d}d\tilde{x}'\leq C\|g\|_{L^\infty(D^\varepsilon(0,2))}.\qedhere
\end{aligned}
\end{equation*}
\end{proof}

\begin{lem}\label{lemueps-u0dual}
For all $d<p\leq\infty$, there exists $C>0$, such that for all $\psi\in C^{1,\nu_0}_{M_0}$, for all $A\in \mathcal A^{0,\nu_0}$, for all $\varepsilon>0$, for all $x_0\in\mathbb R^d$, $f\in C^\infty_c(B(x_0,1))$, the solutions $u^\varepsilon$ of \eqref{sysuepsdual} and $u^0$ of \eqref{sysu0dual} satisfy the estimate
\begin{multline*}
\|u^\varepsilon-u^0\|_{L^\infty(D^\varepsilon(0,1/2))}\leq C\left\{\|u^\varepsilon-u^0\|_{L^2(D^\varepsilon(0,1))}+\varepsilon\|f\|_{L^p(B(x_0,1))}\right.\\
\left.+\varepsilon\|\nabla u^0\|_{L^\infty(D^\varepsilon(0,2))}+\varepsilon\|\nabla^2 u^0\|_{L^p(D^\varepsilon(0,2))}\right\}.
\end{multline*}
Notice that $C$ depends on $d$, $N$, $M_0$, $\lambda$, $\nu_0$ and $p$.
\end{lem}

\begin{rem}[rescaled estimate]
For all $d<p\leq\infty$, there exists $C>0$, such that for all $\psi\in C^{1,\nu_0}_{M_0}$, for all $A\in \mathcal A^{0,\nu_0}$, for all $\varepsilon>0$, for all $r>0$ and $x_0\in\mathbb R^d$, $f\in C^\infty_c(B(x_0,r))$, the solutions $u^\varepsilon$ of \eqref{sysuepsdual} and $u^0$ of \eqref{sysu0dual} satisfy the estimate
\begin{multline}\label{rescestueps-u0dual}
\|u^\varepsilon-u^0\|_{L^\infty(D^\varepsilon(0,r/2))}\leq C\left\{r^{-d/2}\|u^\varepsilon-u^0\|_{L^2(D^\varepsilon(0,r))}+\varepsilon r^{1-d/p}\|f\|_{L^p(B(x_0,r))}\right.\\
\left.+\varepsilon\|\nabla u^0\|_{L^\infty(D^\varepsilon(0,2r))}+\varepsilon r^{1-d/p}\|\nabla^2 u^0\|_{L^p(D^\varepsilon(0,2r))}\right\}.
\end{multline}
Let us enter into some details of the proof of this crucial estimate. Let $r>0$ be fixed. We consider $\tilde{u}^\varepsilon=\tilde{u}^\varepsilon=u^\varepsilon(rz)$ solving
\begin{equation*}
\left\{
\begin{array}{rll}
-\nabla\cdot A(rz/\varepsilon)\nabla \tilde{u}^\varepsilon&=r^2f(rz)=:\tilde{f}(z),&z_d>(\varepsilon/r)\psi(rz'/\varepsilon),\\
\tilde{u}^\varepsilon&=0,&z_d=(\varepsilon/r)\psi(rz'/\varepsilon),
\end{array}
\right.
\end{equation*}
and $\tilde{u}^0=\tilde{u}^0(z)=u^0(rz)$ solving
\begin{equation*}
\left\{
\begin{array}{rll}
-\nabla\cdot \overline{A}\nabla\tilde{u}^0&=r^2f(rz)=:\tilde{f}(z),&z_d>0,\\
\tilde{u}^0&=0,&z_d=0.
\end{array}
\right.
\end{equation*}
The critical point is that the estimate of Lemma \ref{lemueps-u0dual} is uniform in $\varepsilon>0$, so that we can use it for $\tilde{\varepsilon}=\varepsilon/r$:
\begin{multline*}
\|\tilde{u}^\varepsilon-\tilde{u}^0\|_{L^\infty(D^{\varepsilon/r}(0,1/2))}\leq C\left\{\|\tilde{u}^\varepsilon-\tilde{u}^0\|_{L^2(D^{\varepsilon/r}(0,1))}+\varepsilon/r\|\tilde{f}\|_{L^p(B(x_0/r,1))}\right.\\
\left.+\varepsilon/r\|\nabla\tilde{u}^0\|_{L^\infty(D^{\varepsilon/r}(0,2))}+\varepsilon/r\|\nabla^2\tilde{u}^0\|_{L^p(D^{\varepsilon/r}(0,2))}\right\}.
\end{multline*}
Rescaling, we finally get \eqref{rescestueps-u0dual}.
\end{rem}

\begin{proof}[Proof of Lemma \ref{lemueps-u0dual}]
First of all, let us compute $\|u^0\|_{L^\infty(\Delta^\varepsilon)}$. For all $x'\in\mathbb R^{d-1}$, letting $x:=(x',\varepsilon\psi(x'/\varepsilon))$,
\begin{equation*}
u^0(x',\varepsilon\psi(x'/\varepsilon))=\int_{\mathbb R^d_+}G^0(x,\tilde{x})f(\tilde{x})d\tilde{x}
\end{equation*}
so that using the bound 
\begin{equation*}
|G^0(x,\tilde{x})|\leq \frac{Cx_d}{|x-\tilde{x}|^{d-1}}
\end{equation*}
valid for all $x,\ \tilde{x}\in\mathbb R^d_+$, we get for $1\leq p'<d$ and $d<p\leq\infty$
\begin{equation*}
\begin{aligned}
|u^0(x',\varepsilon\psi(x'/\varepsilon))|&\leq C\varepsilon\int_{\mathbb R^d}\frac{1}{|x-\tilde{x}|^{d-1}}|f(\tilde{x})|d\tilde{x}\\
&\leq C\varepsilon\left[\int_{B(0,1)}\frac{1}{|\tilde{x}|^{d-1}}|f(x-\tilde{x})|d\tilde{x}+\int_{B(0,1)^c}\frac{1}{|\tilde{x}|^{d-1}}|f(x-\tilde{x})|d\tilde{x}\right]\\
&\leq C\varepsilon\left[\|f\|_{L^{p'}(B(x_0,1))}+\|f\|_{L^{p}(B(x_0,1))}\right]\\
&\leq C\varepsilon\|f\|_{L^{p}(B(x_0,1))}.
\end{aligned}
\end{equation*}

Let $\vartheta\in C^\infty_c((-3/2,3/2)^d)$ be a cut-off such that $0\leq \vartheta\leq 1$ and $\vartheta\equiv 1$ on $(-5/4,5/4)^d$. Now, for $x\in D^\varepsilon_+$ consider 
\begin{equation*}
r^\varepsilon(x):=u^\varepsilon(x)-u^0(x)-\varepsilon\chi(x/\varepsilon)\cdot\nabla u^0=w^\varepsilon(x)+z^\varepsilon(x),
\end{equation*}
where $w^\varepsilon\in H^1_0(D^\varepsilon(0,2))$ is the unique weak solution to
\begin{equation*}
\left\{\begin{array}{rll}
-\nabla\cdot A(x/\varepsilon)\nabla w^\varepsilon&=-\vartheta(x)\nabla\cdot A(x/\varepsilon)\nabla r^\varepsilon,&x\in D^\varepsilon(0,2),\\
w^\varepsilon&=0,&x\in\partial D^\varepsilon(0,2).
\end{array}
\right.
\end{equation*}
Of course, $z^\varepsilon$ then satisfies
\begin{equation*}
\left\{\begin{array}{rll}
-\nabla\cdot A(x/\varepsilon)\nabla z^\varepsilon&=0,&x\in D^\varepsilon(0,1),\\
z^\varepsilon&=-u^0(x)-\varepsilon\chi(x/\varepsilon)\cdot\nabla u^0,&x\in\Delta^\varepsilon(0,1).
\end{array}
\right.
\end{equation*}
Lemma \ref{lemLinftygLinfty} now implies
\begin{equation*}
\begin{aligned}
\|z^\varepsilon\|_{L^\infty(D^\varepsilon(0,1/2))}&\leq C\left\{\|z^\varepsilon\|_{L^2(D^\varepsilon(0,1))}+\|u^0(x)+\varepsilon\chi(x/\varepsilon)\cdot\nabla u^0\|_{L^\infty(\Delta^\varepsilon(0,2))}\right\}\\
&\leq C\left\{\|u^\varepsilon-u^0\|_{L^2(D^\varepsilon(0,1))}+\|w^\varepsilon\|_{L^2(D^\varepsilon(0,1))}+\varepsilon\|f\|_{L^{p}(B(x_0,1))}+\varepsilon\|\nabla u^0\|_{L^\infty(D^\varepsilon(0,2))}\right\}\\
&\leq C\left\{\|u^\varepsilon-u^0\|_{L^2(D^\varepsilon(0,1))}+\|w^\varepsilon\|_{L^\infty(D^\varepsilon(0,1))}+\varepsilon\|f\|_{L^{p}(B(x_0,1))}+\varepsilon\|\nabla u^0\|_{L^\infty(D^\varepsilon(0,2))}\right\}.
\end{aligned}
\end{equation*}
It remains to bound $\|w^\varepsilon\|_{L^\infty(D^\varepsilon(0,1))}$. For this, we rely on Green's representation formula, using Green's kernel $\widetilde{G}^\varepsilon$ associated to the operator $-\nabla\cdot A(x/\varepsilon)\nabla$ and to the domain $D^\varepsilon(0,2)$. This gives, for all $x\in D^\varepsilon(0,1)$,
\begin{equation*}
\begin{aligned}
w^\varepsilon(x)&=-\int_{D^\varepsilon(0,2)}\widetilde{G}^\varepsilon(x,\tilde{x})\vartheta(\tilde{x})\nabla\cdot A(\tilde{x}/\varepsilon)\nabla r^\varepsilon(\tilde{x})d\tilde{x}\\
&=-\int_{D^\varepsilon(0,2)\cap(-3/2,3/2)^d}\widetilde{G}^\varepsilon(x,\tilde{x})\vartheta(\tilde{x})\nabla\cdot A(\tilde{x}/\varepsilon)\nabla r^\varepsilon(\tilde{x})d\tilde{x}.
\end{aligned}
\end{equation*}
Following \cite[Proposition 2.2]{kls12}, there exists $\Phi=\Phi(y)\in L^\infty(\mathbb T^d;\mathbb R^{3dN})$ such that
\begin{equation*}
-\nabla\cdot A(\tilde{x}/\varepsilon)\nabla r^\varepsilon(\tilde{x})=\varepsilon\partial_\alpha\left\{\left[\Phi^{\alpha\beta\gamma}(\tilde{x}/\varepsilon)+A^{\alpha\beta}(\tilde{x}/\varepsilon)\chi^\gamma(\tilde{x}/\varepsilon)\right]\partial_\beta\partial\gamma u^0(\tilde{x})\right\}.
\end{equation*}
Therefore, integrating by parts we get
\begin{equation*}
\begin{aligned}
&w^\varepsilon(x)=-\varepsilon\int_{D^\varepsilon(0,2)\cap(-3/2,3/2)^d}\partial_{\tilde{x}_\alpha}\left(\widetilde{G}^\varepsilon(x,\tilde{x})\vartheta(\tilde{x})\right)\left[\Phi^{\alpha\beta\gamma}(\tilde{x}/\varepsilon)+A^{\alpha\beta}(\tilde{x}/\varepsilon)\chi^\gamma(\tilde{x}/\varepsilon)\right]\partial_\beta\partial\gamma u^0(\tilde{x})d\tilde{x}\\
&=-\varepsilon\int_{D^\varepsilon(0,2)\cap(-3/2,3/2)^d}\left(\partial_{\tilde{x}_\alpha}\widetilde{G}^\varepsilon(x,\tilde{x})\right)\vartheta(\tilde{x})\left[\Phi^{\alpha\beta\gamma}(\tilde{x}/\varepsilon)+A^{\alpha\beta}(\tilde{x}/\varepsilon)\chi^\gamma(\tilde{x}/\varepsilon)\right]\partial_\beta\partial\gamma u^0(\tilde{x})d\tilde{x}\\
&\qquad-\varepsilon\int_{D^\varepsilon(0,2)\cap\left((-3/2,3/2)^d\setminus(-5/4,5/4)^d\right)}\widetilde{G}^\varepsilon(x,\tilde{x})\partial_{\tilde{x}_\alpha}\vartheta(\tilde{x})\left[\Phi^{\alpha\beta\gamma}(\tilde{x}/\varepsilon)+A^{\alpha\beta}(\tilde{x}/\varepsilon)\chi^\gamma(\tilde{x}/\varepsilon)\right]\partial_\beta\partial\gamma u^0(\tilde{x})d\tilde{x}
\end{aligned}
\end{equation*}
so that
\begin{equation*}
\begin{aligned}
|w^\varepsilon(x)|&\leq C\varepsilon\int_{D^\varepsilon(0,2)\cap(-3/2,3/2)^d}\frac{1}{|x-\tilde{x}|^{d-1}}|\partial_\beta\partial\gamma u^0(\tilde{x})|d\tilde{x}\\
&\qquad+C\varepsilon\int_{D^\varepsilon(0,2)\cap\left((-3/2,3/2)^d\setminus(-5/4,5/4)^d\right)}|\partial_\beta\partial\gamma u^0(\tilde{x})|d\tilde{x}\\
&\leq C\varepsilon\|\nabla^2 u^0\|_{L^p(D^\varepsilon(0,2))}.
\end{aligned}
\end{equation*}
Finally,
\begin{multline*}
\|u^\varepsilon-u^0\|_{L^\infty(D^\varepsilon(0,1/2))}\leq C\left\{\|u^\varepsilon-u^0\|_{L^2(D^\varepsilon(0,1))}+\varepsilon\|f\|_{L^p(B(x_0,1))}\right.\\
\left.+\varepsilon\|\nabla u^0\|_{L^\infty(D^\varepsilon(0,2))}+\varepsilon\|\nabla^2 u^0\|_{L^p(D^\varepsilon(0,2))}\right\}.\qedhere
\end{multline*}
\end{proof}

\begin{rem}[interior estimate]
Following the same scheme, one can prove an interior estimate. For all $d<p\leq\infty$, there exists $C>0$, such that for all $A\in \mathcal A^{0,\nu_0}$, for all $\varepsilon>0$, for all $r>0$ and $x_0\in\mathbb R^d$, $f\in C^\infty_c(B(x_0,r))$, the solutions $u^\varepsilon$ of 
\begin{equation*}
-\nabla\cdot A(x/\varepsilon)\nabla u^\varepsilon=f,\qquad x\in\mathbb R^d 
\end{equation*}
and $u^0$ of 
\begin{equation*}
-\nabla\cdot \overline{A}\nabla u^0=f,\qquad x\in\mathbb R^d
\end{equation*}
satisfy the estimate
\begin{multline}\label{rescestueps-u0dualint}
\|u^\varepsilon-u^0\|_{L^\infty(B(0,r/2))}\leq C\left\{r^{-d/2}\|u^\varepsilon-u^0\|_{L^2(B(0,r))}+\varepsilon r^{1-d/p}\|\nabla^2 u^0\|_{L^p(B(0,2r))}\right\}.
\end{multline}
\end{rem}

\begin{proof}[Proof of Theorem \ref{theogeps-g0}]
Let $\varepsilon>0$ be fixed. Let $x_0,\ \tilde{x}_0\in D^\varepsilon_+$, $r:=|x_0-\tilde{x}_0|/8$ and $f\in C^\infty_c(B(\tilde{x}_0,r))$. Either $\dist(x_0,\Delta^\varepsilon)<r/2$, in which case we rely on the boundary estimate \eqref{rescestueps-u0dual}, or $\dist(x_0,\Delta^\varepsilon)\geq r/2$, in which case we resort to the interior estimate \eqref{rescestueps-u0dualint}. Let $\bar{x}_0\in \Delta^\varepsilon$ such that $\dist(x_0,\Delta^\varepsilon)=|x_0-\bar{x}_0|$.

Assume that $\dist(x_0,\Delta^\varepsilon)<r/2$. Thanks to estimate \eqref{rescestueps-u0dual}, we have
\begin{equation*}
\begin{aligned}
|u^\varepsilon(x_0)-u^0(x_0)|&\leq \|u^\varepsilon-u^0\|_{L^\infty(D^\varepsilon(\bar{x}_0',r/2))}\\
&\leq C\left\{r^{-d/2}\|u^\varepsilon-u^0\|_{L^2(D^\varepsilon(\bar{x}_0',r))}+\varepsilon r^{1-d/p}\|f\|_{L^p(B(\tilde{x}_0,r))}\right.\\
&\qquad\left.+\varepsilon\|\nabla u^0\|_{L^\infty(D^\varepsilon(\bar{x}_0',2r))}+\varepsilon r^{1-d/p}\|\nabla^2 u^0\|_{L^p(D^\varepsilon(\bar{x}_0',2r))}\right\}\\
&\leq C\left\{r^{-d/2}\|u^\varepsilon-u^0\|_{L^2(D^\varepsilon(\bar{x}_0',r))}+\varepsilon r^{1-d/p}\|f\|_{L^p(B(\tilde{x}_0,r))}\right.\\
&\qquad\left.+\varepsilon\|\nabla u^0\|_{L^\infty(\mathbb R^d_+)}+\varepsilon r^{1-d/p}\|\nabla^2 u^0\|_{L^p(\mathbb R^d_+)}\right\}.
\end{aligned}
\end{equation*}
Now, the classical estimates of \cite{adn1,adn2} imply
\begin{equation*}
\begin{aligned}
\|\nabla u^0\|_{L^\infty(\mathbb R^d_+)}&\leq Cr^{1-d/p}\|f\|_{L^p(B(\tilde{x}_0,r))},\qquad\mbox{for any}\ p>d,\\
\|\nabla^2 u^0\|_{L^p(\mathbb R^d_+)}&\leq C\|f\|_{L^p(B(\tilde{x}_0,r))},\qquad\mbox{for any}\ 1<p<\infty.
\end{aligned}
\end{equation*}
It remains to handle $\|u^\varepsilon-u^0\|_{L^2(D^\varepsilon(\bar{x}_0',r))}$. We have
\begin{equation*}
\|u^\varepsilon-u^0\|_{L^2(D^\varepsilon(\bar{x}_0',r))}\leq \|u^\varepsilon-u^0-\varepsilon\chi(x/\varepsilon)\cdot\nabla u^0\|_{L^2(D^\varepsilon(\bar{x}_0',r))}+\varepsilon\|\chi(x/\varepsilon)\cdot\nabla u^0\|_{L^2(D^\varepsilon(\bar{x}_0',r))}.
\end{equation*}
On the one hand, 
\begin{multline*}
\|\chi(x/\varepsilon)\cdot\nabla u^0\|_{L^2(D^\varepsilon(\bar{x}_0',r))}\leq C\|\nabla u^0\|_{L^2(D^\varepsilon(\bar{x}_0',r))}\leq Cr^{d/2}\|\nabla u^0\|_{L^\infty(D^\varepsilon(\bar{x}_0',r))}\\
\leq Cr^{1-d/p+d/2}\|f\|_{L^p(B(\tilde{x}_0,r))}.
\end{multline*}
On the other hand, let us decompose 
\begin{equation*}
r^\varepsilon(x):=u^\varepsilon(x)-u^0(x)-\varepsilon\chi(x/\varepsilon)\cdot\nabla u^0(x)=w^\varepsilon+z^\varepsilon,
\end{equation*}
where $w^\varepsilon\in H^1_0(D^\varepsilon_+)$ is the unique weak solution to
\begin{equation*}
\left\{\begin{array}{rll}
-\nabla\cdot A(x/\varepsilon)\nabla w^\varepsilon&=-\nabla\cdot A(x/\varepsilon)\nabla r^\varepsilon,&x\in D^\varepsilon_+,\\
w^\varepsilon&=0,&x\in\Delta^\varepsilon.
\end{array}
\right.
\end{equation*}
Since by \cite[Proposition 2.2]{kls12}
\begin{equation*}
\nabla\cdot A(x/\varepsilon)\nabla r^\varepsilon=\varepsilon\nabla\cdot\left\{\left(\Phi(x/\varepsilon)+A(x/\varepsilon)\chi(x/\varepsilon)\right)\cdot\nabla^2u^0\right\},
\end{equation*}
a standard energy estimate yields
\begin{equation*}
\|\nabla w^\varepsilon\|_{L^2(D^\varepsilon_+)}\leq \varepsilon\|\left(\Phi(x/\varepsilon)+A(x/\varepsilon)\chi(x/\varepsilon)\right)\cdot\nabla^2u^0\|_{L^2(D^\varepsilon_+)}\leq C\varepsilon\|\nabla^2u^0\|_{L^2(D^\varepsilon_+)}.
\end{equation*}
Now,
\begin{equation*}
\|\nabla^2u^0\|_{L^2(D^\varepsilon_+)}\leq C\|f\|_{L^2(B(\tilde{x}_0,r))}\leq Cr^{d/2-d/p}\|f\|_{L^p(B(\tilde{x}_0,r))}
\end{equation*}
for $p>d\geq 2$,
and by the Poincar\'e inequality 
\begin{equation*}
\|w^\varepsilon\|_{L^2(D^\varepsilon(\bar{x}_0',r))}\leq Cr\|\nabla w^\varepsilon\|_{L^2(D^\varepsilon(\bar{x}_0',r))}\leq Cr\|\nabla w^\varepsilon\|_{L^2(D^\varepsilon_+)}.
\end{equation*}
Combining these inequalities leads to
\begin{equation*}
\|w^\varepsilon\|_{L^2(D^\varepsilon(\bar{x}_0',r))}\leq C\varepsilon r^{1-d/p+d/2}\|f\|_{L^p(B(\tilde{x}_0,r))}.
\end{equation*}
Estimating $z^\varepsilon$ is done through the representation via Poisson's kernel $P^\varepsilon$, since $z^\varepsilon$ is the unique weak solution to 
\begin{equation*}
\left\{\begin{array}{rll}
-\nabla\cdot A(x/\varepsilon)\nabla z^\varepsilon&=0,&x\in D^\varepsilon_+,\\
z^\varepsilon&=-u^0(x)-\varepsilon\chi(x/\varepsilon)\cdot\nabla u^0(x),&x\in\Delta^\varepsilon.
\end{array}
\right.
\end{equation*}
We have for all $x\in D^\varepsilon_+$
\begin{equation*}
|z^\varepsilon(x)|\leq \int_{\Delta^\varepsilon}|P^\varepsilon(x,\tilde{x})||z^\varepsilon(\tilde{x})|d\tilde{x}\leq C\|z^\varepsilon\|_{L^\infty(\Delta^\varepsilon)}.
\end{equation*}
Thus,
\begin{multline*}
\|z^\varepsilon\|_{L^2(D^\varepsilon(\bar{x}_0,r))}\leq Cr^{d/2}\|z^\varepsilon\|_{L^\infty(D^\varepsilon)}\leq Cr^{d/2}\|z^\varepsilon\|_{L^\infty(\Delta^\varepsilon_+)}\\
\leq C\varepsilon r^{d/2}\left\{r^{1-d/p}\|f\|_{L^{p}(B(\tilde{x}_0,r))}+\|\nabla u^0\|_{L^\infty(\mathbb R^d_+))}\right\}\leq C\varepsilon r^{1-d/p+d/2}\|f\|_{L^{p}(B(\tilde{x}_0,r))}.
\end{multline*}
Gathering all the bounds we end up with
\begin{equation*}
|u^\varepsilon(x_0)-u^0(x_0)|\leq C\varepsilon r^{1-d/p}\|f\|_{L^{p}(B(\tilde{x}_0,r))}.
\end{equation*}

Assuming that $\dist(x_0,\Delta^\varepsilon)\geq r/2$ we get using the interior estimate \eqref{rescestueps-u0dual}
\begin{equation*}
|u^\varepsilon(x_0)-u^0(x_0)|\leq C\|u^\varepsilon-u^0\|_{L^\infty(B(x_0,r/4))}\leq C\varepsilon r^{1-d/p}\|f\|_{L^{p}(B(\tilde{x}_0,r))}.
\end{equation*}
Notice that this estimate is far easier to establish, because we do not have to deal with boundary terms.

By the representation formula in terms of Green kernels, we get
\begin{equation*}
u^\varepsilon(x_0)-u^0(x_0)=\int_{D^\varepsilon_+}\left[G^\varepsilon(x_0,\tilde{x})-G^0(x_0,\tilde{x})\right]f(\tilde{x})d\tilde{x}+\int_{0\leq \tilde{x}_d\leq\varepsilon\psi(\tilde{x}'/\varepsilon)}G^0(x_0,\tilde{x})f(\tilde{x})d\tilde{x}.
\end{equation*}
We deal with the remainder term, in a fashion similar to what we have done on $u^0(x',\varepsilon\psi(x'/\varepsilon))$ above (cf. proof of Lemma \ref{lemueps-u0dual}). Using the bound 
\begin{equation*}
|G^0(x_0,\tilde{x})|\leq \frac{C\tilde{x}_d}{|x_0-\tilde{x}|^{d-1}},
\end{equation*}
we obtain
\begin{equation*}
\left|\int_{0\leq \tilde{x}_d\leq\varepsilon\psi(\tilde{x}'/\varepsilon)}G^0(x_0,\tilde{x})f(\tilde{x})d\tilde{x}\right|\leq C\varepsilon r^{1-d/p}\|f\|_{L^{p}(B(\tilde{x}_0,r))}.
\end{equation*}
Therefore, for all $f\in C^\infty_c(B(\tilde{x}_0,r))$,
\begin{equation*}
\left|\int_{D^\varepsilon_+}\left[G^\varepsilon(x_0,\tilde{x})-G^0(x_0,\tilde{x})\right]f(\tilde{x})d\tilde{x}\right|\leq C\varepsilon r^{1-d/p}\|f\|_{L^{p}(B(\tilde{x}_0,r))},
\end{equation*}
which implies by duality that
\begin{equation}\label{dualestgeps-g0p'}
\left(\int_{B(\tilde{x}_0,r)\cap D^\varepsilon_+}\left|G^\varepsilon(x_0,\tilde{x})-G^0(x_0,\tilde{x})\right|^{p'}d\tilde{x}\right)^{1/p'}\leq C\varepsilon r^{1-d/p}
\end{equation}
for $1/p'+1/p=1$.

Finally, relying on the fact that 
\begin{equation*}
-\nabla\cdot A^*(\tilde{x}/\varepsilon)\nabla G^\varepsilon(x_0,\tilde{x})=-\nabla\cdot \overline{A}^*\nabla G^0(x_0,\tilde{x}),
\end{equation*}
we can apply the same type of arguments to get our target bound on $|G^\varepsilon(x_0,\tilde{x}_0)-G^0(x_0,\tilde{x}_0)|$. The rescaled boundary estimate \eqref{rescestueps-u0dual}, slightly modified to take into account that $G^0(x_0,\tilde{x})$ does not vanish on the oscillating boundary $\Delta^\varepsilon$, reads
\begin{multline*}
|G^\varepsilon(x_0,\tilde{x}_0)-G^0(x_0,\tilde{x}_0)|\leq C\Biggl\{r^{-d/2}\left(\int_{B(\tilde{x}_0,r)\cap D^\varepsilon_+}\left|G^\varepsilon(x_0,\tilde{x})-G^0(x_0,\tilde{x})\right|^2\right)^{1/2}\Biggr.\\
+C\|G^0(x_0,\tilde{x})\|_{L^\infty(\overline{B(\tilde{x}_0,2r)}\cap\Delta^\varepsilon)}+C\varepsilon\|\nabla_2G^0(x_0,\cdot)\|_{L^\infty(B(\tilde{x}_0,2r)\cap D^\varepsilon_+)}\\
\Biggl.+C\varepsilon r^{1-d/p}\|\nabla_2^2G^0(x_0,\cdot)\|_{L^p(B(\tilde{x}_0,2r)\cap D^\varepsilon_+)}\Biggr\}.
\end{multline*}
We have,
\begin{equation*}
\|G^0(x_0,\tilde{x})\|_{L^\infty(\overline{B(\tilde{x}_0,2r)}\cap\Delta^\varepsilon)}\leq C\varepsilon\left\|\frac{1}{|x_0-\tilde{x}|^{d-1}}\right\|_{L^\infty(\overline{B(\tilde{x}_0,2r)}\cap\Delta^\varepsilon)}\leq C\varepsilon r^{1-d},
\end{equation*}
similarly
\begin{equation*}
\|\nabla_2G^0(x_0,\cdot)\|_{L^\infty(B(\tilde{x}_0,2r)\cap D^\varepsilon_+)}\leq Cr^{1-d},
\end{equation*}
and 
\begin{equation*}
r^{-d/p}\|\nabla_2^2G^0(x_0,\cdot)\|_{L^p(B(\tilde{x}_0,2r)\cap D^\varepsilon_+)}\leq Cr^{-2}\|G^0(x_0,\cdot)\|_{L^\infty(B(\tilde{x}_0,4r)}\leq Cr^{-d}.
\end{equation*}
Consequently, the latter combined with the dual estimate \eqref{dualestgeps-g0p'} gives
\begin{equation*}
|G^\varepsilon(x_0,\tilde{x}_0)-G^0(x_0,\tilde{x}_0)|\leq C\varepsilon r^{1-d}\leq\frac{C\varepsilon}{|x_0-\tilde{x}|^{d-1}},
\end{equation*}
which concludes the proof.
\end{proof}

\begin{cor}\label{cortheocompgreen}
For all $1<p<d$ and $1/q=1/p-1/d$, or $p>d$ and $q=\infty$, there exists $C>0$, such that for all $\psi\in C^{1,\nu_0}_{M_0}$, for all $A\in \mathcal A^{0,\nu_0}$, for all $\varepsilon>0$, for all $f\in L^p(\mathbb R^d)$, for all $u^\varepsilon=u^\varepsilon(x)$ weak solution of \eqref{sysuepsdual} and $u^0=u^0(x)$ weak solution of \eqref{sysu0dual},
\begin{equation*}
\|u^\varepsilon-u^0\|_{L^q(D^\varepsilon_+)}\leq C\varepsilon\|f\|_{L^p(D^\varepsilon_+)}.
\end{equation*}
\end{cor}

\begin{proof}[Proof of Corollary \ref{cortheocompgreen}]
For all $x\in D^\varepsilon_+$,
\begin{equation}\label{decompueps-u0}
u^\varepsilon(x)-u^0(x)=\int_{D^\varepsilon_+}\left[G^\varepsilon(x,\tilde{x})-G^0(x,\tilde{x})\right]f(\tilde{x})d\tilde{x}+\int_{0<\tilde{x}_d<\varepsilon\psi(\tilde{x}'/\varepsilon)}G^0(x,\tilde{x})f(\tilde{x})d\tilde{x}.
\end{equation}
On the one hand, by Theorem \ref{theogeps-g0} and well-known estimates for fractional integrals, we get for $1<p<d$ and $1/q=1/p-1/d$,
\begin{equation*}
\begin{aligned}
\left\|\int_{D^\varepsilon_+}\left[G^\varepsilon(x,\tilde{x})-G^0(x,\tilde{x})\right]f(\tilde{x})d\tilde{x}\right\|_{L^q(D^\varepsilon_+)}&\leq C\varepsilon \left\|\int_{D^\varepsilon_+}\frac{|f(\tilde{x})|}{|x-\tilde{x}|^{d-1}}d\tilde{x}\right\|_{L^q(D^\varepsilon_+)}\\
&\leq C\varepsilon\|f\|_{L^p(D^\varepsilon_+)}.
\end{aligned}
\end{equation*}
In the case of $p>d$ and $q=\infty$, the latter simply results from H\"older's inequality. On the other hand, the second term in the right hand side of \eqref{decompueps-u0} is small because we integrate on a small neighborhood of the boundary. Using the bound
\begin{equation*}
|G^0(x,\tilde{x})|\leq \frac{C\tilde{x}_d}{|x-\tilde{x}|^{d-1}}\leq \frac{C\varepsilon}{|x-\tilde{x}|^{d-1}}
\end{equation*}
for all $x,\ \tilde{x}\in\mathbb R^d_+$ such that $0\leq\tilde{x}_d\leq \varepsilon\psi(\tilde{x}'/\varepsilon)$, this term is bounded by $C\varepsilon\|f\|_{L^p(D^\varepsilon_+)}$ using either estimates on weakly singular integrals, or H\"older's inequality according to $p$ and $q$.
\end{proof}

\begin{rem}[Lipschitz estimate and estimate on Poisson's kernel]
One of the issues here is that $u^0$ is oscillating too much on the boundary $\Delta^\varepsilon$. These oscillations do not obstruct the proof of the $L^\infty$ estimate in Lemma \ref{lemueps-u0dual}. However, they prevent us from proving a Lipschitz bound as in \cite[Lemma 3.5]{kls12}. It seems that something more is necessary to refine the expansion of $u^\varepsilon$ close to the boundary, in other words to improve the way our expansion approximates the oscillations of $u^\varepsilon$. The thing we need is probably more structure on the boundary, allowing to carry out a boundary layer analysis as in \cite{David09} for instance. This would result in replacing the Dirichlet boundary condition for $u^0$ by another boundary condition. Such a study is beyond the scope of this paper, since our main concern is to work with boundaries without any structure.
\end{rem}

\section{Two different scales}
\label{sec2scales}

In this section we address the generalization of uniform boundary estimates to the situation where the coefficients and the boundary graph oscillate at two different scales
\begin{equation}\label{systwoscalealphabeta}
\left\{
\begin{array}{rll}
-\nabla\cdot A(x/\alpha)\nabla u^{\alpha,\beta}&=f+\nabla\cdot F,&x\in D^\beta(0,1),\\
u^{\alpha,\beta}&=0,&x\in\Delta^\beta(0,1),
\end{array}
\right.
\end{equation}
where $\alpha,\ \beta>0$. 

Using the three-step compactness method, one can prove that a boundary H\"older estimate uniform in $\alpha$ and $\beta$ holds.

\begin{prop}\label{propboundaryholderab}
Let $\kappa,\ \kappa'>0$. There exist $C>0$, $\varepsilon_0>0$ such that for all 
\begin{equation*}
0<\mu<\min\left(1-d/(d+\kappa),2-d/(d/2+\kappa')\right), 
\end{equation*}
for all $\psi\in\mathcal C_{M_0}^{1,\omega}$, for all $A\in\mathcal A^{0,\nu_0}$, for all $\alpha,\ \beta>0$, for all $f\in L^{d/2+\kappa'}(D^\beta(0,1))$, for all $F\in L^{d+\kappa}(D^\beta(0,1))$, for all $u^{\alpha,\beta}$ weak solution to \eqref{systwoscalealphabeta} the bounds
\begin{equation*}
\|u^{\alpha,\beta}\|_{L^2(D^\beta(0,1))}\leq 1,\quad\|f\|_{L^{d/2+\kappa'}(D^\beta(0,1))}\leq \varepsilon_0,\quad\|F\|_{L^{d+\kappa}(D^\beta(0,1))}\leq \varepsilon_0
\end{equation*}
imply
\begin{equation*}
[u^{\alpha,\beta}]_{C^{0,\mu}(\overline{D^\beta(0,1/2)})}\leq C.
\end{equation*}
Notice that $C$ and $\varepsilon_0$ depend on $d$, $N$, $M_0$, on the modulus of continuity $\omega$ of $\nabla\psi$, $\lambda$, $\kappa$ and $\kappa'$.
\end{prop}

The proof of this estimate is much simpler than the proof of a Lipschitz estimate uniform in $\alpha$ and $\beta$. The latter is our main focus. 

\begin{theo}\label{theoboundarylipalphabeta}
Let $0<\mu<1$ and $\kappa>0$. There exist $C>0$, $\varepsilon_0>0$, such that for all $\psi\in\mathcal C_{M_0}^{1,\nu_0}$, for all $A\in\mathcal A^{0,\nu_0}$, for all $\alpha,\ \beta>0$, for all $f\in L^{d+\kappa}(D^\beta(0,1))$, for all $F\in C^{0,\mu}(D^\beta(0,1))$, for all $u^{\alpha,\beta}$ weak solution to \eqref{systwoscalealphabeta} the bounds
\begin{equation*}
\|u^{\alpha,\beta}\|_{L^\infty(D^\beta(0,1))}\leq 1,\quad\|f\|_{L^{d+\kappa}(D^\beta(0,1))}\leq \varepsilon_0,\quad\|F\|_{C^{0,\mu}(D^\beta(0,1))}\leq \varepsilon_0
\end{equation*}
imply
\begin{equation*}
\|\nabla u^{\alpha,\beta}\|_{L^\infty(D^\beta(0,1/2))}\leq C.
\end{equation*}
Notice that $C$ and $\varepsilon_0$ depend on $d$, $N$, $M_0$, $\lambda$, $\nu_0$, $\kappa$ and $\mu$. Again, the salient point is the uniformity in $\alpha$ and $\beta$ of the constant $C$.
\end{theo}

The main idea is to compare $\alpha$ to $\beta$. Let $\varepsilon_0>0$ be a fixed threshold (to be given later on by a proof by contradiction). There are three main cases: (i) when both $\alpha$ and $\beta$ are bigger than $\varepsilon_0$, the Lipschitz estimate follows from classical estimates; (ii) when only one of the two parameters is bigger than $\varepsilon_0$, the Lipschitz estimate follows from the combination of classical estimates in the slowly oscillating variable and the homogenization (or flattening) properties in the rapidly oscillating variable (see section \ref{subsec2cases} below); (iii) when both parameters are small with respect to $\varepsilon_0$, the uniform estimate is a consequence of a three-step compactness scheme, similar to the case $\alpha=\varepsilon=1$. The proofs of improvement and iteration lemmas go through without obstruction provided that one has appropriate controls of boundary and Dirichlet correctors. Refined estimates on the correctors taking into account the two scales $\alpha$ and $\beta$ are proved in section \ref{secbdarydiralphabetacor}.

We concentrate on the case $d\geq 3$ and $f=F=0$. As in the case when $\alpha=\beta=\varepsilon$, we need boundary and Dirichlet correctors. Let $\gamma:=\max(\alpha,\beta)$. Let $v^{\alpha,\beta}=v^{\alpha,\beta}(x)\in W^{1,2}(D^\beta(0,2))$, the boundary corrector, be the unique weak solution of
\begin{equation}\label{sysbdarycorrab}
\left\{
\begin{array}{rll}
-\nabla\cdot A(x/\alpha)\nabla v^{\alpha,\beta}&=\nabla\cdot A(x/\alpha)\nabla(\psi(x'/\beta)\Theta(x',x_d/\gamma)),&x\in D^\beta(0,2),\\
v^{\alpha,\beta}&=0,&x\in\partial D^\beta(0,2).
\end{array}
\right. 
\end{equation}
Let also $w^{\alpha,\beta}=w^{\alpha,\beta}(x)\in W^{1,2}(D^\beta(0,2))$, the Dirichlet corrector, be the unique weak solution of
\begin{equation}\label{sysbulkcorrab}
\left\{
\begin{array}{rll}
-\nabla\cdot A(x/\alpha)\nabla w^{\alpha,\beta}&=\nabla\cdot A(x/\alpha)\nabla(\chi^d(x/\alpha)\Theta(x',x_d/\gamma)),&x\in D^\beta(0,2),\\
w^{\alpha,\beta}&=0,&x\in\partial D^\beta(0,2).
\end{array}
\right. 
\end{equation}

\subsection{Two particular cases}
\label{subsec2cases}

The boundary Lipschitz estimates in the particular case $\alpha=1$, $\beta=\varepsilon$ on the one hand, and $\alpha=\varepsilon$, $\beta=1$ on the other hand are particularly important. Indeed, these estimates are the ones we will rely on when blowing-up in $\alpha$ or $\beta$.

Let us first consider the case $\alpha=1$, $\beta=\varepsilon$:
\begin{equation}\label{systwoscale1eps}
\left\{
\begin{array}{rll}
-\nabla\cdot A(x)\nabla u^{\varepsilon}&=0,&x\in D^\varepsilon(0,1),\\
u^{\varepsilon}&=0,&x\in\Delta^\varepsilon(0,1),
\end{array}
\right.
\end{equation}

\begin{prop}\label{theoboundarylip1eps}
Let $0<\mu<1$ and $\kappa>0$. There exist $C>0$, such that for all $\psi\in\mathcal C_{M_0}^{1,\nu_0}$, for all $A$ satisfying \eqref{smoothA2} and \eqref{elliptA}, for all $\varepsilon>0$, for all $u^{\varepsilon}$ weak solution to \eqref{systwoscale1eps} the bound
\begin{equation*}
\|u^{\varepsilon}\|_{L^\infty(D^\varepsilon(0,1))}\leq 1
\end{equation*}
implies
\begin{equation*}
\|\nabla u^{\varepsilon}\|_{L^\infty(D^\varepsilon(0,1/2))}\leq C.
\end{equation*}
Notice that $C$ depends on $d$, $N$, $M_0$, $\lambda$, $\nu_0$, $\kappa$ and $\mu$. 
\end{prop}

The proof of Proposition \ref{theoboundarylip1eps} is very similar to the one of Proposition \ref{proplap}. The coefficients are not highly oscillating, and thus no structure (periodicity) is needed for $A$. The classical regularity for the operator $-\nabla\cdot A(x)\nabla$ is used intensively. Notice that when one blows-up at the microscale, the matrix $A(\varepsilon\cdot)$ is very slowly oscillating and 
\begin{equation*}
\|A(\varepsilon\cdot)\|_{L^\infty(\mathbb R^d)}+[A(\varepsilon\cdot)]_{C^{0,\nu_0}(\mathbb R^d)}\leq M_0.
\end{equation*} 

Let us now consider the case $\alpha=\varepsilon$, $\beta=1$:
\begin{equation}\label{systwoscaleeps1}
\left\{
\begin{array}{rll}
-\nabla\cdot A(x/\varepsilon)\nabla u^{\varepsilon}&=0,&x\in D^1(0,1),\\
u^{\varepsilon}&=0,&x\in\Delta^1(0,1),
\end{array}
\right.
\end{equation}

\begin{prop}\label{theoboundarylipeps1}
Let $0<\mu<1$ and $\kappa>0$. There exist $C>0$, such that for all $\psi\in C^{1,\nu_0}(\mathbb R^{d-1})$,
\begin{equation*}
\|\nabla\psi\|_{L^\infty(\mathbb R^{d-1})}+[\nabla\psi]_{C^{0,\nu_0}(\mathbb R^{d-1})}\leq M_0,
\end{equation*}
for all $A\in\mathcal A^{0,\nu_0}$, for all $\varepsilon>0$, for all $u^{\varepsilon}$ weak solution to \eqref{systwoscaleeps1} the bound
\begin{equation*}
\|u^{\varepsilon}\|_{L^\infty(D^1(0,1))}\leq 1
\end{equation*}
implies
\begin{equation}\label{unifeps1}
\|\nabla u^{\varepsilon}\|_{L^\infty(D^1(0,1/2))}\leq C.
\end{equation}
Notice that $C$ depends on $d$, $N$, $M_0$, $\lambda$, $\nu_0$, $\kappa$ and $\mu$. 
\end{prop}

This theorem is the Lipschitz estimate of \cite[Lemma 20]{alin}. A salient point is the fact that $C$ in \eqref{unifeps1} does not depend on $\|\psi\|_{L^\infty}$ as emphasized in \cite{alin}.

\subsection{Boundary and Dirichlet correctors}
\label{secbdarydiralphabetacor}

The key in order to estimate the boundary and Dirichlet correctors are the following estimates of Green's function $G^{\alpha,\beta}=G^{\alpha,\beta}(x,\tilde{x})$ associated to the operator $-\nabla\cdot A(x/\alpha)\nabla$ and to the domain $D^\beta(0,2)$.

\begin{lem}[estimate of Green's kernel, $d\geq 3$]\label{lemestgreend=3ab}
For all $0<\tau<1$, there exists $C>0$ such that for all $\psi\in\mathcal C_{M_0}^{1,\omega}$, for all $A\in\mathcal A^{0,\nu_0}$, for all $0<\alpha<1$, $0<\beta<1$:
\begin{enumerate}[label=(\arabic*)]
\item for all $x,\ \tilde{x}\in D^\beta(0,7/4)$,
\begin{align}
|\widetilde{G}^{\alpha,\beta}(x,\tilde{x})|\leq \frac{C\delta(x)^\tau}{|x-\tilde{x}|^{d-2+\tau}},\label{est3gepsaab}\\
|\widetilde{G}^{\alpha,\beta}(x,\tilde{x})|\leq \frac{C\delta(x)^\tau\delta(\tilde{x})^\tau}{|x-\tilde{x}|^{d-2+2\tau}},\label{est3gepsbab}
\end{align}
\item for all $x,\ \tilde{x}\in D^\beta(0,3/2)$,
\begin{align}
|\nabla_2\widetilde{G}^{\alpha,\beta}(x,\tilde{x})|\leq \frac{C\delta(x)^\tau}{|x-\tilde{x}|^{d-1+\tau}},\quad\mbox{for}\quad |x-\tilde{x}|\leq\alpha,\label{est3nablagepsaabalpha}\\
|\nabla_2\widetilde{G}^{\alpha,\beta}(x,\tilde{x})|\leq \frac{C\delta(x)^\tau\delta(\tilde{x})^\tau}{\alpha|x-\tilde{x}|^{d-2+2\tau}},\quad\mbox{for}\quad |x-\tilde{x}|>\alpha,\label{est3nablagepsbabalpha}
\end{align}
\item for all $x,\ \tilde{x}\in D^\beta(0,3/2)$,
\begin{align}
|\nabla_2\widetilde{G}^{\alpha,\beta}(x,\tilde{x})|\leq \frac{C\delta(x)^\tau}{|x-\tilde{x}|^{d-1+\tau}},\quad\mbox{for}\quad |x-\tilde{x}|\leq\beta,\label{est3nablagepsaabbeta}\\
|\nabla_2\widetilde{G}^{\alpha,\beta}(x,\tilde{x})|\leq \frac{C\delta(x)^\tau\delta(\tilde{x})^\tau}{\beta|x-\tilde{x}|^{d-2+2\tau}},\quad\mbox{for}\quad |x-\tilde{x}|>\beta,\label{est3nablagepsbabbeta}
\end{align}
\end{enumerate}
\end{lem}

\begin{proof}[Proof of Lemma \ref{lemestgreend=3ab}]
The proof of these estimates follows the proof of Lemma \ref{lemestgreend=3}. Remember that estimates \eqref{est3gepsaab} and \eqref{est3gepsbab} are a consequence of the uniform boundary H\"older estimates. There are some subtleties in the gradient estimates, which we now underline. Instead of relying on classical estimates as in the proof of Lemma \ref{lemestgreend=3}, we need to resort to estimates uniform in $\alpha$ or $\beta$. The estimates \eqref{est3nablagepsaabalpha} and \eqref{est3nablagepsbabalpha} rely on the boundary Lipschitz estimate of Proposition \ref{theoboundarylip1eps} uniform in $\beta$ applied at small scale $O(\alpha)$. The estimates \eqref{est3nablagepsaabbeta} and \eqref{est3nablagepsbabbeta} rely on the Lipschitz estimates of \cite{alin} (see Proposition \ref{theoboundarylipeps1}) uniform in $\alpha$ applied at small scale $O(\beta)$. The uniformity of all the constants $C$ in $\alpha$ and $\beta$ is crucial. Let us show \eqref{est3nablagepsaabbeta} and \eqref{est3nablagepsbabbeta} in full details.

Let $x,\ \tilde{x}\in D^\beta(0,3/2)$ and $r:=|x-\tilde{x}|$. Notice that 
\begin{equation*}
\widetilde{G}^{\alpha,\beta}(x,\cdot)^T=\widetilde{G}^{*,\alpha,\beta}(\cdot,x)\quad\mbox{so that}\quad\nabla_2\widetilde{G}^{\alpha,\beta}(x,\cdot)^T=\nabla_1 \widetilde{G}^{*,\alpha,\beta}(\cdot,x),
\end{equation*}
and
\begin{equation*}
\left\{
\begin{array}{rll}
-\nabla\cdot A^*(\hat{x}/\alpha)\nabla \widetilde{G}^{*,\alpha,\beta}(\hat{x},x)&=0,&\hat{x}\in D^\beta(\tilde{x},r/2),\\
\widetilde{G}^{*,\alpha,\beta}(\hat{x},x)&=0,&\hat{x}\in \Delta^\beta(\tilde{x},r/2).
\end{array}
\right.
\end{equation*}
If $r\leq\beta$, consider for $z\in D^{\beta'}_{\psi'}(\tilde{x},1)$
\begin{equation*}
u^{\alpha',\beta'}(z)=\widetilde{G}^{*,\alpha,\beta}(r/2(z-\tilde{x})+\tilde{x},x),\ \alpha':=2\alpha/r,\ \beta':=2\beta/r,\ \psi':=\psi(\cdot+\tilde{x}'(1/\beta-1/\beta'))
\end{equation*}
which solves
\begin{equation*}
\left\{
\begin{array}{rll}
-\nabla\cdot A^*(z/\alpha'+\tilde{x}(1/\alpha-1/\alpha'))\nabla u^{\alpha',\beta'}&=0,&z\in D^{\beta'}_{\psi'}(\tilde{x}',1),\\
u^{\alpha',\beta'}&=0,&z\in \Delta^{\beta'}_{\psi'}(\tilde{x}',1).
\end{array}
\right.
\end{equation*}
Now, $A^*(\cdot+\tilde{x}(1/\alpha-1/\alpha'))\in\mathcal A^{0,\nu_0}$, and the estimate of Proposition \ref{theoboundarylipeps1} uniform in $\alpha'$ yields
\begin{equation}\label{classlipestueps'}
\|\nabla u^{\alpha',\beta'}\|_{L^\infty(D^{\beta'}_{\psi'}(\tilde{x}',1/2))}\leq C'\|u^{\alpha',\beta'}\|_{L^\infty(D^{\beta'}_{\psi'}(\tilde{x}',1))},
\end{equation}
where $C'$ in the previous inequality depends a priori on $\beta'$ but only through $\|\nabla(\beta'\psi'(\cdot/\beta'))\|_{C^{0,\nu_0}}$: since $\beta'\geq 2$,
\begin{equation*}
\|\nabla(\beta'\psi'(\cdot/\beta'))\|_{C^{0,\mu}}=O(1).
\end{equation*}
An important point is that $C'$ does not depend on the $L^\infty$ norm of the boundary graph. Thus $C'$ can be taken uniform in $\alpha'$ and $\beta'$. Rescaling and applying \eqref{est3gepsaab} finally gives
\begin{equation*}
\|\nabla_2\widetilde{G}^{\alpha,\beta}(x,\cdot)\|_{L^\infty(D^\beta(\tilde{x},r/4))}\leq \frac{C}{r}\|\widetilde{G}^{\alpha,\beta}(x,\cdot)\|_{L^\infty(D^\beta(\tilde{x},r/2))}\leq \frac{C}{r}\sup_{\hat{x}\in D^\beta(\tilde{x},r/2)}\frac{\delta(x)^\tau}{|x-\hat{x}|^{d-2+\tau}}\leq \frac{C\delta(x)^\tau}{r^{d-1+\tau}},
\end{equation*}
since $|x-\hat{x}|\geq |x-\tilde{x}|-|\tilde{x}-\hat{x}|\geq r/2$ for all $\hat{x}\in D^\beta(\tilde{x},r/2)$. Thus
\begin{equation*}
|\nabla_2\widetilde{G}^{\alpha,\beta}(x,\tilde{x})|\leq \frac{C\delta(x)^\tau}{|x-\tilde{x}|^{d-1+\tau}}.
\end{equation*}
If $r>\beta$, $D^\beta(\tilde{x},\beta/2)\subset D^\beta(\tilde{x},r/2)$, so that we may directly apply the Lipschitz estimate of Proposition \ref{theoboundarylipeps1} at small scale $O(\beta)$ in combination with \eqref{est3gepsbab}
\begin{multline*}
\|\nabla_2\widetilde{G}^{\alpha,\beta}(x,\cdot)\|_{L^\infty(D^\beta(\tilde{x},\beta/4))}\leq \frac{C}{\beta}\|\widetilde{G}^{\alpha,\beta}(x,\cdot)\|_{L^\infty(D^\beta(\tilde{x},\beta/2))}\\
\leq \frac{C}{\beta}\sup_{\hat{x}\in D^\beta(\tilde{x},\beta/2)}\frac{\delta(x)^\tau\delta(\hat{x})^\tau}{|x-\hat{x}|^{d-2+2\tau}}\leq \frac{C\delta(x)^\tau\delta(\tilde{x})^\tau}{\beta|x-\tilde{x}|^{d-2+2\tau}},
\end{multline*}
which implies \eqref{est3nablagepsbabbeta}.
\end{proof}

\begin{lem}[boundary corrector]\label{bdarycorlemab}
For all $1/2<\tau<1$, there exists $C>0$ such that for all $\psi\in\mathcal C_{M_0}^{1,\omega}$, for all $A\in\mathcal A^{0,\nu_0}$, for all $0<\alpha<1$, $0<\beta<1$, the unique weak solution $v^{\alpha,\beta}\in W^{1,2}(D^\beta(0,2))$ of \eqref{sysbdarycorrab} satisfies the following estimate: for all $x\in D^\beta(0,3/2)$,
\begin{equation}\label{estbdarycorlemab}
|v^{\alpha,\beta}(x)|\leq C_0\left\{\begin{array}{ll}
\delta(x)^\tau/\beta^\tau,&\mbox{if}\ \beta>\alpha,\\
\left(\frac{1}{\alpha}+\frac{1}{\beta}\right)\alpha^{1-\tau}\delta(x)^\tau,&\mbox{if}\ \beta\leq\alpha,
\end{array}\right.
\end{equation}
where $\delta(x):=x_d-\beta\psi(x'/\beta)$.
\end{lem}

\begin{lem}[Dirichlet corrector]\label{bulkcorlemab}
For all $1/2<\tau<1$, there exists $C_0'>0$ such that for all $\psi\in\mathcal C_{M_0}^{1,\omega}$, for all $A\in\mathcal A^{0,\nu_0}$, for all $0<\alpha<1$, $0<\beta<1$, the unique weak solution $w^{\alpha,\beta}\in W^{1,2}(D^\beta(0,2))$ of \eqref{sysbulkcorrab} satisfies the following estimate: for all $x\in D^\beta(0,3/2)$,
\begin{equation}\label{estdircorlemab}
|w^{\alpha,\beta}(x)|\leq C_0'\left\{\begin{array}{ll}
\left(\frac{1}{\alpha}+\frac{1}{\beta}\right)\beta^{1-\tau}\delta(x)^\tau,&\mbox{if}\ \beta>\alpha,\\
\delta(x)^\tau/\alpha^\tau,&\mbox{if}\ \beta\leq\alpha,
\end{array}\right.
\end{equation}
where $\delta(x):=x_d-\beta\psi(x'/\beta)$.
\end{lem}

\begin{proof}
The proof of Lemma \ref{bulkcorlemab} follows the lines of the proof of Lemma \ref{bdarycorlem} for the boundary corrector, using the new estimates \eqref{est3nablagepsaabalpha} and \eqref{est3nablagepsbabalpha} when $\beta\leq\alpha$, and the estimates \eqref{est3nablagepsaabbeta} and \eqref{est3nablagepsbabbeta} when $\beta>\alpha$. 

Let us write down the calculations when $\beta>\alpha$. We rely on estimate \eqref{est3nablagepsbabbeta}. For all $x\in D^\beta(0,3/2)$,
\begin{align*}
w^{\alpha,\beta}(x)&=-\int_{D^\beta(0,2)}\nabla_2 \widetilde{G}^{\alpha,\beta}(x,\tilde{x})A(\tilde{x}/\alpha)\nabla(\chi^d(\tilde{x}'/\alpha)\Theta(\tilde{x}',\tilde{x}_d/\beta))d\tilde{x}\\
&=-\int_{D^\beta(0,2)\cap\{|x-\tilde{x}|\leq\beta\}}\nabla_2 \widetilde{G}^{\alpha,\beta}(x,\tilde{x})A(\tilde{x}/\alpha)\nabla(\chi^d(\tilde{x}'/\alpha)\Theta(\tilde{x}',\tilde{x}_d/\beta))d\tilde{x}\\
&\qquad-\int_{D^\beta(0,2)\{|x-\tilde{x}|>\beta\}}\nabla_2 \widetilde{G}^{\alpha,\beta}(x,\tilde{x})A(\tilde{x}/\alpha)\nabla(\chi^d(\tilde{x}'/\alpha)\Theta(\tilde{x}',\tilde{x}_d/\beta))d\tilde{x}\\
&=I_1+I_2.
\end{align*}
We get on the one hand by \eqref{est3nablagepsaabbeta}
\begin{align*}
|I_1|&\leq C\left(\frac{1}{\alpha}+\frac{1}{\beta}\right)\int_{\{|x-\tilde{x}|\leq\beta\}}\frac{\delta(x)^\tau}{|x-\tilde{x}|^{d-1+\tau}}d\tilde{x}\\
&\leq C\delta(x)^\tau\left(\frac{1}{\alpha}+\frac{1}{\beta}\right)\int_{|\tilde{y}|\leq 1}\frac{1}{|\tilde{y}|^{d-1+\tau}}d\tilde{y}\leq C\left(\frac{1}{\alpha}+\frac{1}{\beta}\right)\beta^{1-\tau}\delta(x)^\tau,
\end{align*}
and on the other hand by \eqref{est3nablagepsbabbeta}
\begin{align*}
|I_2|&\leq C\beta^{-1}\left(\frac{1}{\alpha}+\frac{1}{\beta}\right)\int_{D^\beta(0,2)\cap[-3/2,3/2]^{d-1}\times[-3M_0\beta/2,3M_0\beta/2]\cap\{|x-\tilde{x}|>\beta\}}\frac{\delta(x)^\tau\delta(\tilde{x})^\tau}{|x-\tilde{x}|^{d-2+2\tau}}d\tilde{x}\\
&\leq C\delta(x)^\tau\beta^{\tau-1}\left(\frac{1}{\alpha}+\frac{1}{\beta}\right)\int_{D^\beta(0,2)\cap[-3/2,3/2]^{d-1}\times[-3M_0\beta/2,3M_0\beta/2]\cap\{|x-\tilde{x}|>\beta\}}\frac{1}{|x-\tilde{x}|^{d-2+2\tau}}d\tilde{x}\\
&\leq C\delta(x)^\tau\beta^{\tau-1}\left(\frac{1}{\alpha}+\frac{1}{\beta}\right)\Biggl\{\int_{|x'-\tilde{x}'|>\beta}\frac{1}{|x'-\tilde{x}'|^{d-2+2\tau}}\int_{-3M_0\beta/2}^{3M_0\beta/2}\frac{1}{\left\{1+\frac{(x_d-\tilde{x}_d)^2}{|x'-\tilde{x}'|^{2}}\right\}^{(d-2+2\tau)/2}}d\tilde{x}_dd\tilde{x}'\Biggr.\\
&\qquad\Biggl.+\int_{\{|x-\tilde{x}|>\beta\}\cap\{|x'-\tilde{x}'|<\beta\}\times[-3M_0\beta/2,3M_0\beta/2]}\frac{1}{|x-\tilde{x}|^{d-2+2\tau}}d\tilde{x}\Biggr\}\\
&\leq C\delta(x)^\tau\beta^{\tau-1}\left(\frac{1}{\alpha}+\frac{1}{\beta}\right)\Biggl\{\beta^{2-2\tau}\int_{|\tilde{y}'|>1}\frac{1}{|\tilde{y}'|^{d-2+2\tau}}d\tilde{y}'+\beta^{2-2\tau}\Biggr\}\\
&\leq C\left(\frac{1}{\alpha}+\frac{1}{\beta}\right)\beta^{1-\tau}\delta(x)^\tau,
\end{align*}
for $1/2<\tau<1$.
\end{proof}

\subsection{Proof of the Lipschitz estimate}

We now tackle the proof of Theorem \ref{theoboundarylipalphabeta} relying on the three-step compactness method. The two usual auxiliary lemmas (improvement and iteration) read as follows.
\begin{lem}[improvement lemma]\label{lem1lipab}
There exist $0<\varepsilon_0<1,\ 0<\theta<1/8$, such that for all $\psi\in\mathcal C_{M_0}^{1,\nu_0}$, for all $A\in\mathcal A^{0,\nu_0}$, for all $0<\alpha,\beta<\varepsilon_0$, for all $u^\varepsilon$ weak solution to \eqref{systwoscalealphabeta}, 
if
\begin{equation*}
\|u^{\alpha,\beta}\|_{L^\infty(D^\beta(0,1/2))}\leq 1,
\end{equation*}
then 
\begin{multline}\label{estlem1ablip}
\bigl\|u^{\alpha,\beta}(x)-\left(\overline{\partial_{x_d} u^{\alpha,\beta}}\right)_{0,\theta}\left\{x_d-\alpha(1-\Theta(x',x_d/\gamma))\chi^d(x/\alpha)-\alpha w^{\alpha,\beta}(x)\right.\bigr.\\
\bigl.\left.-\beta\psi(x'/\beta)\Theta(x',x_d/\gamma)-\beta v^{\alpha,\beta}(x)\right\}\bigr\|_{L^\infty(D^\beta(0,\theta))}\leq\theta^{1+\mu},
\end{multline}
where $\gamma:=\max(\alpha,\beta)$.
\end{lem}

\begin{lem}[iteration lemma]\label{lem2lipab}
Let $0<\varepsilon_0<1$ and $\theta>0$ as given by Lemma \ref{lem1lipab}. There exists $C_1>0$, for all $k\in\mathbb N$, $k\geq 1$, for all $0<\alpha,\beta<\theta^{k-1}\varepsilon_0$, for all $\psi\in\mathcal C_{M_0}^{1,\nu_0}$, for all $A\in\mathcal A^{0,\nu_0}$, for all $u^{\alpha,\beta}$ weak solution to \eqref{systwoscalealphabeta}, if
\begin{equation*}
\|u^{\alpha,\beta}\|_{L^\infty(D^\beta(0,1/2))}\leq 1,
\end{equation*}
then there exist $a^{\alpha,\beta}_k\in \mathbb R$, $V^{\alpha,\beta}_k=V^{\alpha,\beta}_k(x)$ and $W^{\alpha,\beta}_k=W^{\alpha,\beta}_k(x)$ such that
\begin{equation*}
\begin{aligned}
|a^{\alpha,\beta}_k|&\leq (C_0/\theta)[1+\theta^\mu+\ldots\ \theta^{(k-1)\mu}],\\
|V^{\alpha,\beta}_k(x)|&\leq (C_0C_1/\theta)[1+\theta^\mu+\ldots\ \theta^{(k-1)\mu}]\left\{\begin{array}{ll}
\delta(x)^\tau/\beta^\tau,&\mbox{if}\ \beta>\alpha,\\
\left(\frac{1}{\alpha}+\frac{1}{\beta}\right)\alpha^{1-\tau}\delta(x)^\tau,&\mbox{if}\ \beta\leq\alpha,
\end{array}\right.\\
|W^{\alpha,\beta}_k(x)|&\leq (C_0'C_1/\theta)[1+\theta^\mu+\ldots\ \theta^{(k-1)\mu}]\delta(x)^\tau\left\{\begin{array}{ll}
\left(\frac{1}{\alpha}+\frac{1}{\beta}\right)\beta^{1-\tau}\delta(x)^\tau,&\mbox{if}\ \beta>\alpha,\\
\delta(x)^\tau/\alpha^\tau,&\mbox{if}\ \beta\leq\alpha,
\end{array}\right.
\end{aligned}
\end{equation*}
where $C_0$ (resp. $C_0'$) is the constant appearing in \eqref{estbdarycorlemab} (resp. \eqref{estdircorlemab}) and
\begin{multline*}
\bigl\|u^{\alpha,\beta}(x)-a^{\alpha,\beta}_k\left\{x_d-\alpha(1-\Theta(x',x_d/\gamma))\chi^d(x/\alpha)-\beta\psi(x'/\beta)\Theta(x',x_d/\gamma)\right\}\bigr.\\
\bigl.-\alpha W^{\alpha,\beta}_k(x)-\beta V^{\alpha,\beta}_k(x)\bigr\|_{L^\infty(D^\beta(0,\theta^k))}\leq\theta^{k(1+\mu)},
\end{multline*}
with $\gamma:=\max(\alpha,\beta)$.
\end{lem}

We emphasize a few points of the proof of Lemma \ref{lem1lipab}. The general scheme developed in the proof of Lemma \ref{lem1} works all the same here. As usual we assume that there exist sequences $\alpha_k,\ \beta_k\rightarrow 0$, $\psi_k\in \mathcal C_{M_0}^{1,\nu_0}$, $A_k\in\mathcal A^{0,\nu_0}$ such that 
\begin{equation*}
\|u^{\alpha_k,\beta_k}\|_{L^\infty(D^{\beta_k}(0,1/2))}\leq 1,
\end{equation*}
and 
\begin{multline}\label{contraestlem1ablip}
\bigl\|u^{\alpha_k,\beta_k}(x)-\left(\overline{\partial_{x_d} u^{\alpha_k,\beta_k}}\right)_{0,\theta}\left\{x_d-\alpha_k(1-\Theta(x',x_d/\gamma_k))\chi^d_k(x/\alpha_k)-\alpha_k w^{\alpha_k,\beta_k}(x)\right.\bigr.\\
\bigl.\left.-\beta_k\psi_k(x'/\beta_k)\Theta(x',x_d/\gamma_k)-\beta_k v^{\alpha_k,\beta_k}(x)\right\}\bigr\|_{L^\infty(D^{\beta_k}(0,\theta))}>\theta^{1+\mu},
\end{multline}
with $\gamma_k:=\max(\alpha_k,\beta_k)$. The control of $\|u^{\alpha_k,\beta_k}\|_{L^\infty(D^{\beta_k}(0,1/2))}$ and the H\"older estimate imply that $u^{\alpha_k,\beta_k}$ is bounded in $C^{0,\mu}$ uniformly in $\alpha_k$ and $\beta_k$. We can therefore extract subsequences converging strongly in $C^0(\overline{D^{\beta_k}(0,1/2)})$. Furthermore, Cacciopoli's inequality yields weak compactness in $H^1$. These convergences make it possible to pass to the limit in \eqref{contraestlem1ablip}, which contradicts the estimate for constant homogenized coefficients in the flat domain.

The key for the proof of the iteration Lemma \ref{lem2lipab} is the uniformity of \eqref{estlem1ablip} in $\alpha$ and $\beta$. Notice also that $\beta/\theta^k>\alpha/\theta^k$ for all $k\geq 1$ whenever $\beta>\alpha$.

\begin{proof}[Proof of Theorem \ref{theoboundarylipalphabeta}]
If both $\alpha$ and $\beta>\varepsilon_0$, then the theorem follows from classical estimates. Assume now that (at least) one of the two parameters is less than $\varepsilon_0$. If $\beta\geq\varepsilon_0$ (resp. $\alpha\geq\varepsilon_0$), then Theorem \ref{theoboundarylipalphabeta} follows from the Proposition \ref{theoboundarylipeps1} (resp. Proposition \ref{theoboundarylip1eps}). Assume now that $0<\alpha,\ \beta<\varepsilon_0$ and let $\psi\in\mathcal C^{1,\nu_0}_{M_0}$ and $A\in\mathcal A^{0,\nu_0}$ be fixed for the rest of the proof. Let also $x_0:=(0,x_{0,d})\in D^\beta(0,1/2)$ be fixed.

Assume that $\alpha<\beta$. In that case, we will blow-up at the biggest of the two scales, i.e. $\beta$. There exist a $k\geq 1$ such that $\theta^k\leq\beta/\varepsilon_0<\theta^{k-1}$. We have
\begin{multline*}
\|u^{\alpha,\beta}\|_{L^\infty(D^\beta(0,\beta/\varepsilon_0))}\leq \theta^{k(1+\mu)}+|a^{\alpha,\beta}_k|\left\{\|x_d-\beta\psi(x'/\beta)\Theta(x',x_d/\beta)\|_{L^\infty(D^\beta(0,\theta^k))}\right.\\
\left.+\alpha\|(1-\Theta(x',x_d/\beta))\chi^d(x/\alpha)\|_{L^\infty(D^\beta(0,\theta^k))}+\alpha\|W^{\alpha,\beta}_k(x)\|_{L^\infty(D^\beta(0,\theta^k))}+\beta\|V^{\alpha,\beta}_k(x)\|_{L^\infty(D^\beta(0,\theta^k))}\right\}\\
\leq C\theta^k\leq C\beta.
\end{multline*}
Assume that $\alpha\geq\beta$. In that case, we will blow-up at the biggest of the two scales, i.e. $\alpha$. There exist a $k\geq 1$ such that $\theta^k\leq\alpha/\varepsilon_0<\theta^{k-1}$. We have
\begin{multline*}
\|u^{\alpha,\beta}\|_{L^\infty(D^\beta(0,\alpha/\varepsilon_0))}\leq \theta^{k(1+\mu)}+|a^{\alpha,\beta}_k|\left\{\|x_d-\beta\psi(x'/\beta)\Theta(x',x_d/\alpha)\|_{L^\infty(D^\beta(0,\theta^k))}\right.\\
\left.+\alpha\|(1-\Theta(x',x_d/\alpha))\chi^d(x/\alpha)\|_{L^\infty(D^\beta(0,\theta^k))}+\alpha\|W^{\alpha,\beta}_k(x)\|_{L^\infty(D^\beta(0,\theta^k))}+\beta\|V^{\alpha,\beta}_k(x)\|_{L^\infty(D^\beta(0,\theta^k))}\right\}\\
\leq C\theta^k\leq C\alpha.
\end{multline*}
The case of arbitrary points $x_0\in D^\beta(0,1/2)$ is treated in a standard way.
\end{proof}

\section{Further generalizations}
\label{secgen}

\subsection{A generalization to boundaries with macroscopic behavior}

For this part only, we address a generalization of the Lipschitz estimate to domains with oscillating boundaries with a macroscopic dependence
\begin{equation*}
\psi=\varepsilon\psi(x',x'/\varepsilon).
\end{equation*}
The Lipschitz estimate should hold provided that we have enough regularity ensuring classical estimates. There are only small modifications to be done in the iteration procedure (Lemma \ref{lem2}). Let $\psi=\psi(x',y')\in C^{1,\nu_0}(\mathbb R^{d-1}\times\mathbb R^{d-1})$ be a function such that $0\leq\psi\leq M_0$ and
\begin{multline}\label{boundspsigenmacro}
\sup_{y'\in\mathbb R^{d-1}}\|\nabla_1\psi(\cdot,y')\|_{L^\infty(\mathbb R^{d-1})}+\sup_{x'\in\mathbb R^{d-1}}\|\nabla_2\psi(x',\cdot)\|_{L^\infty(\mathbb R^{d-1})}\\
+\sup_{y'\in\mathbb R^{d-1}}[\nabla_1\psi(\cdot,y')]_{C^{0,\nu_0}(\mathbb R^{d-1})}+\sup_{x'\in\mathbb R^{d-1}}\|\nabla_2\psi(x',\cdot)\|_{C^{0,\nu_0}(\mathbb R^{d-1})}\leq M_0.
\end{multline}
The improvement lemma and the blow-up analysis go through in the exact same way as for $\psi=\psi(y')\in\mathcal C^{1,\nu_0}_{M_0}$. The estimate corresponding to the one of Lemma \ref{lem1} is
\begin{equation}\label{estlem1macro}
\bigl\|u^\varepsilon(x)-\left(\overline{\partial_{x_d} u^\varepsilon}\right)_{0,\theta}\left\{x_d-\varepsilon\psi(x',x'/\varepsilon)\Theta(x',x_d/\varepsilon)-\varepsilon v^\varepsilon(x)\right\}\bigr\|_{L^\infty(D^\varepsilon(0,\theta))}\leq\theta^{1+\mu}.
\end{equation}
For the sake of completeness, let us do the first step of the iteration. We consider
\begin{equation*}
U^\varepsilon(x):=\frac{1}{\theta^{1+\mu}}\left[u^\varepsilon(\theta x)-\left(\overline{\partial_{x_d} u^\varepsilon}\right)_{0,\theta}\left\{\theta x_d-\varepsilon\psi(\theta x',\theta x'/\varepsilon)\Theta(\theta x',\theta x_d/\varepsilon)-\varepsilon v^\varepsilon(\theta x)\right\}\right],
\end{equation*}
which solves
\begin{equation*}
\left\{\begin{array}{rll}
-\Delta U^\varepsilon&=0
,&x\in D^{\varepsilon/\theta}_{\psi^1}(0,1/2),\\
U^\varepsilon&=0,&x\in\Delta^{\varepsilon/\theta}_{\psi^1}(0,1/2),
\end{array}
\right. 
\end{equation*}
where $\psi^1:=\psi(\theta\cdot,\cdot)$. Notice that $\psi^1$ satisfies the bounds \eqref{boundspsigenmacro} and that by estimate \eqref{estlem1macro} $\|U^\varepsilon\|_{L^\infty(D^\varepsilon_{\psi^1}(0,1/2))}\leq 1$. Applying now the ad hoc improvement lemma, we get for $\varepsilon/\theta<\varepsilon_0$
\begin{equation*}
\left\|U^\varepsilon-(\overline{\partial_{x_d}U^\varepsilon})_{0,\theta}\left\{x_d-\varepsilon/\theta\psi^1(x',\theta x'/\varepsilon)\Theta(x',\theta x_d/\varepsilon)-\varepsilon/\theta v^{\varepsilon/\theta}_{\psi^1}(x)\right\}\right\|_{L^\infty(D^{\varepsilon/\theta}_{\psi^1}(0,\theta))}\leq \theta^{1+\mu},
\end{equation*}
where $v^{\varepsilon/\theta}_{\psi^1}$ is a solution of \eqref{sysbdarycorr} with $\psi^1$ in place of $\psi$, $N=1$ and $A=\Idd_d$. Therefore
\begin{multline*}
\bigl\|u^\varepsilon(\theta x)-\left(\overline{\partial_{x_d} u^\varepsilon}\right)_{0,\theta}\left\{\theta x_d-\varepsilon\psi(\theta x',\theta x'/\varepsilon)\Theta(\theta x',\theta x_d/\varepsilon)-\varepsilon v^\varepsilon(\theta x)\right\}\bigr.\\
\bigl.-\theta^\mu(\overline{\partial_{x_d}U^\varepsilon})_{0,\theta}\left\{\theta x_d-\varepsilon\psi^1(x',\theta x'/\varepsilon)\Theta(x',\theta x_d/\varepsilon)-\varepsilon v^{\varepsilon/\theta}_{\psi^1}(x)\right\}\bigr\|_{L^\infty(D^{\varepsilon/\theta}_{\psi^1}(0,\theta))}\leq \theta^{2(1+\mu)},
\end{multline*}
which using $\psi^1(x',\theta x'/\varepsilon)=\psi(\theta x',\theta x'/\varepsilon)$ and  $\Theta(x',\theta x_d/\varepsilon)=\Theta(\theta x',\theta x_d/\varepsilon)$ for $|x'|<1$ boils down to
\begin{multline*}
\bigl\|u^\varepsilon(x)-\left[\left(\overline{\partial_{x_d} u^\varepsilon}\right)_{0,\theta}+\theta^\mu(\overline{\partial_{x_d}U^\varepsilon})_{0,\theta}\right]\left\{x_d-\varepsilon\psi(x',x'/\varepsilon)\Theta(x',x_d/\varepsilon)\right\}\bigr.\\
\bigl.-\varepsilon \left(\overline{\partial_{x_d} u^\varepsilon}\right)_{0,\theta}v^\varepsilon(x)-\varepsilon\theta^\mu(\overline{\partial_{x_d}U^\varepsilon})_{0,\theta} v^{\varepsilon/\theta}_{\psi^1}(x)\bigr\|_{L^\infty(D^{\varepsilon}_{\psi}(0,\theta^2))}\leq \theta^{2(1+\mu)}.
\end{multline*}

\subsection{Inclined bumpy half-spaces}

Consider now an inclined half-space domain $\Omega^\varepsilon\subset\mathbb R^d$, such that there is a $\psi_1\in\mathcal C^{1,\nu_0}_{M_0}$ and an orthogonal matrix $O\in M_d(\mathbb R)$, such that
\begin{equation*}
\begin{aligned}
\Omega^\varepsilon_1&:=O\Omega^\varepsilon=\{x\in\mathbb R^d:\ \exists z\in\Omega^\varepsilon,\ x=Oz\}\\
&=\{(x',x_d)\in\mathbb R^d:\ x_d>\varepsilon\psi_1(x'/\varepsilon)\}.
\end{aligned}
\end{equation*}
This means that $\Omega^\varepsilon$ is just the rotation of an oscillating half-space domain $\Omega^\varepsilon_1$ of the type studied all along this paper.

We intend to emphasize a recent result of Schmutz \cite{Schmutz08}, which has been communicated to us by Zhongwei Shen \cite{Shenperiodichomodraft}. This result may be of interest to the PDE community.

\begin{theo}[\cite{Schmutz08}, Theorem 3.1]\label{rstSchmutz}
Let $O\in M_d(\mathbb R)$ an orthogonal matrix. For any $\delta>0$, there exists an orthogonal matrix $T\in M_d(\mathbb Q)$, i.e. with rational entries, such that:
\begin{enumerate}
\item $|O-T|=\sup_{1\leq\alpha,\beta\leq d}|O^{\alpha\beta}-T^{\alpha\beta}|<\delta$;
\item each entry of $T$ has a denominator bounded from below by a constant depending only on $d$ and $\delta$.
\end{enumerate}
\end{theo}

We would like to apply this theorem to get estimates on the system
\begin{equation}\label{sysueps(z)}
\left\{\begin{array}{rll}
-\nabla\cdot A(z/\varepsilon)\nabla u^\varepsilon&=0,&z\in\Omega^\varepsilon\cap B(0,1),\\
u^\varepsilon&=0,&z\in\partial\Omega^\varepsilon\cap B(0,1),
\end{array}
\right.
\end{equation}
using the theory we have developed in this paper.

Let $\delta>0$ be small enough, such that for any orthogonal matrix $T$ satisfying $|O-T|<\delta$, we have the existence of $\psi_2\in\mathcal C^{1,\nu_0}_{M_0}$ and
\begin{equation*}
\begin{aligned}
\Omega^\varepsilon_2&:=T\Omega^\varepsilon=\{\hat{x}\in\mathbb R^d:\ \exists z\in\Omega^\varepsilon,\ \hat{x}=Tz\}\\
&=\{(\hat{x}',\hat{x}_d)\in\mathbb R^d:\ \hat{x}_d>\varepsilon\psi_2(\hat{x}'/\varepsilon)\}.
\end{aligned}
\end{equation*}
In other words, since we control $\|\nabla\psi\|_{L^\infty(\mathbb R^{d-1})}$ uniformly in the class $\mathcal C^{1,\nu_0}_{M_0}$, a slight rotation of $\Omega^\varepsilon_1$ will still be a half-space above a graph. We fix $\delta$.

Take now a matrix $T$ approximating $O$ such as given by Theorem \ref{rstSchmutz}. In particular, both $T$ and $T^{-1}$ are orthogonal with rational entries, $|O-T|<\delta$, and there exists a large integer $N$, depending only on $d$ and $\delta$ such that $NT^{-1}$ has integer entries. Now, let 
\begin{equation*}
\Omega^\varepsilon_3:=T\Omega^\varepsilon=\{\hat{x}\in\mathbb R^d:\ \exists z\in\Omega^\varepsilon,\ \hat{x}=Tz\}.
\end{equation*}
There exists $\psi_3\in\mathcal C^{1,\nu_0}_{M_0}$ such that
\begin{equation*}
\Omega^\varepsilon_3=\{(\hat{x}',\hat{x}_d)\in\mathbb R^d:\ \hat{x}_d>\varepsilon\psi_3(\hat{x}'/\varepsilon)\}.
\end{equation*}
We proceed to the change of variables $\hat{x}:=N^{-1}Tz$ and let $w^\varepsilon(\hat{x}):=u^\varepsilon(z)$, where $u^\varepsilon$ solves \eqref{sysueps(z)}. After this change of variables, $w^\varepsilon=w^\varepsilon(\hat{x})$ solves
\begin{equation}\label{sysweps}
\left\{\begin{array}{rll}
-\nabla\cdot B(\hat{x}/\varepsilon)\nabla w^\varepsilon&=0,&\hat{x}\in\Omega^\varepsilon_3\cap B(0,1),\\
w^\varepsilon&=0,&\hat{x}\in\partial\Omega^\varepsilon_3\cap B(0,1),
\end{array}
\right.
\end{equation}
where for all $1\leq\alpha,\ \beta\leq d$ and $1\leq i,\ j\leq N$, for all $\hat{x}\in\Omega^\varepsilon_3$,
\begin{equation*}
B^{\alpha\beta}_{ij}(\hat{x}/\varepsilon)=N^{-2}T_{\alpha\gamma}T_{\beta\delta}A^{\gamma\delta}_{ij}(NT^{-1}\hat{x}).
\end{equation*}
Since $NT^{-1}$ has integer entries, $A(NT^{-1}\cdot)$ is periodic (with a very large period), so $B$ is also periodic. Therefore, we can apply all the estimates we have proved for \eqref{sysoscueps} to the system \eqref{sysweps}.

\bibliographystyle{plain}
\bibliography{BLtailosc.bib} 

\end{document}